\documentclass{article}

\usepackage{amsthm}  

\usepackage{amsmath}
\usepackage{latexsym}
\usepackage{amssymb}
\usepackage{verbatim}
\usepackage{setspace}
\usepackage{qtree}
\usepackage{pbox}
\usepackage{enumerate}

\usepackage{graphicx}
\usepackage{tikz}
\usetikzlibrary{tikzmark}

\usepackage{tikz}
\usetikzlibrary{arrows,decorations.pathmorphing,backgrounds,positioning,fit}

\newtheorem{df}{Definition}[section]  
\newtheorem{teo}[df]{Theorem}

\newtheorem{lem}[df]{Lemma}
\newtheorem{lemma}[df]{Lemma}

\newtheorem{example}[df]{Example}

\newtheorem{notat}[df]{Notation}
\newtheorem{rema}[df]{Remark}

\newcommand{\dep}[2]{=\hspace{-3pt}({#1};{#2})}
\newcommand{\con}[1]{=\hspace{-3pt}({#1})}

\newcommand{\cf}{\hspace{2pt}\Box\hspace{-4pt}\rightarrow}

\newcommand{\iimp}{\hspace{2pt}|\hspace{-4pt}\rightarrow}

\newcommand{\CD}{\mathcal{CD}}
\newcommand{\ICD}{\mathcal{ICD}}
\newcommand{\ICDP}{\mathcal{ICD}_+}
\newcommand{\ARROW}[1]{{#1^{\rightarrow}}}
\newcommand{\ACD}{\ARROW{\mathcal{CD}}}
\newcommand{\AICD}{\ARROW{\mathcal{ICD}}}
\newcommand{\AICDP}{\ARROW{\mathcal{ICD}_+}}

\newcommand{\PCD}{\mathcal{PCD}}
\newcommand{\PC}{\mathcal{PC}}
\newcommand{\PO}{\mathcal{PO}}
\newcommand{\PP}{\mathcal{P}}
\newcommand{\C}{\mathcal{C}}
\newcommand{\CO}{\mathcal{CO}}

\newcommand{\SET}[1]{\mathbf{#1}}

\newcommand{\pindep}{\rotatebox[origin=c]{90}{$\models$}}

\author{Fausto Barbero\thanks{The present work has been developed within the Academy of Finland Project n.286991, ``Dependence and Independence in Logic: Foundations and Philosophical Significance''.
} \qquad Gabriel Sandu\footnotemark[1] \\ 
 University of Helsinki}
\title{Team semantics for interventionist counterfactuals and causal dependence}

\begin{document}

\setcounter{page}{1} 

\maketitle

\begin{abstract}
We introduce a generalization of team semantics which provides a framework for 
manipulationist theories of causation based on structural equation models, such as 
Woodward's and Pearl's; our causal teams incorporate (partial or total) information about 
functional dependencies that are invariant under interventions. We give a unified treatment of  
observational and causal aspects of causal models by isolating two operators on causal 
teams which correspond, respectively, to conditioning and to interventionist counterfactual 
implication.

The evaluation of counterfactuals may involve the production of partially determined teams. We 
suggest a way of dealing with such cases by 1) the introduction of formal entries in causal 
teams, and 2) the introduction of weaker truth values (falsifiability and admissibility), 
for which we suggest some plausible semantical clauses.

We introduce formal languages for both deterministic and probabilistic causal discourse, and study in some detail their inferential aspects. Finally, we apply our framework to the analysis of direct and total causation, and other notions of dependence and invariance.
\end{abstract}

\tableofcontents

\section{Introduction}

Notions of dependence and independence entered the realm of logical investigation in the early days of mathematical logic, essentially with the introduction, by Frege, of nested quantification; this aspect of quantification was made explicit by the notion of Skolem function (\cite{Sko1920}). It is only in the last decades, however, that a systematical analysis of (in)dependence notions within predicative, propositional, modal logical languages has been undertaken. One of the main unifying tools in this enterprise is the so-called \emph{team semantics} (\cite{Hod1997},\cite{Hod1997b},\cite{Vaa2007}), whose key idea is that formulas involving dependencies acquire meaning only when evaluated over \emph{sets} of assignments. Variations of this methodology have allowed a systematical study of logical systems enriched with dependencies that arise from database theory, probabilistic theory and  quantum information theory. In many cases, distinct notions of (in)dependence can coexist in one and the same formal language, and this kind of interplay has been systematically investigated from the point of view of definability and complexity. 
 However, to the best of our knowledge, the framework of team semantics has not yet been used to investigate notions of causal and counterfactual dependence. In the present paper, we provide a generalization of team semantics that is also adequate to capture causal, counterfactual and probabilistic notions of (in)dependence which arise from modern manipulationist theories of causation such as Pearl's (\cite{Pea2000}) and Woodward's (\cite{Woo2003}).
 The generalization is not trivial. It is usually acknowledged in the literature that causal relationships cannot be reduced to mere correlations of data (the latter can be represented by a set of assignments.) Instead, a richer structure, encoding counterfactual assumptions, is needed. 


The plan of the paper is as follows. In section \ref{THCAUS} we give a short motivation for the manipulationist (interventionist) approach to causation.
 In section \ref{TEAMSEC} we present team semantics and show how to adequately enrich it so that it can handle interventionist counterfactuals. 
We will introduce several languages to express various (deterministic) notions of dependence. In section \ref{LOGLAWS} we analyze the logical properties of these languages, up to some basic soundness and completeness proofs; we use some of these properties to compare our counterfactuals with those of Stalnaker (\cite{Sta1968}), Lewis (\cite{Lew1973}) and Galles\&Pearl (\cite{GalPea1998}). Section \ref{NONPARAM} is dedicated to the logical issues that arise from nonparametric models. In section \ref{PROBSEC} we introduce probabilistic causal languages. In section \ref{CAUSMOD} we will discuss various notions of causation (mainly taken from Woodward) and invariance, in the light of the logic developed in earlier sections.

\section{Theories of causation: background} \label{THCAUS}

The reductive approach to causation has been philosophers' favourite
tool. It aims, roughly, at finding necessary and sufficient conditions
for causal relationships like ``$X$ causes $Y$''. Two such conditions
have been prominent in the literature: those formulated in terms of
conditional probabilities, and those based on counterfactuals (counterfactual
dependence). We discuss them shortly in the next two sections. The
main purpose of presenting them is to understand some of the reasons
why the reductive approach has been found unsatisfactory and replaced by non-reductive approaches, such as the \emph{manipulationist} or \emph{interventionist} accounts of counterfactuals and causation (Pearl, Woodward, Halpern, Hitchcock, Briggs, among others).


\subsection{Conditional probabilities}

Two well known endeavours to connect causal relationships with conditional
probabilities are due to P. Suppes (\cite{Sup1970}) and N. Cartwright (\cite{Car1983}).
For instance, Cartwright (\cite{Car1983}, p. 26) requires that causes raise
the probabilities of their effect:
\begin{description}
\item [{(CC)}] $C$ causes $E$ iff $Pr(E/C,K_{j})>Pr(E/K_{j})$ for all
state descriptions $K_{j}$ which satisfy certain conditions. 
\end{description}
Woodward (\cite{Woo2001}) finds (CC) defective in two ways. Firstly, the requirement
to conditionalize on all $K_{j}$ is too strong, and a weaker, existential
condition would suffice. Secondly, (CC) holds, as initially intended,
only for positive causes. When negative (inhibiting) changes are taken
into account, it is natural to replace $Pr(E/C,K_{j})>Pr(E/K_{j})$
by $Pr(E/C,K_{j})\neq Pr(E/K_{j})$ which only requires $C$ and $E$
to be (probabilistically) dependent. With these two points in mind,
Woodward (\cite{Woo2001}) proposes to replace (CC) with something of the following
sort:
\begin{description}
\item [{({*})}] $X$ causes $Z$ if and only if $X$ and $Z$ are dependent
conditional on certain other factors $F$. 
\end{description}
where $X$ and $Z$ are variables standing for properties. 

The problem now becomes that of specifying the other factors $F$.
Cartwright suggests that they include other causes of $Z$ with the
exception of those which are on a causal chain from $X$ to $Z$.
She recognizes, however (\cite{Car1983}, p. 30; \cite{Car1989}, p. 95 ff; \cite{Woo2001}, p. 58) that the claim that we should never condition on all such
intermediate variables is too strong. Woodward (\cite{Woo2001}) makes it clear
that in order to understand these restrictions, and more broadly,
in order for the project of specifying the other factors $F$ to have
any hope of success, we need to refer to other causes of $Y$ and
the way they are connected. For instance, to see why it is inappropriate
to conditionalize on the variables which lie on a causal chain from
$X$ to $Z$, as Cartwright first suggestion goes, it is enough to
consider the causal structure 
\[
X\rightarrow Y \rightarrow Z
\]
Now if we were to conditionalize on $Y$, we would expect, intuitively,
$X$ and $Z$ to be independent, which according to the definition
above would result in $X$ not being a cause of $Z$. But this is not
what we want. 

On the other side, to see that the requirement of never conditionalizing
on the variables on causal chains from $X$ to $Z$ is too strong,
it is enough to consider the causal structure 
\begin{figure}[htbp]
	\centering
		\begin{tikzpicture}
		
		\node(Y) {$Y$};
		
		\node(X) at (-1.5,-1.2) {$X$};

		\node(Z)  at (1.5, -1.2) {$Z$}; 
		
		\draw[->] (X)--(Y) ; 
		
		\draw[->] (X)--(Z); 
		
		\draw[->] (Y)--(Z); 
		
		\end{tikzpicture}
\end{figure}
in which both $X$ and $Y$ are ``direct'' causes of $Z$ and are
on the causal paths between $X$ and $Z$. If the causal connection
between $X$ and $Z$ is to be reflected in the probabilistic dependence
of $Z$ on $X$ conditional on some other properties $F$, then these
other properties must include $Y$. In other words, to determine the
causal influence of $X$ on $Z$ we must take into account the influence
of $Y$ on $Z$. 

What all this shows, according to Woodward, is that the project to
connect causal relationships between $X$ and $Y$ with conditional
probabilities, which finds its expression in ({*}), goes via a mechanism
which provides information about other causes (contributing or or
total) of $y$ besides $X$ and how these causes are connected with
one another and with $Z$. (\cite{Woo2001}, p. 58.) 

One such mechanism is that of \textit{causal Bayesian networks} (Pearl 2000/2009 \cite{Pea2000},
Spirtes, Glymour and Scheines 1993/2001 \cite{SpiGlySch1993}). They are built on Directed
Acyclic Graphs (DAGs) of the kind we have already encountered in our
earlier examples. We start with a set of variables $V=\left\{ X_{1},....X_{n}\right\} $
whose causal relationships we want to investigate and a set of (directed)
edges. A directed edge from the variable $X$ (parent) to the variable
$Z$ (child) is intended to represent the fact that $X$ is a \textit{direct
cause} of $Z$. The set of all parents of $Z$ is denoted as $PA_Z$. The project is now to investigate the connection between
causal relationships generated by the edges of the graph and conditional
probabilities $P(Z/PA_{Z})$ determined by a joint probability distribution
$P$ over $V$. More exactly, let $P$ be a joint probability distribution
on the set $V$. We denote by $x_{1},x_{2},...,x_{n}$ the values
associated with the variables in $V$, that is, $X_{1}=x_{1},...,X_{n}=x_{n}.$
A DAG $G$ is said to \textit{represent} $P$ if the equation 
\begin{equation}
P(x_{1},...,x_{n})=\prod_{i}P(x_{i}/pa_{i})
\end{equation}
 holds, where the variables $PA_{i}$ obey
the parent-child relation of $G$: $Y\in PA_{i}$ if and only if there
is an arrow in $G$ from $Y$ to $X_{i}$. 

A well known result states the following:
\begin{teo}[Markov Condition, \cite{VerPea1990}] Let $G$ be a DAG with $V$ its set
of variables. $G$ represents a probability distribution $P$ if and
only if every variable in $V$ is independent of all its nondescendants
(in G) conditional on its parents. 
\end{teo}
A DAG $G$ which represents a probability distribution $P$ is often
referred as \textit{causal Bayesian network}. The Markov Condition
is thought to be important because it establishes a connection between
causal relationships as represented by the arrows of a DAG and dependence
relationships. For instance, if we take ``$Z$ is a nondescendant
of $X$'' to stand for ``$X$ does not cause $Z$'', then the Markov
condition implies that if ``$X$ does not cause $Z$'' then conditional
on its parents, $X$ is independent of $Z$ , which by contraposition
gives us the right-to-left direction of ({*}):
\begin{itemize}
\item If variables $X$ and $Z$ are dependent (that is: $Pr(X/Parents(X),Z) \neq Pr(X/Parents(X))$),
then $X$ causes $Z$.
\end{itemize}
(Cf. e.g. \cite{Woo2001})

As handy as causal Bayesian networks are to connect causes with conditional
probabilities, they fail to represent caunterfactual reasoning (\cite{Pea2000}, p. 37). In other words, if we want a robust notion of cause which
sustains counterfactuals, we need to supplement causal Bayesian networks
with an additional, deterministic component.

\subsection{Counterfactual dependence}  \label{SUBSCOUNT}

Accounts of causal relationships based on counterfactual dependence
have been available starting with the works of David Lewis and
G. H. von Wright. Lewis (\cite{Lew1973b}, \cite{Lew1979}) reduces causal relationships
to counterfactual dependence, which, in the end, is defined in terms
of similarity between possible worlds. Von Wright (\cite{Wri1971}) distinguishes
a causal connection between $p$ and $q$ from an accidental generalization
(the concomitance of $p$ and $q$) on the basis of the fact that the
former, unlike the latter, sustains a counterfactual assumption of
the form ''on occasions where $p$, in fact, was not the case, $q$
would have accompanied it, had $p$ been the case''; that is, $p$
is a state of affairs which we can \textit{produce or suppress at
will}. It is thus the \textit{manipulativity} of the antecedent which
is the individuating aspect of the cause factor (\cite{Wri1971},
p. 70). Lewis's account has been criticized for relying too much on
``dubious metaphysics'' and von Wright's account for being too ``anthropomorphic''. 

Both Lewis's and von Wright's accounts contain ingredients which have
been incorporated later on into interventionist accounts of counterfactuals
and causation. Roughly, one needs a mechanism to represent the exogenous
process which is the manipulation of the (variables of the) antecedent
of a counterfactual. This mechanism has become known as \textit{intervention}\footnote{Lewis does not speak of interventions, but their role is mirrored, in the work of Lewis, by the notions of ``local miracle'' and ``non-backtracking counterfactual''.}.
In addition, we need another mechanism to measure the \textit{effects}
the changes of the intervened variables have on the (variables of
the) consequent. This mechanism is encoded into the so-called \textit{structural
(functional) equations} (\cite{Pea2000}, \cite{Woo2001}). 

In more details, we divide the set of variables whose causal relationships
we want to investigate into two disjoint sets: a set $V$ of endogenous
variables and a set $U$ of exogenous variables. With each endogenous
variable $X\in V,$ an equation of the form 
\[
X=f_{X}(PA_{X})
\]
is associated, where the variables $PA_{X}\subseteq U\cup V\setminus\left\{ X\right\} $
are called the \textit{parents of} $X$. The standard interpretation
of a functional equation $X=f_{X}(PA_{X})$ is that of a law which
specifies the value of $X$ given every possible combination of the
values of $PA_{X}$. If we draw an arrow from each variable in $PA_{X}$
to the variable $X$ we obtain directed graphs as in the case of the causal
Bayesian frameworks mentioned in the previous section. The crucial
difference between the two frameworks is that in the present case,
instead of characterizing the child-parent relationships stochastically
in terms of conditional probabilities $P(X/PA_X),$ we characterize
them deterministically using the equations. 

Various notions of intervention have been proposed, both in the context
of causal Bayesian networks and in that of structural equations (e.g.,
\cite{SpiGlySch1993}, \cite{Woo1997}, \cite{Hau1998},
and \cite{Pea2000}). A detailed discussion of this variety is outside
the scope of this paper. Suffice it to say that an intervention $do(X=x)$
on a variable $X$ is an action which disconnects the variable $X$
from all the incoming arrows into $X$ while preserving all the other
arrows of the graph including those directed out of $X$. This is
known as the \textit{arrow-breaking} conception of interventions.
In the structural equations framework where the causal graph is induced
by the appropriate set of equations, the intervention $do(X=x)$ results
also in the alteration of that set: the equation $X=f_{X}(PA_{X})$
associated with the variable $X$ is replaced with a new equation $X=x$,
while keeping intact the other equations in the set. 

It may be useful to illustrate these notions by way of an example
(\cite{Woo2001}). 

Consider the following two equations:

\[
\begin{array}{cccccl}
(6) & Y=aX &  &  & (7) & Z=bX+cY\end{array}
\]
This set of equations induces the DAG 

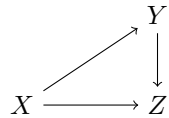
\begin{figure}[htbp]
	\centering
		\begin{tikzpicture}
		
		\node(Y) {$Y$};
		
		\node(X) at (-1.8,-1.2) {$X$};

		\node(Z)  at (0, -1.2) {$Z$}; 
		
		\draw[->] (X)--(Y) ; 
		
		\draw[->] (X)--(Z); 
		
		\draw[->] (Y)--(Z); 
		
		\end{tikzpicture}
\end{figure}

 If we intervene on $Y$ and set its value to $1$ (i.e., $do(Y=1))$,
the result will be the altered system of equations: 
\[
\begin{array}{cccccc}
(6') & Y=1 &  &  & (7) & Z=bX+cY\end{array}
\]
corresponding to the new DAG: 

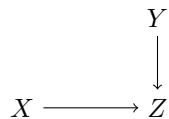
\begin{figure}[htbp]
	\centering
		\begin{tikzpicture}
		
		\node(Y) {$Y$};
		
		\node(X) at (-1.8,-1.2) {$X$};

		\node(Z)  at (0, -1.2) {$Z$}; 
		
		
		\draw[->] (X)--(Z); 
		
		\draw[->] (Y)--(Z); 
		
		\end{tikzpicture}
\end{figure}

\subsection{Various notions of cause} \label{SUBSCAUSES}

Woodward (\cite{Woo1997},\cite{Woo2001}) uses interventions in the context of structural
equations to \textit{define} ``$X$ is a cause of $Y$''. It turns
out, however, that there are several distinct notions of cause, each
satisfying the central commitment of the manipulability theory of
causation: there is a causal relationship between $X$ and $Y$ whenever
there is a possible intervention that changes the value of $X$ such
that carrying it out changes the value of $Y$ (or its probability).
\cite{Woo2001}, p. 54). Here are several notions of cause:
\begin{description}
\item [{(DC)}] (Direct cause) A necessary and sufficient condition for
$X$ to be a direct cause of $Y$ with respect to some variable set
$Z$ is that there be a possible intervention on $X$ that will change
$Y$ (or the probability distribution of $Y$) when all the other
variables in $Z$ besides $X$ and $Y$ are held fixed at some values
by interventions. (\cite{Woo2001}, p. 52)
\end{description}
For Woodward, direct causes correspond to the parent-child relationships in the
underlying DAG. The other notions cannot be always recovered from
the arrows of the DAG: 
\begin{description}
\item [{(TC)}] (Total cause) $X$ is a total cause of $Y$ if and only
if it has a non-null total effect on $Y$- that is, if and only if
there is some intervention on $X$ alone such that for some values
of the other variables, this intervention on $X$ will change $Y$.
The total effect of a change $dx$ in $X$ on $Y$ is the change in
the value of $Y$ that would result from an intervention on $X$ alone
that changes it by amount $dx$ (given the values of other variables
that are not descendants of $X$). (\cite{Woo2001}, p. 54)
\item [{(CC)}] (Contributing cause) $X$ is a contributing cause of $Y$
if and only if it makes non-null contribution to $Y$ along some directed
path in the sense that there is some set of values of variables that
are not on this path such that if these variables were fixed at those
values, there is some intervention on $X$ that will change the value
of $Y$. The contribution to a change in the value of $Y$ due to
a change $dx$ in the value of $X$ along some directed path is the
change in the value of $Y$ that would result from this change in
$X$, given that the values of off path variables are fixed by
independent interventions. (\cite{Woo2001}, pp. 54-55)
\end{description}
For instance in our earlier example consisting of the set of equations
(6) and (7), $X$ is a direct cause of $Z$. To make this more transparent
we shall assume that the coefficients $a,b,c$ are
all equal to 1. 

We first intervene on $Y$ and set its value to $1$ as we did above
(i.e., $do(Y=1))$. The result will be the altered system of equations:
\[
\begin{array}{cccccc}
(6') & Y=1 &  &  & (7) & Z=X+Y.\end{array}
\]
 Next we perform two (independent) interventions on the system (6')-(7):
$do(X=1)$ yields $Z=2$; and $do(X=2)$ yields $Z=3$. We conclude
that $X$ is a direct cause of $Z$ in the sense of (DC).

Consider now the set of equations 
\[
\begin{array}{cccccc}
(8) & Y=aX &  &  & (9) & Z=dY\end{array}
\]

which corresponds to the DAG 
\[
X\longrightarrow Y  \longrightarrow  Z
\]

Any intervention $do(Y=e)$ leads to the system of equations 
\[
\begin{array}{cccccc}
(8') & Y=e &  &  & (9) & Z=dY\end{array}
\]
 Now it is obvious that no change in the value of $X$ will have any
influence on the value of $Z$, hence $X$ is not a direct cause of
$Z$ as expected. On the other side, it is easy to see that $X$ is
a total cause of $Z$ in the sense of (TC) -- except in the special case that $d=\frac{1}{a}$. 

In section \ref{CAUSMOD} we shall represent some of these causal notions in the framework of causal team semantics, to which we now turn.

\section{Causal team semantics} \label{TEAMSEC}

\subsection{Teams}

Team semantics was introduced by W. Hodges (\cite{Hod1997},\cite{Hod1997b}) in order to provide a compositional presentation of the (game-theoretically defined) semantics of Independence-Friendly logic (\cite{HinSan1989},\cite{ManSanSev2011}). In the following years, team semantics has been used to extend first-order logic with database dependencies (e.g. Dependence logic \cite{Vaa2007}, Independence logic \cite{GraVaa2013}, Inclusion logic \cite{Gal2012}); similar approaches have been applied to propositional logics (\cite{YanVaa2016},\cite{YanVaa2017}) and modal logics (\cite{Vaa2008}, \cite{Tul2003}, \cite{BraFro2002}). Appropriate generalizations of teams have been used as descriptive languages for probabilistic dependencies (\cite{DurHanKonMeiVir2016}), for quantum phenomena (\cite{HytPaoVaa2015}), for Bayes networks (\cite{CorHytKonPenVaa2016}). 

The basic idea of team semantics is that notions such as dependence and independence, which express properties of relations (instead of individuals), cannot be captured by Tarskian semantics, which evaluates formulas on single assignments\footnote{This can be formally proved, see \cite{CamHod2001}.}; the appropriate unit for semantical evaluation is instead the \emph{team}, i.e., a \emph{set} of assignments (all sharing a common variable domain). In the standard approach, all the values of the variables come from a unique domain of individuals associated with an underlying model. However, in order to model causal and counterfactual dependence, we shall need to relax
the assumption that variables may take as values only individuals coming from a common domain.
Instead we shall take variables to represent properties which may have 
all kinds of values (as specified by their range). That is, once
a set $Dom$ of variables is fixed, each assignment will be a mapping
$s:\:Dom\rightarrow\bigcup_{X\in Dom}Ran(X)$ such that $s(X)\in Ran(X)$
for each $X\in Dom$. A \textbf{team} $T$ of domain $dom(T)=Dom$ will be any set of such assignments. 

As an example, recall Woodward's DAG  corresponding to the structural
equations (6) and (7) (subsection \ref{SUBSCOUNT}). We may take $X$ to express the property ``(whether
it is) winter'', $Y$ the property ``(whether it is) cloudy'' and
$Z$ the property ``(whether it is) snowing'' and take them to be
represented in the team $T=\left\{ s_{1},s_{2}\right\} $:

\[
\begin{array}{c|c|c|c}
 & X & Y & Z\\
\hline s_{1} & 0 & 0 & 0\\
\hline s_{2} & 1 & 1 & 1
\end{array}
\]

The basic semantic relation is now $T\models\psi$: the team $T$
satisfies the formula $\psi$. 

One can define a team semantics already for classical propositional languages. However
such a semantics does not really add anything new, in the sense that 

$T\models\psi$ if and only if for all $s\in T$, $s\models\psi$ (in the Tarskian sense). 

That is, the meaning of a first-order formula is always reducible
to a property of single assignments. However, once their semantics is
expressed in terms of teams, propositional languages can be extended in
ways that would be unavailable within Tarskian semantics; for instance,
one can add (functional) dependence atoms $\dep{X_1,\dots,X_n}{Y}$ whose
semantics is defined by 
\begin{description}
\item [{({*})}] $T\models\dep{X_1,\dots,X_n}{Y}$ if and only if for all $s,s'\in X$,
if $s(X_1)=s'(X_1),\dots,$ $s(X_n)=s'(X_n)$, then $s(Y)=s'(Y)$ 
\end{description}
expressing that $Y$ is functionally determined by $\{X_1,\dots,X_n\}$;
and this is a global property of the team, not reducible to properties
of the single assignments\footnote{Actually, it is a property of \emph{pairs} of assignments of $X$, but more complex formulas arise whose meaning cannot be analogously reduced.}. 

Thus, in our example, it holds that whether it is snowing depends completely
on whether it is winter, $T\models=(X;Z)$, and whether it is cloudy,
$T\models\dep{Y}{Z}$; and whether it is cloudy depends entirely on whether
it is winter, $T\models\dep{X}{Y}$. 

A team therefore might be used, for example, to represent a set of individual records coming from a statistical or experimental investigation; or, to represent all possible configurations that are compatible with the ranges of each variable; or yet, a subset of all possible configurations. This last case is particularly important in our context: it may well happen that some configurations are forbidden, even though they respect all variable ranges; this in particular happens if there are functional dependencies between the variables. For what regards the first possibility (analysis of statistical data) for many purposes it may be more suitable to use \emph{multi}teams, which allow multiple copies of assignments. Team-theoretical logical languages could thus be used to express global properties of distributions of values. On a second, different interpretation, teams and multiteams may be used to represent epistemic uncertainty about the current state of affairs; if one thinks of each assignment as a possible world, then the team may be thought as representing a set of equally plausible worlds. An intervention on such an object should, therefore, produce a new set of equally plausible candidates for the actual world.

\subsection{Causal teams}

Despite some claims to the contrary in the literature, teams
are insufficient to represent counterfactuals and causal notions based
on them. As explained in subsection \ref{SUBSCOUNT}, a proper treatment of causal notions requires an account of counterfactual information, as encoded e.g. in invariant structural equations.  It is true that a team sustains a number of functional dependencies among variables; but such dependencies may well be contingent, and disappear if the system is intervened upon. 
For this reason, we must extend teams so that they incorporate invariant dependencies or functions; and we must explain what may count as an intervention \emph{on a team}.


 Besides the ideal case in which a causal model contains a complete description of the functions involved in the structural equations (\emph{parametric} case), we develop our semantics with enough generality as to accomodate the more realistic case in which we possess only partial information about the functions (\emph{nonparametric} case); as an extreme case, the model might only incorporate information as to which functional \emph{dependencies} are invariant. In the nonparametric case, not all counterfactual statements can be evaluated; the additional logical complications related to this case will be examined in section \ref{NONPARAM}. Most of the paper will focus on the parametric case.

Before proceeding, we want to fix some notational conventions. 
  As is often done in the literature on causal models, we use the symbol $PA_Y$ ambiguously, so that it may mean either the set of parents of $Y$, or a sequence of the same variables in some fixed alphabetical ordering. For other sets/sequences of variables, we will adhere to the following conventions:

\begin{notat}
 \begin{itemize}
\item We use boldface letters such as \textbf{X} to denote either a set $\{X_1,\dots,X_n\}$ of variables or a sequence of the same variables (in the fixed alphabetical order) 
\item We use \textbf{x} to denote a set or sequence of values, each of which is a value for exactly one of the variables in \textbf{X}. We leave the details of these correspondences between variables and values as non-formalized.  
\item Writing $s(\SET X)$ we mean the set/sequence of values $s(X_1),\dots,s(X_n)$ that the assignment $s$ assigns to each of the variables in $\SET X$
\item $Ran(\SET X)$ is an abbreviation for $\prod_{X\in \SET X} Ran(X)$
\item By $\SET{X}\setminus \SET{Y}$ we denote the set/the sequence (in alphabetical order) of variables occurring in $\SET{X}$ but not in $\SET{Y}$, and by $\SET{x}\setminus\SET{y}$ a corresponding set/sequence of values
\item By $\SET{X}\cap\SET{Y}$  we denote the set/sequence of variables occurring in both $\SET{X}$ and $\SET{Y}$, and by $\SET{x}\cap\SET{y}$ the corresponding set/sequence of values
\end{itemize}

and so on.
\end{notat}

Given a team $T^-$ and a variable $X\in dom(T^-)$, we write $T^-(X)$ for the set of values that are obtained for $X$ in the team $T^-$; that is, $T^-(X)= \{s(X)|s\in T^-\}$. 
As before, we say that a team $T^-$ satisfies a functional dependence $\dep{X_1,\dots, X_n}{Y}$, and we write $T^-\models \dep{X_1,\dots,X_n}{Y}$, if:
\[
\text{for all } s,s'\in T^- \text{, if }s(X_i) = s'(X_i) \text{ for all }i=1..n, \text{ then }s(Y) = s'(Y).  
\]

\begin{df}
A \textbf{causal team} $T$ over variable domain $dom(T)$ with endogenous variables $\mathbf V\subseteq dom(T)$ is a quadruple $T = (T^-,G(T),\mathcal{R}_T,\mathcal{F}_T)$, where:
\begin{enumerate}
\item $T^-$ is a team.
\item $G(T) =(dom(T),E)$ is a graph over the set of variables. For any $X\in dom(T)$, we denote as $PA_X$ the set of all variables $Y\in dom(T)$ such that the arrow $(Y,X)$ is in $E$.
\item $\mathcal{R}_T = \{(X,Ran(X))|X\in dom(T)\}$ 
 (where the $Ran(X)$ may be arbitrary sets) is a function which assigns a range to each variable
\item $\mathcal{F}_T$ is a function $\{(V_i,f_{V_i})|V_i\in\mathbf V\}$ that assigns to each endogenous variable a $|PA_{V_i}|$-ary function $f_{V_i}:dom(f_{V_i})\rightarrow 
ran(V_i)$ \\(for some $dom(f_{V_i})\subseteq Ran(PA_{V_i})$)
\end{enumerate}
which satisfies the further restrictions:
\begin{enumerate}[a)]
\item $T^-(X) \subseteq Ran(X)$ for each $X\in dom(T)$
\item If $PA_Y=\{X_1,\dots, X_n\}$, then $T^-\models \dep{X_1,\dots, X_n}{Y}$
\item if $s\in T^-$ is such that $s(PA_Y)\in dom(f_Y)$, then $s(Y)= f_Y(s(PA_Y))$.
\end{enumerate}
In case $dom(f_V)= Ran(PA_V)$ for each $V\in \SET V$, we say the causal team is \textbf{parametric}; otherwise it is \textbf{nonparametric}.
\end{df}

Clause b) is there to ensure that whenever the graph contains an arrow $X_i\rightarrow Y$, and $\{X_1,\dots X_n\}$ is the maximal set of variables whence arrows come to $Y$, then the team satisfies the corresponding functional dependency $\dep{X_1,\dots X_n}{Y}$. Clause c) further ensures that such functional dependency is in accordance with the (partial description of the)  function $f_Y \in \mathcal{F}_T$.

The functional component $\mathcal F_T$ induces an associated system of structural equations, say
\[
Y := \mathcal F_T(Y)(PA_Y)
\]
for each variable $Y\in dom(T)$.

\begin{example} \label{EXCAUSALTEAM}
 Consider a causal team $T$ which has underlying team $T^- =\{\{(U,2),(X,1),(Y,2),(Z,4)\},$ $\{(U,3),$ $(X,1),(Y,2),(Z,4)\},\{(U,1),(X,3),(Y,3),$ $(Z,1)\},\{(U,1),(X,4),(Y,1),(Z,1)\},\{(U,4),(X,4),$ $(Y,1),(Z,1)\}\}$, graph $G(T) = (\{U,X,Y,Z\},$ $ \{(U,Z),(X,Y),(Y,Z),(X,Z)\})$, ranges $Ran(U) = Ran(X)$ $ = Ran(Y) = Ran(Z) = \{1,2,3,4\}$, and partial description of (one value of) the invariant function for $Z$: $\mathcal F(Z)(4,1,2):= 3$. We represent the $T^-$ and $G(T)$ components of $T$ by means of a decorated table:
\begin{center}
\begin{tabular}{|c|c|c|c|}
\hline
 \multicolumn{4}{|l|}{ } \\
 \multicolumn{4}{|l|}{U\tikzmark{U100} \ \tikzmark{X100}X\tikzmark{X100'} \ \  \tikzmark{Y100}Y\tikzmark{Y100'} \,  \tikzmark{Z100}Z} \\
\hline
 $2$ & $1$ & $2$ & $4$\\
\hline
 $3$ & $1$ & $2$ & $4$\\
\hline
 $1$ & $3$ & $3$ & $1$\\
\hline
 $1$ & $4$ & $1$ & $1$\\
\hline
$4$ & $4$ & $1$ & $1$\\ 
\hline
\end{tabular}
 \begin{tikzpicture}[overlay, remember picture, yshift=0.25\baselineskip,
 shorten >=.5pt, shorten <=.5pt]
  \draw [->] ([yshift=3pt]{pic cs:X100'})  [line width=0.2mm] to ([yshift=3pt]{pic cs:Y100});
	\draw [->] ([yshift=3pt]{pic cs:Y100'})  [line width=0.2mm] to ([yshift=3pt]{pic cs:Z100});
  \draw ([xshift = 1pt,yshift=7pt]{pic cs:X100'})  edge[line width=0.2mm, out=35,in=125,->] ([yshift=6pt]{pic cs:Z100});
	\draw ([yshift=8pt]{pic cs:U100})  edge[line width=0.2mm, out=35,in=125,->] ([yshift=8pt]{pic cs:Z100});
  \end{tikzpicture}
\end{center}
\end{example}

\subsection{Explicit causal teams}   \label{SUBSEXPTEAM}

For many purposes -- first of all, to keep a smoother account of the operations of taking subteams, and of applying iterated interventions -- it will be convenient to restrict attention to causal teams of a special form. This restriction causes no loss of generality within the developments pursued in the present paper. 

\begin{df}
A causal team $T=(T^-,G(T),\mathcal{R}_T,\mathcal{F}_T)$ with endogenous variables $\mathbf{V}$ is \textbf{explicit} if, for every $V\in\mathbf{V}$, the following additional condition holds:

d) Let $(X_1,\dots,X_n)$ be the list of the parents of $V$ in the fixed alphabetical order. For every $s\in T^-$, 
$(s(X_1),\dots,s(X_n))\in dom(f_Y)$.
\end{df}

(Here, as in the definition of causal team, $f_Y$ is a shorthand for $\mathcal F_T(Y)$).

In words, the $\mathcal F_T$ component of an explicit causal team encodes \emph{all} the information of the team that concerns invariant functions; no further values of the functions can be reconstructed from the team component $T^-$.

Given any causal team, it is always possible to construct in a canonical way an explicit causal team that corresponds to it, in the sense that it encodes exactly the same information (but it is more stable under the operations of taking subteams or interventions -- to be defined in the following subsections). 
For this purpose, given a causal team $T=(T^-,G(T),\mathcal{R}_T,\mathcal{F}_T)$ with endogenous variables $\mathbf{V}$, to any variable $Z\in \mathbf{V}$ we may associate an 
\textbf{explicit function} $h^T_Z:Ran(W)\rightarrow Ran(Z)$ as follows; given $pa_Z \in Ran(W)$, define\\ 

$h^T_Z(pa_Z)=$
				$\left\{\begin{array}{cc}
															f_Z(pa_Z) & \text{if this is defined in $T$} \\
															s(Z) & \text{if there is some row $s$ in $T$ with $s(PA_                                            Z) = pa_Z$}\\
				                      \end{array}\right.$   \\
\\
(Conditions b) and c) in the definition of causal team ensure that $h_Z^T$ is well-defined.)

We collect these explicit functions in a single function $\mathcal{H}_T$ such that, for each variable $Z\in\mathbf{V}$, $\mathcal{H}_T(Z):= h^T_Z$. Then:
\begin{center}
The explicit causal team associated to $T$ is $T^E:=(T^-,G(T),\mathcal R_T,\mathcal H_T)$.
\end{center}

\subsection{Causal subteams}

It will be important, in order to define a semantics for our languages, to talk about causal subteams. A causal subteam $S$ of a causal team $T$ is meant to express a condition of lesser uncertainty; this will be encoded by the fact that the assignments in $S^-$ form a subset of the assignments of $T^-$, which may be interpreted as the fact that less configurations are considered possible. At the same time, the transition to a subteam should not erase information concerning the graph, the ranges of variables, and the invariant functions. The definitions in the previous subsection should make it clear that this proviso is easily guaranteed if $T$ is an \emph{explicit} causal team. In this case, we can define:

\begin{df}
Given an explicit causal team $T$, a \textbf{causal subteam} $S$ of $T$ is a causal team with the same domain and the same set of endogenous variables, which satisfies the following conditions:
\begin{enumerate}
\item $S^-\subseteq T^-$
\item $G(S) = G(T)$
\item $\mathcal{R}_S = \mathcal{R}_T$
\item $\mathcal{F}_S = \mathcal{F}_T$.
\end{enumerate}
\end{df}

In case the team $T$ is not explicit, what can go wrong is that, for some endogenous variable $Y$, there may be some assignment $s\in T^- \setminus S^-$ such that $(s(X_1),\dots,s(X_n))\notin dom(f_Y)$ (where $X_1,\dots,X_n$ list $PA_Y$ in alphabetical order); in such case, the team $T$ encodes the fact that the invariant function which produces $Y$ assigns, to the list of arguments $(s(X_1),\dots,s(X_n))$, the value $s(Y)$; but this information is lost in the subteam $S$. To avoid this problem, we can define more generally:

\begin{df}
Given a causal team $T$, a \textbf{causal subteam} $S$ of $T$ is a causal subteam of the associated explicit team $T^E$ (as defined in subsection \ref{SUBSEXPTEAM})
\end{df}

This second definition obviously coincides with the previous one over explicit causal teams.

\subsection{A basic language and its semantics}   \label{BASICLAN}

Before discussing interventions and counterfactuals, we need to specify what it means for a causal team to satisfy atomic formulas and their boolean combinations. The kind of language we consider, for now, contains atomic dependence statements of the form $\dep{\SET X}{Y}$; atomic formulas of the forms $Y=y$ and $Y\neq y$, where $Y\in dom(T^-)$ and $y\in 
Ran(Y)$; connectives $\land$ and $\lor$.

By analogy with the other kinds of team semantics that have been proposed in the literature, we can define satisfaction of a formula by a causal team by the clauses: 
\begin{itemize}
\item $T\models \dep{\SET X}{Y}$ if for all $s,s'\in T$, $s(\SET X)=s'(\SET X)$ implies $s(Y)=s'(Y)$.
\item $T\models Y=y$ if, for all $s\in T^-$, $s(Y)=y$.
\item $T\models Y\neq y$ if, for all $s\in T^-$, $s(Y)\neq y$.
\item $T\models \psi\land \chi$ if $T\models \psi$ and $T\models \chi$.
\item $T\models \psi\lor \chi$ if there are two causal subteams $T_1,T_2$ of $T$ such that $T_1^-\cup T_2^- = T^-$, $T_1\models \psi$ and $T_2\models \chi$.\footnote{Notice that defining the union of any pair of \emph{causal} teams of the same domain is problematic, as the information on invariant functions given by each team might be incompatible with the information encoded in the other team.}
\end{itemize}

\subsection{Selective implication}

Our main goal is to give an exact semantics to counterfactual statements
of the form ``If $\psi$ had been the case, then $\chi$ would have
been the case''. Very often, however, one find examples in the literature
where these statements are embedded into a larger context. We have
seen that von Wright (\cite{Wri1971}) considers examples of the form ``on occasions
where $p$, in fact was not the case, $q$ would have accompanied
it, had $p$ been the case''. Pearl (\cite{Pea2000}) analyzes the following
query: ``what is the probability $Q$ that a subject who died under
treatment $(X=1,Y=1)$ would have recovered $(Y=0)$ had he or she
not been treated $(X=0)$? 

The appropriate representation of the last statement seems to be:
\[
(X=1\wedge Y=1)\supset(X=0\cf Y=0).
\]
where the symbol $\cf$ stands for counterfactual implication, while the \emph{selective  implication} $\supset$ is a form of restriction of the range of application of the
counterfactual to the available evidence. 
 What it does is to generate a subteam by selecting those assignments which satisfy the antecedent; and then it is checked whether the consequent holds in this subteam. 

Given a causal team $T$, and a classical formula $\psi$ (that is, a formula as in subsection \ref{BASICLAN}, but without dependence atoms) define the subteam $T^\psi$ by the condition:
\begin{itemize}
\item $(T^\psi)^- = \{s\in T^- | \{s\}\models \psi\}$.
\end{itemize}   

Then, we define selective implication by the clause:
\begin{itemize}
\item $T\models \psi
\supset \chi$ iff $T^\psi
\models \chi$.
\end{itemize}
Here the consequent $\chi$ can be any formula of our current logical language; therefore, it might happen not to be a property of single assignments. Instead, we require for now the antecedent to be classical. The general idea is that selective implication is a reasonable operator only for antecedents which are \emph{flat} formulas, in the sense with which this word is used in the literature on logics of dependence (which will be reviewed in the following subsections). Actually, typical applications involve at most conjunctions of atomic formulas of the type $Z=z$. 

\begin{example}
We observe that the selective implication
\[
T\models Z=3 \supset Y=2
\]
holds on any causal team that is based on the team $T$ which is depicted in the figure:

\begin{center}
\begin{tabular}{|c|c|c|}
\hline
 \multicolumn{3}{|c|}{Z \ Y \ X} \\
\hline
 $1$ & $2$ & $3$ \\
\hline
 $2$ & $1$ & $1$ \\
\hline
 $3$ & $2$ & $1$ \\
\hline
 $3$ & $2$ & $2$ \\
\hline
\end{tabular}
\end{center}

To see that the formula holds on it, we have to construct the reduced subteam $T^{Z=3}$:

\begin{center}
\begin{tabular}{|c|c|c|}
\hline
 \multicolumn{3}{|c|}{Z \ Y \ X} \\
\hline
 $3$ & $2$ & $1$ \\
\hline
 $3$ & $2$ & $2$ \\
\hline
\end{tabular}
\end{center}

which is obtained by selecting the third and fourth row of the previous table (the rows that satisfy $Z=3$). We can see, then, that $Y=2$ is satisfied by each row of this smaller table. Therefore, by the semantical clause for this kind of atomic formulas, $T^{Z=3}\models Y=2$. Then, the semantical clause for selective implication allows us to conclude that $T\models Z=3 \supset Y=2$.

Notice that we only needed the team structure in order to evaluate a selective implication; all the information required to evaluate it is already encoded in the team, and no information about the structural equations or the underlying DAG is needed.\\
\end{example}

\subsection{Interventions on teams: some examples}

 Our goal is to define (interventionist) \emph{counterfactual implication}. The task is more complicated than in the case of selective implication; we first illustrate the idea with some  examples; formal definitions will be provided after that. Informally, the idea is that a counterfactual 
$X=x\cf\psi$ is true in the causal team $T$ if $\psi$ is
true in the causal team which results from an intervention $do(X=x)$ applied to
 the team $T$.

\begin{example}
Consider any causal team $T$ with $T^- = \{ \{(X,1),(Y,2),(Z,3)\}, $ $ \{(X,2),(Y,1),(Z,4)\}, \{(X,4),(Y,1),(Z,4)\}, \{(X,3),(Y,3),(Z,4)\} \}$, and underlying graph $G(T)= (\{X,Y,Z\}, \{(X,Y), (Y,Z)\})$ (we omit specifying the $\mathcal R_T$ and $\mathcal F_T$ components). We can represent it as an annotated table:

\begin{center}

\begin{tabular}{|c|c|c|}
\hline
 
 \multicolumn{3}{|c|}{X\tikzmark{X} \  \tikzmark{Y}Y\tikzmark{Y'} \  \tikzmark{Z}Z} \\
\hline
 $1$ & $2$ & $3$ \\
\hline
 $2$ & $1$ & $4$ \\
\hline
 $4$ & $1$ & $4$ \\
\hline
 $3$ & $3$ & $4$ \\
\hline
\end{tabular}
 \begin{tikzpicture}[overlay, remember picture, yshift=.25\baselineskip, shorten >=.5pt, shorten <=.5pt]
  \draw [->] ([yshift=3pt]{pic cs:X})  [line width=0.2mm] to ([yshift=3pt]{pic cs:Y});
	\draw [->] ([yshift=3pt]{pic cs:Y'})  [line width=0.2mm] to ([yshift=3pt]{pic cs:Z});
  \end{tikzpicture}
\end{center}
We want to establish whether the counterfactual $Y=2 \cf Z=3$ holds in $T$. The idea is to intervene in $T$ by setting the value of $Y$ to $2$ (this corresponds to replacing the function $f_y$ with the constant function $2$); updating all other variables that \emph{invariantly} depend on $Y$; and removing all the arrows that enter into $Y$. The causal team thus produced will be denoted by $T_{Y=2}$. So, first we intervene on $Y$:

\begin{center}

\begin{tabular}{|c|c|c|}
\hline
 
 \multicolumn{3}{|l|}{X \ \  Y \tikzmark{Y1} \  \tikzmark{Z1}Z} \\
\hline
 $1$ & $\mathbf{2}$ & $\dots$ \\
\hline
 $2$ & $\mathbf{2}$ & $\dots$ \\
\hline
 $4$ & $\mathbf{2}$ & $\dots$ \\
\hline
 $3$ & $\mathbf{2}$ & $\dots$ \\
\hline
\end{tabular}
 \begin{tikzpicture}[overlay, remember picture, yshift=.25\baselineskip, shorten >=.5pt, shorten <=.5pt]
	\draw [->] ([yshift=3pt]{pic cs:Y1})  [line width=0.2mm] to ([yshift=3pt]{pic cs:Z1});
  \end{tikzpicture}
\end{center}
$Z$ is the only variable which has an invariant dependence on $Y$; therefore, we have to update its column. Since the function that determines $Z$ has $Y$ as its only parameter, we just have to consult $T$ and see that, in rows where $Y$ has value $2$, $Z$ takes value 3. Therefore, $T_{Y=2}$ looks like this:
\begin{center}

\begin{tabular}{|c|c|c|}
\hline
 
 \multicolumn{3}{|l|}{X \   Y \tikzmark{Y2} \  \tikzmark{Z2}Z} \\
\hline
 $1$ & $\mathbf{2}$ & $3$ \\
\hline
 $2$ & $\mathbf{2}$ & $3$ \\
\hline
 $4$ & $\mathbf{2}$ & $3$ \\
\hline
 $3$ & $\mathbf{2}$ & $3$ \\
\hline
\end{tabular}
 \begin{tikzpicture}[overlay, remember picture, yshift=.25\baselineskip, shorten >=.5pt, shorten <=.5pt]
	\draw [->] ([yshift=3pt]{pic cs:Y2})  [line width=0.2mm] to ([yshift=3pt]{pic cs:Z2});
  \end{tikzpicture}
\end{center}
(Again, we are omitting a representation of the $\mathcal R_T$ and $\mathcal F_T$ components). Now $T_{Y=2}\models Z=3$, therefore we conclude that $T\models Y=2 \cf Z=3$. 
The team $T_{Y=2}$ describes a counterfactual situation in which we have certainty about the values of $Y$ and $Z$, but not about the value of $X$.

Notice that this example uses both the team and the graph structure, but the specific invariant functions are not needed in the evaluation of this specific sentence. However, the observations above do not cover, for example, evaluation of counterfactuals of the form $Y=4\cf\dots$, because team and graph structure do not tell us anything about what value should $Z$ take in circumstances in which $Y=4$.

\end{example}

\begin{example}
We consider a slightly less trivial example. Here we evaluate the counterfactual $Y=2\cf (Z=2\lor Z=3)$ in the causal team $T$ shown in the picture below

\begin{center}
\begin{tabular}{|c|c|c|}
\hline
 \multicolumn{3}{|c|}{ }\\ 
 \multicolumn{3}{|l|}{X\tikzmark{X3} \  \tikzmark{Y3}Y\tikzmark{Y3'} \  \tikzmark{Z3}Z} \\
\hline
 $1$ & $1$ & $1$ \\
\hline
 $1$ & $2$ & $2$ \\
\hline
 $2$ & $2$ & $3$ \\
\hline
\end{tabular}
 \begin{tikzpicture}[overlay, remember picture, yshift=.25\baselineskip, shorten >=.5pt, shorten <=.5pt]
	\draw [->] ([yshift=3pt]{pic cs:Y3'})  [line width=0.2mm] to ([yshift=3pt]{pic cs:Z3});
  \draw ([yshift=8pt]{pic cs:X3})  edge[line width=0.2mm, out=55,in=125,->] ([yshift=7pt]{pic cs:Z3});
  \end{tikzpicture}
\end{center}

(The components $\mathcal R_T$ and $\mathcal F_T$ are omitted as before). So we must produce again $T_{Y=2}$. First we intervene on $Y$ 

\begin{center}
\begin{tabular}{|c|c|c|}
\hline
 \multicolumn{3}{|c|}{ }\\ 
 \multicolumn{3}{|l|}{X\tikzmark{X4} \, \tikzmark{Y4}Y \tikzmark{Y4'} \ \  \tikzmark{Z4}Z} \\
\hline
 $1$ & $\mathbf{2}$ & $\dots$ \\
\hline
 $1$ & $\mathbf{2}$ & $\dots$ \\
\hline
 $2$ & $\mathbf{2}$ & $\dots$ \\
\hline
\end{tabular}
 \begin{tikzpicture}[overlay, remember picture, yshift=.25\baselineskip, shorten >=.5pt, shorten <=.5pt]
	\draw [->] ([yshift=3pt]{pic cs:Y4'})  [line width=0.2mm] to ([yshift=3pt]{pic cs:Z4});
  \draw ([yshift=8pt]{pic cs:X4})  edge[line width=0.2mm, out=55,in=125,->] ([yshift=7pt]{pic cs:Z4});
  \end{tikzpicture}
\end{center}

and next we update the value of $Z$, taking into account both the values of $X$ and $Y$ in each row (since $Z$ is reached both by arrows coming from $X$ and from $Y$). Notice that now two of the modified rows become identical, so that only one appears in the table:

\begin{center}
\begin{tabular}{|c|c|c|}
\hline
 \multicolumn{3}{|c|}{ }\\ 
 \multicolumn{3}{|c|}{X\tikzmark{X5} \  \tikzmark{Y5}Y\tikzmark{Y5'} \  \tikzmark{Z5}Z} \\
\hline
 $1$ & $\mathbf{2}$ & $2$\\
\hline
 $2$ & $\mathbf{2}$ & $3$\\
\hline
\end{tabular}
 \begin{tikzpicture}[overlay, remember picture, yshift=.25\baselineskip, shorten >=.5pt, shorten <=.5pt]
	\draw [->] ([yshift=3pt]{pic cs:Y5'})  [line width=0.2mm] to ([yshift=3pt]{pic cs:Z5});
  \draw ([yshift=8pt]{pic cs:X5})  edge[line width=0.2mm, out=55,in=125,->] ([yshift=7pt]{pic cs:Z5});
  \end{tikzpicture}
\end{center}

This team can be partitioned into two causal subteams:

\begin{center}
\begin{tabular}{|c|c|c|}
\hline
 \multicolumn{3}{|c|}{ }\\ 
 \multicolumn{3}{|c|}{X\tikzmark{X6} \  \tikzmark{Y6}Y\tikzmark{Y6'} \  \tikzmark{Z6}Z} \\
\hline
 $1$ & $2$ & $2$\\
\hline
\end{tabular}
 \begin{tikzpicture}[overlay, remember picture, yshift=.25\baselineskip, shorten >=.5pt, shorten <=.5pt]
	\draw [->] ([yshift=3pt]{pic cs:Y6'})  [line width=0.2mm] to ([yshift=3pt]{pic cs:Z6});
  \draw ([yshift=8pt]{pic cs:X6})  edge[line width=0.2mm, out=55,in=125,->] ([yshift=7pt]{pic cs:Z6});
  \end{tikzpicture}
\begin{tabular}{|c|c|c|}
\hline
 \multicolumn{3}{|c|}{ }\\ 
 \multicolumn{3}{|c|}{X\tikzmark{X7} \  \tikzmark{Y7}Y\tikzmark{Y7'} \  \tikzmark{Z7}Z} \\
\hline
 $2$ & $2$ & $3$\\
\hline
\end{tabular}
 \begin{tikzpicture}[overlay, remember picture, yshift=.25\baselineskip, shorten >=.5pt, shorten <=.5pt]
	\draw [->] ([yshift=3pt]{pic cs:Y7'})  [line width=0.2mm] to ([yshift=3pt]{pic cs:Z7});
  \draw ([yshift=8pt]{pic cs:X7})  edge[line width=0.2mm, out=55,in=125,->] ([yshift=7pt]{pic cs:Z7});
  \end{tikzpicture}
\end{center}

the first of which satisfies $Z=2$, while the second satisfies $Z=3$. Therefore, by the semantical clause for disjunction,  $T_{Y=2}$ satisfies $Z=2 \lor Z=3$. We may conclude that $T\models Y=2\cf (Z=2\lor Z=3)$. 

\end{example}

\begin{example}
Here we present an example involving nonparametric teams. Our goal is to evaluate $X=1 \cf Y=2$ in \emph{the} nonparametric team $T$ shown in the picture:

\begin{center}
\begin{tabular}{|c|c|c|c|}
\hline
 \multicolumn{4}{|l|}{ } \\
 \multicolumn{4}{|l|}{U\tikzmark{U8} \ \tikzmark{X8}X\tikzmark{X8'} \ \  \tikzmark{Y8}Y\tikzmark{Y8'} \,  \tikzmark{Z8}Z} \\
\hline
 $2$ & $1$ & $2$ & $4$\\
\hline
 $3$ & $1$ & $2$ & $4$\\
\hline
 $1$ & $3$ & $3$ & $1$\\
\hline
 $1$ & $4$ & $1$ & $1$\\
\hline
\end{tabular}
 \begin{tikzpicture}[overlay, remember picture, yshift=.25\baselineskip, shorten >=.5pt, shorten <=.5pt]
  \draw [->] ([yshift=3pt]{pic cs:X8'})  [line width=0.2mm] to ([yshift=3pt]{pic cs:Y8});
	\draw [->] ([yshift=3pt]{pic cs:Y8'})  [line width=0.2mm] to ([yshift=3pt]{pic cs:Z8});
  \draw ([xshift=1,yshift=7pt]{pic cs:X8'})  edge[line width=0.2mm, out=35,in=125,->] ([yshift=6pt]{pic cs:Z8});
	\draw ([yshift=8pt]{pic cs:U8})  edge[line width=0.2mm, out=35,in=125,->] ([yshift=8pt]{pic cs:Z8});
  \end{tikzpicture}
\end{center}

Here we are assuming that $dom(\mathcal F_T(Y))=dom(\mathcal F_T(Z))= \emptyset$; ranges might be given, for example, by $\mathcal R_T(U) = \mathcal R_T(X) = \mathcal R_T(Y) = \mathcal R_T(Z) = \{1,2,3,4\}$.
 In order to evaluate $X=1 \cf Y=2$ we need to generate the causal team $T_{X=1}$. First we intervene on $X$; this will affect all descendants of $X$, which in this case are the (children) $Y$ and $Z$. 

\begin{center}
\begin{tabular}{|c|c|c|c|}
\hline
 \multicolumn{4}{|l|}{ } \\
 \multicolumn{4}{|l|}{U\tikzmark{U9} \, \tikzmark{X9}X\tikzmark{X9'} \ \ \  \tikzmark{Y9}Y\tikzmark{Y9'} \ \ \ \ \tikzmark{Z9}Z} \\
\hline
 $2$ & $\mathbf{1}$ & $\dots$ & $\dots$\\
\hline
 $3$ & $\mathbf{1}$ & $\dots$ & $\dots$\\
\hline
 $1$ & $\mathbf{1}$ & $\dots$ & $\dots$\\
\hline
 $1$ & $\mathbf{1}$ & $\dots$ & $\dots$\\
\hline
\end{tabular}
 \begin{tikzpicture}[overlay, remember picture, yshift=.25\baselineskip, shorten >=.5pt, shorten <=.5pt]
  \draw [->] ([yshift=3pt]{pic cs:X9'})  [line width=0.2mm] to ([yshift=3pt]{pic cs:Y9});
	\draw [->] ([yshift=3pt]{pic cs:Y9'})  [line width=0.2mm] to ([yshift=3pt]{pic cs:Z9});
  \draw ([xshift=1pt, yshift=7pt]{pic cs:X9'})  edge[line width=0.2mm, out=25,in=135,->] ([yshift=6pt]{pic cs:Z9});
	\draw ([yshift=8pt]{pic cs:U9})  edge[line width=0.2mm, out=20,in=135,->] ([yshift=8pt]{pic cs:Z9});
  \end{tikzpicture}
\end{center}

Notice, however, that we cannot yet evaluate $Z$ (which is a function of $U,X$ and $Y$) unless we first update $Y$:

\begin{center}
\begin{tabular}{|c|c|c|c|}
\hline
 \multicolumn{4}{|l|}{ } \\
 \multicolumn{4}{|l|}{U\tikzmark{U10} \, \tikzmark{X10}X\tikzmark{X10'} \ \   \tikzmark{Y10}Y\tikzmark{Y10'} \ \,   \tikzmark{Z10}Z} \\
\hline
 $2$ & $1$ & $\mathbf{2}$ & $\dots$\\
\hline
 $3$ & $1$ & $\mathbf{2}$ & $\dots$\\
\hline
 $1$ & $1$ & $\mathbf{2}$ & $\dots$\\
\hline
\end{tabular}
 \begin{tikzpicture}[overlay, remember picture, yshift=.25\baselineskip, shorten >=.5pt, shorten <=.5pt]
  \draw [->] ([yshift=3pt]{pic cs:X10'})  [line width=0.2mm] to ([yshift=3pt]{pic cs:Y10});
	\draw [->] ([yshift=3pt]{pic cs:Y10'})  [line width=0.2mm] to ([yshift=3pt]{pic cs:Z10});
  \draw ([xshift = 1pt,yshift=7pt]{pic cs:X10'})  edge[line width=0.2mm, out=25,in=135,->] ([yshift=6pt]{pic cs:Z10});
	\draw ([yshift=8pt]{pic cs:U10})  edge[line width=0.2mm, out=20,in=135,->] ([yshift=8pt]{pic cs:Z10});
  \end{tikzpicture}
\end{center}

Finally we update $Z$, but we have a surprise:

\begin{center}
\begin{tabular}{|c|c|c|c|}
\hline
 \multicolumn{4}{|l|}{ } \\
 \multicolumn{4}{|l|}{U\tikzmark{U11} \, \tikzmark{X11}X\tikzmark{X11'} \ \   \tikzmark{Y11}Y\tikzmark{Y11'} \, \; \; \;  \tikzmark{Z11}Z} \\
\hline
 $2$ & $1$ & $2$ & $\mathbf{4}$\\
\hline
 $3$ & $1$ & $2$ & $\mathbf{4}$\\
\hline
 $1$ & $1$ & $2$ & $\hat f_Z(1,1,2)$\\
\hline
\end{tabular}
 \begin{tikzpicture}[overlay, remember picture, yshift=.25\baselineskip, shorten >=.5pt, shorten <=.5pt]
  \draw [->] ([yshift=3pt]{pic cs:X11'})  [line width=0.2mm] to ([yshift=3pt]{pic cs:Y11});
	\draw [->] ([yshift=3pt]{pic cs:Y11'})  [line width=0.2mm] to ([yshift=3pt]{pic cs:Z11});
  \draw ([xshift = 1pt, yshift=7pt]{pic cs:X11'})  edge[line width=0.2mm, out=25,in=150,->] ([yshift=6pt]{pic cs:Z11});
	\draw ([yshift=8pt]{pic cs:U11})  edge[line width=0.2mm, out=20,in=140,->] ([yshift=8pt]{pic cs:Z11});
  \end{tikzpicture}
\end{center}

The information contained in the team $T^-$ is insufficient for the evaluation of the $Z$-value of the last row of $T_{X=1}$. Since the triple $(1,1,2)$ is not in the domain of $\mathcal F_T(Z)$, 
 the best we could do is to fill that part of the table with a formal term.  Here we wrote $\hat f_Z$ as a formal symbol distinguished from the function $f_Z$.
In general, in case of repeated interventions, we can expect also complex terms, with incapsulated function symbols, to be produced. 
Notice however that we have no uncertainties about the $Y$ column; so, it is natural to state that $T_{X=1}\models Y=2$, and that, therefore, $T\models X=1 \cf Y=2$. 
What if we were instead trying to evaluate some statement about $Z$ under these same counterfactual circumstances? We might renounce completely to assign truth values to such statements; or we might perhaps add a further truth value, which holds of such statement if the statement is \emph{admissible} given the form of the terms involved. A precise semantical definition is not trivial, and we address it later.
\end{example}

\begin{example}
Consider now a causal team $S$ which is identical to that which was considered in the previous example, except for that its functional component, instead of being empty, contains at least a partial description of the invariant function for $Z$; we assume that $\mathcal{F}_S(Z) = f_Z$ is such that $(1,1,2)\in dom(f_Z)$, and that $f_Z(1,1,2)=3$. If we now apply the intervention $do(X=1)$ to this causal team, we can use this extra information to explicitly evaluate the last entry of $S_{X=1}$, obtaining the causal team

\begin{center}
$S_{T=1}$: \begin{tabular}{|c|c|c|c|}
\hline
 \multicolumn{4}{|l|}{ } \\
 \multicolumn{4}{|l|}{U\tikzmark{U11BIS} \, \tikzmark{X11BIS}X\tikzmark{X11BIS'} \ \   \tikzmark{Y11BIS}Y\tikzmark{Y11BIS'} \,   \tikzmark{Z11BIS}Z} \\
\hline
 $2$ & $1$ & $2$ & $4$\\
\hline
 $3$ & $1$ & $2$ & $4$\\
\hline
 $1$ & $1$ & $2$ & $\mathbf{3}$\\
\hline
\end{tabular}
 \begin{tikzpicture}[overlay, remember picture, yshift=.25\baselineskip, shorten >=.5pt, shorten <=.5pt]
  \draw [->] ([yshift=3pt]{pic cs:X11BIS'})  [line width=0.2mm] to ([yshift=3pt]{pic cs:Y11BIS});
	\draw [->] ([yshift=3pt]{pic cs:Y11BIS'})  [line width=0.2mm] to ([yshift=3pt]{pic cs:Z11BIS});
  \draw ([xshift = 1pt, yshift=7pt]{pic cs:X11BIS'})  edge[line width=0.2mm, out=25,in=150,->] ([yshift=6pt]{pic cs:Z11BIS});
	\draw ([yshift=8pt]{pic cs:U11BIS})  edge[line width=0.2mm, out=20,in=140,->] ([yshift=8pt]{pic cs:Z11BIS});
  \end{tikzpicture}
\end{center}

\end{example}

\subsection{Formal definition of interventions and counterfactuals} \label{SUBSDEFINT}

All the examples in the previous subsection have one aspect in common: the graphs underlying causal teams are acyclic. In this case the corresponding causal team is called \emph{recursive}, by analogy with the literature on structural equation models; the examples show that, for these kinds of causal teams, the notion of intervention is naturally conceived in algorithmic terms. We now move towards a precise definition.


We have learned a number of lessons from the previous examples:
\begin{enumerate}[A.]
\item An intervention $do(X=x)$ amounts to 1) setting the whole $X$-column to $x$; 2) eliminating all arrows that enter into $X$; 3) updating the columns that correspond to descendants of $X$.
\item It might not be possible to update all the descendants of $X$ in a single step (actually, the order of updating might not be trivial to decide).
\item The information encoded in the causal team might be insufficient for generating, under intervention,  a proper causal team; we must then admit teams which assign formal terms to some variables.
\end{enumerate}

We begin by addressing this last problem. Given a graph $G$ whose set of vertices $V$ is a set of variables, we call $L_G$ the set of function symbols $\hat f_X$ (of arity $card(PA_X)$), for each $X\in V$ (actually, we only need one symbol for each endogenous variable, that is, for variables which have indegree at least one). We call $G$-terms the terms generated from variables in $V$ and from symbols in $L_G$ by the obvious inductive rules; the set of $G$-terms will be denoted as $Term_G$. When speaking of a causal team $T$, we will implicitly assume, from now on, that the ranges of all variables contain the set of terms $Term_{G(T)}$. Actually, it is not difficult to prove that (iterated) interventions (as defined below) on a recursive causal team with finite variable ranges  can generate only a finite number of formal terms. Therefore, in principle a finite causal team could always be extended to a \emph{finite} causal team with formal terms, by using an appropriate finite subset of $Term_{G(T)}$ to extend the ranges of variables. 

We may combine these ideas with those of subsection \ref{SUBSEXPTEAM}. Given a causal team $T$, let $T'$ be the causal team obtained by extending the ranges of variables with formal terms, as described above. Form then the corresponding explicit team $T^{FE}:= (T')^E$ following the construction given in subsection \ref{SUBSEXPTEAM}. We call $T^{FE}$ the \textbf{fully explicit} causal team corresponding to $T$. The invariant functions of $T^{FE}$ now encode sufficient information for performing any kind of intervention that we will consider in this paper. The fact that the ranges of variables contain formal terms solves issue C. The fact that the team is explicit ensures that no information is lost after an intervention.

Now we address problem $B$. How do we decide the order of updating for the columns? To understand the problem, observe the graph in the figure:

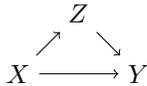
\begin{figure}[htbp]
	\centering
		\begin{tikzpicture}
		
		\node(X) {$X$};
		
		\node(Y) at (1.6,0) {$Y$};
		
		\node(Z)  at (0.8, 0.8) {$Z$}; 
		
		\draw[->] (X)--(Y) ; 
		
		\draw[->] (X)--(Z); 
		
		\draw[->] (Z)--(Y); 
		
		\end{tikzpicture}
  \label{fig:BEXAMPLE}
\end{figure}

$Y$ is connected to $X$ by an arrow, so one might think that, in an intervention on $X$, it is possible to update $Y$ immediately after updating $X$. However, this is impossible, because also the updated values of $Z$ are needed in the evaluation of the new values for $Y$. The point is that, from $X$ to $Y$, there is a longer path than the direct one. This suggests defining a distance from $X$, $d(X,\cdot)$, given by the maximum length of paths from $X$ to $\cdot$. Nodes at distance 1 can be immediately evaluated after updating $X$; at this point it is safe to update nodes at distance 2; and so on. Nodes that are not accessible by directed paths from $X$ will be assigned a negative distance and excluded from the updating procedure. Of course, this strategy will work provided there are no loops in the graph, or at least no loops in the part of the graph which is accessible by directed paths from $X$ (the set of descendants of $X$). Since in the applications it is very common to have conjunctive interventions, say $do(X_1=x_1 \land\dots \land X_n=x_n)$, we will define, more generally, the notion of distance from a \emph{set} of variables $\mathbf{X}$. In this more complex case, one must consider a reduced graph in which all arrows entering in $\SET{X}$ have been removed, and examine the directed paths of this reduced graph. The reader may think of the $card(\mathbf{X})=1$ case for ease of visualization. 


\begin{df}
Given a graph $G=(\SET V,E)$ and $\SET{X}\subseteq \SET V$, 

\begin{itemize}
\item We denote as $G_{-\SET{X}}=(\SET V,E_{-\SET{X}})$ the graph obtained by removing all arrows going into some vertex of $\SET X$ (i.e., an edge $(V_1,V_2)$ is in $E_{-\SET{X}}$ iff it is in $E$ and $V_2\notin \SET X$). Notice that, in the special case that $\SET{X} = \{X\}$, the set of directed paths of $G_{-\SET{X}}$ starting from $X$ coincides with the set of directed paths of $G$ starting from $X$.

\item Let $Y\in \SET V$. We call \textbf{(evaluation) distance} between $\mathbf{X}$ and $Y$ the value $d_G(\mathbf{X},Y)= sup\{lenght(P)| P \text{  dir. path of $G_{-\SET{X}}$ going from some  $X\in \mathbf{X}$}$ \\
$\text{to } Y\}$. In case no such path exists, we set $d_G(\mathbf{X},Y)=-1$. Clearly, if the graph is finite and acyclic, $d_G(\mathbf{X},Y)\in \mathbb{N}\cup\{-1\}$ for any pair $\mathbf{X},Y$. When the graph is clear from the context, we simply write $d(\SET X,Y)$.
\end{itemize}
\end{df}

Let $T$ be a \emph{fully explicit}, \emph{recursive} causal team of endogenous variables $\SET V$. Let $\mathbf{X} =\{X_1,\dots,X_n\}\subseteq dom(T)$ and $x_1\in Ran(X_1),\dots, x_n\in Ran(X_n)$ corresponding values with the additional property that, if $X_i$ and $X_j$ denote the same variable, then $x_i=x_j$. We define the \textbf{intervention} $do(\bigwedge_{i=1..n} X_i=x_i)$ (in short, $do(\SET{X}=\SET{x})$) as an algorithm\footnote{Notice that our definition allows the presence of many occurrences of the same variable in different conjuncts, provided that all such conjuncts assign the same value to the given variable. We do \emph{not} define interventions associated to contradictory conjunctions. E.g., the intervention $do(X=x \land X=x \land Y=y)$ is defined, and coincides with the intervention $do(X=x \land Y=y)$; while the intervention $do(X=1 \land X=2)$ is not defined.}: \\ 

Stage $0$. Delete all arrows coming into $\SET X$, and replace each assignment $s\in T$ with $s(\SET{x}/\SET{X})$. Denote the resulting team\footnote{A warning: the teams $T_n$ produced before the last step of the algorithm may fail to form a causal team together with the other coomponents described, because of violations of conditions b) and c) of the definition of causal team.} as $T_0$. Replace $\mathcal F_T$ with its restriction $\mathcal F_T'$ to $\SET V\setminus\SET X$. \\

Stage $n+1$. If $\{Z_1,\dots, Z_{k_{n+1}}\}$ is the set of all the variables $Z_j$ such that $d_{G(T)}(\SET{X},Z_j) = n+1$, define a new team $T_{n+1}$ by replacing each $s\in T_n$ with the assignment $s(f_{Z_1}(s(PA_{Z_1}))/Z_1, \dots, f_{Z_{k_{n+1}}}(s(PA_{Z_{k_{n+1}}}))/Z_{k_{n+1}})$. \\ 

End the procedure after step $\hat n = sup\{d_{G(T)}(\SET{X},Z)|Z\in dom(T)\}$.\\

Notice that $\mathcal R_T$ is not modified by the algorithm; and that, except for the modifications to $G(T)$, 
 it would be the same to apply the algorithm to each assignment separately. 
In case the causal team $T$ is recursive but not fully explicit, we should begin the algorithm with an additional step:\\

Step $-1$. Replace $T$ with the corresponding fully explicit causal team $T^{FE}$.\\

In case the intervention $do(\SET{X}=\SET{x})$ is a terminating algorithm on $T$, we define the causal team $T_{\SET{X}=\SET{x}}$ (of endogenous variables $\SET V\setminus \SET X$) as the quadruple $(T^{\hat n},G(T)_X,\mathcal R_T, \mathcal F_T')$ which is produced when $do(\SET{X}=\SET{x})$ is applied to $T$. Actually, relaxing a bit the notion of algorithm, the $do$ algorithm applies as well to causal teams with infinite variable ranges:

\begin{teo}
If $G(T)$ is a finite acyclic graph, then $T_{\SET{X}=\SET{x}}$ is well-defined.
\end{teo}

\begin{proof}
We can assume without loss of generality that $T$ is fully explicit.

Assume also, at first, that $T^-$ is finite.

Suppose that, for $\SET{X}\subseteq dom(T), Y\in dom(T)$, $d(\SET X,Y)> card(G(T))$. Then there is a path $P$ from some $X\in\mathbf{X}$ to $Y$ such that $lenght(P)>card(G(T))$. Then there is a node which is crossed at least twice by $P$; so $G(T)$ contains a cycle: contradiction. 

Therefore, $sup\{d(\mathbf{X},Z)|Z\in dom(T)\}\leq card(G(T))$; this means that the ``for'' cycle in the algorithm goes through a finite number of iterations over the variable $n$.

Finally, notice that, for each $n$, there are a finite number of variables $Z$ such that $d(\SET{X},Z)=n$ (due to finiteness of $G(T)$) and a finite number of assignments $t$ in the team that is undergoing modification (due to the finiteness of $T^-$). Therefore, the algorithm terminates after a finite number of steps.

If instead $T^-$ is infinite, we can replace $do(\SET{X}=\SET{x})$ with an infinitary algorithm which, in each iteration of the ``for'' cycle, performs simultaneously the substitution $t \hookrightarrow t(f_{Z_1}(s(PA_{Z_1}))/Z_1, \dots, f_{Z_{k_{n+1}}}(s(PA_{Z_{k_{n+1}}}))/Z_{k_{n+1}})$ for all the assignments $t$ in the current team. By the arguments above, such an ``algorithm'' terminates and yields a well-defined causal team.
\end{proof}

In case a causal team is not recursive (i.e., its graph is cyclic), the algorithm above may well fail to terminate. In this case, a definition of the intervention $do(\SET{X}=\SET{x})$ could still be given in terms of the set of solutions of the modified system of structural equations (in which the equations for $\SET{X}$ are replaced by $\SET{X}:=\SET{x}$), provided the team is parametric. Each assignment $s\in T$ is a solution of the system of structural equations encoded in $\mathcal F_T$ (this is ensured by parametricity and conditions a) and c) in the definition of causal team). Galles and Pearl (\cite{GalPea1998}) consider the case of systems with unique solutions, defined as follows: 1) for fixed values of the exogenous variables, the system has a unique solution, and 2) each ``intervened'' system of equations obtained from the initial one replacing some equations of the form $X:=f_X(PA_X)$ with constant equations $X:=x$ still has a unique solution for each choice of values for the exogenous variables. In case the $\mathcal F_T$ component of a team encodes a system with unique solutions, then the natural way to define an intervention $do(X=x)$ on the team is to replace each assignment $s\in T^-$ with the (unique) assignment $t$ which encodes the solution of the intervened system for the choice $s(\SET{U})$ of values to the exogenous variables\footnote{In case the intervention acts also on some of the exogenous variables, this idea should be modified in an obvious way.}.  The definition of the other components of the causal team produced by the intervention is straightforward. Over recursive parametric causal teams the team so obtained coincides with the causal team that is produced by the algorithm. 

It is not difficult to extend the previous ideas to systems that may have no solution at all. In case the intervened system obtained replacing the appropriate equations with $\SET{X}=\SET{x}$ has no solution corresponding to the choice $s(\SET{U})$ for the exogenous variables, then no assignment should correspond to $s$ in the intervened team $T_{\SET X=\SET x}$. Although this extension is straightforward, we can expect significant changes in the underlying logic.

 The case of systems with multiple solutions is more problematic. If many solutions correspond to an assignment of the initial team, should we include them all in the intervened team? And, for what regards the probabilistic developments of the following sections, should we consider all the assignments thus produced as equiprobable? These matters should probably be settled according to the kind of interpretation we want to give to nonrecursive causal teams in a given application, and one is lead to the classical problems of interpretation of nonrecursive causal models (see e.g \cite{StrWol1960}). It might also be reasonable to model such an intervention as producing not one, but multiple teams, corresponding to possible different outcomes of the intervention. This set of ``accessible teams'' would then induce a nontrivial modality, and it would then be reasonable to treat counterfactuals as necessity operators in a dynamic logic setting (in the spirit of \cite{Hal2000}). This would lead us far away from the approach of the present paper, in which we focus on the nonproblematic recursive case. We now return to it.

Having defined the intervened team $T_{\SET X=\SET x}$, we are immediately led to a semantical clause for counterfactuals of the form $\SET X=\SET x \cf \psi$:

\[
T\models \SET X=\SET x \cf \psi \iff T_{\SET X=\SET x} \models \psi.
\]
In case the antecedent is inconsistent (i.e., it contains two conjuncts $X_i = x_i, X_i = x_i'$ with $x_i \neq x_i'$), the corresponding intervention is not defined; in this case, we postulate the counterfactual to be (trivially) true.

\subsection{Logical languages}

We call the (basic) \emph{language of causal dependence}, $\mathcal{CD}$, the language formed by the following rules:
\[
Y = y \ | \ Y\neq y \ | \ \dep{\SET X}{Y} \ | \ \psi \land \chi \ | \ \psi \land \chi \ | \ \theta \supset \chi \ | \ \SET X= \SET x \cf \chi
\]
for $Y,\SET{X}$ variables, $y,\SET x$ values, $\psi,\chi$ formulae of $\CD$, $\theta$ classical formula. 

If we interpret this language by the semantical clauses introduced so far, it can be shown that:

\begin{teo}
The logic $\mathcal{CD}$ is downwards closed, that is: if $\varphi\in \mathcal{CD}$, $T$ is a parametric causal team with at most unique solutions, $T'$ is a causal subteam of $T$, and $T\models\varphi$, then also $T'\models\varphi$.
\end{teo}

\begin{proof}
We prove it by induction on the synctactical structure of $\varphi$. The atomic and propositional cases are routine.

Suppose $T\models \psi \supset \chi$. Then $T^{\psi}\models \chi$. Now $T'^{\psi}\subseteq T^{\psi}$; therefore, by inductive hypothesis, $T'^{\psi}\models \chi$. Thus, by the semantic clause for selective implication, $T'\models  \psi \supset \chi$.

Suppose $T\models \SET X= \SET x \cf \chi$. Then $T_{\SET X= \SET x}\models \chi$. But since the algorithm $do(\SET X= \SET x)$ acts on each assignment separately, we can conclude that $(T'_{\SET X= \SET x})^-\subseteq (T_{\SET X= \SET x})^-$. So, by the inductive hypothesis, $T'_{\SET X= \SET x}\models \chi$. Applying again the semantic clause, $T'\models \SET X= \SET x \cf \chi$.
\end{proof}

\begin{teo}
The logic $\mathcal{CD}$ has the empty set property, that is: for every $\varphi\in \mathcal{CD}$, $\emptyset\models\varphi$.
\end{teo}

\begin{proof}
By induction on the syntax of the formula. The new cases are those for $\supset$ and $\cf$.

Let $\varphi$ be $\psi\supset\chi$. Notice that $\emptyset^{\psi}= \emptyset$. Trivially, any assignment in it satisfies $\chi$. So $\emptyset\models \psi \supset \chi$.

Let $\varphi$ be of the form $\SET{X} = \SET{x}\cf \chi$. Any $do$ algorithm, applied to $\emptyset$, produces $\emptyset$ again; therefore, by the inductive hypothesis, $\emptyset_{\SET X= \SET x}= \emptyset\models \chi$. Thus, by the clause for counterfactuals we obtain $\emptyset\models \SET X= \SET x\cf \chi$.
 \end{proof}

We also give names to some relevant fragments of $\CD$:
\begin{itemize}
\item $\C$: the \emph{counterfactual fragment}, where symbols $\dep{\cdot}{\cdot}$ and $\supset$ are not allowed

\item $\CO$: the \emph{causal-observational fragment} where $\dep{\cdot}{\cdot}$ is not allowed.
\end{itemize}
We may also define variants $\CO^{neg},\C^{neg}$ that are closed (except for antecedents of counterfactuals) under a dual negation $\neg$, which satisfies the additional semantical clause:
\begin{itemize}
\item $T\models \neg\psi \iff $ for all $s\in T^-$, $\{s\}\not \models \psi$.\footnote{This definition is reasonable due to the flatness of $\CO$, which is entailed by theorem \ref{TEOFLAT} below.} 
\end{itemize}


We can show that the language $\CO^{neg}$ (and thus also $\CO,\C^{neg},\C$) satisfies a much more restrictive condition: flatness. Here and in the following, we will often abuse notation and write, say, $\{s\}$ for a causal team $T$ whose support $T^-$ contains the single assignment $s$; we will then write $\{s\}_{\SET X= \SET x}$ for the causal team $T_{\SET X= \SET x}$, and so on.

\begin{teo} \label{TEOFLAT}
The logic $\CO^{neg}$ over unique-solution parametric causal teams is flat, that is: for every formula $\varphi$ of $\CO^{neg}$ and every unique-solution parametric causal team $T$, $T\models \varphi$ iff $\{s\}\models \varphi$ for every assignment $s\in T^-$.
\end{teo}

\begin{proof}
It can be shown by induction on the complexity of the involved formulas. The cases related to propositional connectives are routine. We only have to check the cases of $\varphi = \SET X= \SET x\cf \chi$ and $\varphi = \psi \supset \chi$.

Case $\cf$) Suppose first that $T\models \SET X= \SET x\cf \chi$, with $\chi$ flat. Then by the semantical clause for counterfactuals, $T_{\SET X= \SET x}\models \chi$. By flatness of $\chi$, $\{t\}\models\chi$ for each $t\in (T_{\SET X= \SET x})^-$. Now, notice that $\{s\}$ is a unique-solution causal team (because $T$ is such), which entails that, for every $s\in T^-$, $\{s\}_{\SET X= \SET x}$ is again a singleton causal team (see the details of the $do$ algorithm). 
 Therefore, by the remarks above, $\{s\}_{\SET X= \SET x}\models \chi$ for each $s\in T$; so, by the semantical clauses, $\{s\}\models \SET X= \SET x\cf \chi$ for every $s\in T$.

In the other direction, suppose $\{s\}\models \SET X= \SET x\cf \chi$ for every $s\in T^-$. Then for every $s\in T^-$, $\{s\}_{\SET X= \SET x}\models \chi$. Again, any such causal team is a singleton set, say $\{s\}_{\SET X= \SET x} =: \{t_s\}$. But notice that, if $t$ is any assignment in $(T_{\SET X= \SET x})^-$, then $t=t_s$ for some $s\in T^-$. So we can apply the inductive hypothesis and conclude that $T_{\SET X= \SET x}\models \chi$. Therefore 
$T\models \SET X= \SET x\cf \chi$.

Case $\supset$) Suppose $T\models \psi \supset \chi$, with $\psi, \chi$ flat.  Then, by the semantical clause for $\supset$, $T^\psi\models \chi$. By induction hypothesis, $\{s\}^\psi=\{s\}\models \chi$ for each $s\in (T^\psi)^-$. 
 So $\{s\}\models \psi\supset \chi$. Suppose instead $s'\in T$ is such that $\{s'\}\not\models\psi$. Then $\{s'\}^\psi = \emptyset$. Now $\{s'\}^\psi = \emptyset \models \chi$; so, by the clause for $\supset$, $\{s'\}\models \psi \supset \chi$. In conclusion, for any $s\in T^-$,  $\{s\}\models \psi \supset \chi$.

Viceversa, suppose $\{s\}\models \psi \supset \chi$ for all $s\in T^-$. Let $s\in T^-$ be such that $\{s\}\models \psi$. Then $\{s\}^\psi = \{s\}$; our assumption then yields $\{s\}\models \chi$. Since $\chi$ is flat, we can conclude that $T^\psi \models \chi$. So $T\models \psi \supset\chi$. 
\end{proof}


This flatness theorem uses in an essential way the recursiveness of the causal team. The previous downward closure and empty set property theorems may work as well in the nonrecursive case, provided a definition of intervention is given also in that case, along the lines sketched at the end of subsection \ref{SUBSDEFINT}.

Notice also that, unlike in similar results for team semantics (see e.g. \cite{Vaa2007}), the truth of a flat formula in a causal team based on a singleton team uses, through structural equations, information that goes beyond that which is encoded in the single assignment of the team. A singleton causal team can be identified with a structural equation model, enriched with a description of an actual state for all variables in the system; therefore, our flatness result can be seen as expressing a sort of conservativity of the causal team semantics over structural equation modeling. As long as one works within a sufficiently poor language, statements on causal teams can be reduced to statements over multiple enriched structural equation models; but as soon as further non-flat language components are introduced, such as dependence atoms, or the probabilities and the intuitionistic disjunction (see later sections), structural equation models become insufficient, and causal teams are needed.

\section{Some logical principles} \label{LOGLAWS}

In this section, we analyze some validities and inference rules for the languages $\CD,\CO$ and $\C$ when their semantics ranges over \emph{parametric, recursive} causal teams (we will assume this restriction throughout the section). Most observations extend also to unique solution causal teams.

We shall write $\psi \models \chi$ to mean that $\chi$ holds on any parametric, recursive causal team on which $\psi$ holds. By $\models\chi$ we mean that $\chi$ holds on any parametric, recursive causal team.

\subsection{The law of excluded middle} \label{SUBSEM}

Does the fact that the logic $\mathcal{CO}$ is flat entail that its propositional logic is classical? The issue is somewhat problematic.

It is well known that in predicate logics of dependence and independence the law of  excluded middle fails; in the propositional case this point is more subtle. We propose two formulations of the law of excluded middle, a weak one and a stronger one that comes closer to the classical version:\\
\[
\hspace{-108pt}(WEM): \hspace{15pt}\models \psi \lor \neg\psi 
\]
\[
(SEM): \hspace{15pt}\text{For any team } T, T\models \psi \text{ or } T\models \neg \psi
\] 
Neither formulation covers the full classical scheme, because for many formulas we lack negation. Within $\CO^{neg}$, instead, SEM fundamentally corresponds to the classical law of excluded middle.


It is very easy to find a counterexample to SEM, even at the atomic level. Notice indeed that the following team 
\begin{center}
\begin{tabular}{|c|}
\hline
 X  \\
\hline
  1 \\
\hline
  2 \\
\hline
\end{tabular}
\end{center}
does not satisfy $X=1$ nor its negation $X\neq 1$.

With regard to WEM, 
we get:

\begin{teo}
All formulas in $\CO^{neg}$ satisfy WEM.
\end{teo}

\begin{proof}
Let $\varphi$ be a formula of $\CO^{neg}$. Let $T$ be a team. Let $T_1 = \{s\in T^-| \{s\}\models\varphi\}$, and $T_2 = \{s\in T^-| \{s\}\not\models\varphi\} = \{s\in T^-| \{s\}\models\neg\varphi\}$. Then, by flatness, $T_1\models \varphi$ and $T_2\models \neg \varphi$. By the clause for disjunction, $T \models \varphi \lor \neg \varphi$.
\end{proof}

For trivial reasons (since it lacks negation of all but atomic formulas), $\CD$ itself satisfies WEM. It is not obvious how to extend $\CD$ itself with a negation. One way followed in the literature is to treat the negation of a dependence atom as satisfied only by the empty team. In this case, one immediately finds counterexamples to WEM.

Another possibility would be to extend $\CD$ with a contradictory negation of the dependence atom:
\begin{itemize}
\item  $T\models \not\dep{\SET{X}}{Y} \iff$ there are $s,s'\in T$ such that $s(\SET{X})=s'(\SET{ X})$ but $s(Y) \neq s'(Y)$.
\end{itemize}
 This option does not save WEM, either: just let $T$ be any team that satisfies the atom $\dep{\SET{X}}{Y}$. Then, in any partition $T_1, T_2$ of $T$, both parts satisfy the atom, and not its negation (also in case one of the two subteams is the empty subteam). Notice also the failure of the empty set property.

\subsection{Conditional excluded middle}

In the present section we consider the law that is commonly called \emph{conditional excluded middle}: 

\[
\hspace{-86pt}(CEM): \hspace{15pt}\models \psi \cf (\chi \lor \neg \chi).
\]
This is often considered in pair with the formally similar law called \emph{distribution}:
\[
(D): \hspace{15pt}\psi \cf (\chi \lor \chi') \models [(\psi \cf \chi) \lor (\psi \cf \chi')]
\] 

These principles are known to hold for Stalnaker counterfactuals and fail for Lewis counterfactuals. We will show in a later subsection (\ref{SUBSCOMPL}) that distribution is valid for our languages. Here we focus mainly on the CEM law.

What about systems of interventionist counterfactuals, such as those given by Galles and Pearl (\cite{GalPea1998}), Halpern (\cite{Hal2000}) and Briggs (\cite{Bri2012})? In the system considered by Galles and Pearl, only atomic consequents are allowed; therefore CEM and D have no instances. For what regards the remaining systems, these fundamentally evaluate formulas on single assignments, that we might think as singleton causal teams. Each intervention produces again an assignment/singleton team. Therefore, CEM trivially holds for atomic consequents. Simple induction arguments allow to extend this observation to all the formulas of the language considered by Halpern (which allow boolean operators in the consequents of counterfactuals). Briggs also allows disjunctions and negations in the antecedents of counterfactuals; we do not consider this possibility in the paper. The languages of these papers also have an operator corresponding to classical implication; therefore, D can be internalized and shown to hold in there, as is done explicitly in \cite{Hal2000}.

As far as causal teams are concerned, by the flatness of $\CO^{neg}$, CEM holds in $\CO^{neg}$, and therefore also in $\CD$. However, it is natural to wonder about a corresponding semantical property:
\[
(CEM'): \text{ For every causal team $T$, } T\models \theta\cf\chi \text{ or } T\models \theta\cf\neg\chi
\]

We show that CEM'  can fail in very simple examples on causal teams; therefore, our logic does \emph{not} coincide with that of usual interventionist counterfactuals. The simplest possible counterexample is of the form:
\begin{center}
$T$:
\begin{tabular}{|c|c|}
\hline
 \multicolumn{2}{|c|}{X \ Y} \\
\hline
 1 & 1 \\
\hline
 1 & 2 \\
\hline
\end{tabular}
\end{center}

The intervention $do(X=1)$ does not modify this team, that is: $T_{X=1} = T$. On the other hand, $T\not\models Y=1$, and $T\not\models Y\neq 1$. So, $T\not\models X=1 \cf Y=1$, and $T\not\models X=1 \cf Y\neq 1$. Thus, CEM' is falsified.



The peculiarity of this example might suggest that the failure of these laws could be attributed to the fact that there are no arrows between the variables of concern. This is not the case, as we can show with a slightly more complex example. Consider the team
\begin{center}
$S$: 
\begin{tabular}{|c|c|c|}
\hline
 \multicolumn{3}{|c|}{ }\\ 
 \multicolumn{3}{|c|}{X\tikzmark{X18} \  \tikzmark{Z18}Z\tikzmark{Z18'} \  \tikzmark{Y18}Y} \\
\hline
 1 & 1 & 2\\
\hline
2 & 2 & 4\\
\hline
\end{tabular}
 \begin{tikzpicture}[overlay, remember picture, yshift=.25\baselineskip, shorten >=.5pt, shorten <=.5pt]
	\draw [->] ([yshift=3pt]{pic cs:Z18'})  [line width=0.2mm] to ([yshift=3pt]{pic cs:Y18});
  \draw ([yshift=7pt]{pic cs:X18})  edge[line width=0.2mm, out=55,in=125,->] ([yshift=7pt]{pic cs:Y18});
  \end{tikzpicture}
\end{center}

where the behavior of $Y$ is determined by the function $\mathcal{F}_S(Y)(X,Z):= X+Z$. The intervention $do(X=1)$ generates the team 
\begin{center}
$S_{X=1}$: 
\begin{tabular}{|c|c|c|}
\hline
 \multicolumn{3}{|c|}{ }\\ 
 \multicolumn{3}{|c|}{X\tikzmark{X19} \  \tikzmark{Z19}Z\tikzmark{Z19'} \  \tikzmark{Y19}Y} \\
\hline
 1 & 1 & 2\\
\hline
1 & 2 & 3\\
\hline
\end{tabular}
 \begin{tikzpicture}[overlay, remember picture, yshift=.25\baselineskip, shorten >=.5pt, shorten <=.5pt]
	\draw [->] ([yshift=3pt]{pic cs:Z19'})  [line width=0.2mm] to ([yshift=3pt]{pic cs:Y19});
  \draw ([yshift=7pt]{pic cs:X19})  edge[line width=0.2mm, out=55,in=125,->] ([yshift=7pt]{pic cs:Y19});
  \end{tikzpicture}
\end{center}
It is then immediate to verify that $S$ neither satisfies  $X=1 \cf Y=2$ nor $X=1 \cf Y\neq 2$, thereby falsifying CEM'. 

\subsection{Interactions between implications and dependencies} \label{SUBSINTER}

As causal team semantics defines both dependencies based on correlations (functional dependence) and causal and counterfactual dependencies, it allows us to investigate the logical principles which govern their interaction. Here we list some of the inference rules that connect the various kinds of dependence that are referred to within the language $\CD$; the simple soundness proofs are left to the reader.

Given that the selective implication is strongly analogous to classical implication, the following inference rules hold (assuming the classicality of all antecedents):
\[
\frac{\theta \supset \chi}{(\theta \land \psi) \supset \chi} \hspace{25pt} \frac{\theta \supset \chi}{\theta \supset (\chi \lor\psi)} \hspace{25pt} \frac{\theta \supset (\chi \supset \psi)}{(\theta \land \chi) \supset \psi}
\]
the third of which is also invertible. The second rule can be expected to fail in logics that lack the empty set property.

It is time to make the relationship between selective and classical implication more explicit. First of all, notice that a restricted version of a semantic Deduction Theorem holds for selective implication: if any team that satisfies a \emph{classical} formula $\theta$ also satisfies $\chi$, then this in particular holds for any team of the form $T^\theta$; therefore $T^\theta\models \chi$, from which we obtain $T\models \theta\supset \chi$ (for any causal team $T$). Therefore, all the valid inference rules that are stated in this section can be internalized in the formal language as selective implications, \emph{provided the antecedents are classical formulas}\footnote{We might consider extending our languages by allowing \emph{flat} instead of classical antecedents for selective implications. This would extend the range of application of the Deduction Theorem, but would have the drawback that the syntax of the resulting language might not be decidable anymore.}.

Secondly, we wish to establish whether our approach, of allowing selective implications, is equivalent or not to the approach of adding to our languages the dual negations of all classical formulas; could then a selective implication $\theta\supset\chi$ be replaced by an expression of the form $\neg\theta\lor \chi$?
The answer is yes \emph{for downward closed logics}, as those that we have considered thus far; but the approach via selective implication is more general. The point is this: applying the clause for disjunction, $T\models\neg\theta\lor \chi$ holds if and only if there are $S_1,S_2 \subseteq T^-$ such that $S_2\models \chi$ and, for all $s\in S_1$, $\{s\}\not\models\theta$. Then $(T^{\theta})^- \subseteq S_2$; if the logic under consideration is downward closed, then, we can conclude that $T^{\theta}\models \chi$, that is, $T\models \theta\supset\chi$. But it is easy to find counterexamples in logics that are not downward closed. Consider for example the (marginal) \emph{independence atoms} $X \perp Y$ defined (along the lines of \cite{GraVaa2013}) by the semantical clause:
\begin{itemize}
\item $T\models X \perp Y \iff \text{for every } s,s'\in T^- \text{ there is an } s''\in T^- \text{ s.t. } s''(X) = s(X) \text{ and } s''(Y) = s'(Y)$.
\end{itemize}
Then the following team
\begin{center}
$U$:
\begin{tabular}{|c|c|}
\hline
 \multicolumn{2}{|c|}{X \ Y} \\
\hline
 $1$ & $1$ \\
\hline
 $1$ & $2$ \\
\hline
 $2$ & $1$ \\
\hline
 $2$ & $2$ \\
\hline
\end{tabular}
\end{center}
satisfies $X\neq 1 \lor X \perp Y$ (just consider the partition of $U$ into $\emptyset$ and $U$; $\emptyset\models X\neq 1$ and $U\models X \perp Y$). But $U\not \models X=1 \supset X \perp Y$.


We now turn to the relationship between selective and counterfactual implication.
We start by showing that their commutativity, that is, the equivalence between $\theta \cf (\chi \supset \psi)$ and $\chi \supset (\theta \cf \psi)$ fails in causal teams in both directions. We show this by a couple of counter-examples.
 For all the teams involved, the graph will be $\{\{X,Y,Z\}, \{(X,Y),(Z,Y)\}\}$ (so, we do not draw it every time) and the only structural equation will be $Y:= X+Z$; it is not really important to specify the ranges of variables. Consider the following teams:
\begin{center}
$T$:
\begin{tabular}{|c|c|c|}
\hline
 \multicolumn{3}{|c|}{X \ Z \ Y} \\
\hline
 1 & 1 & 2 \\
\hline
 1 & 2 & 3 \\
\hline
 2 & 3 & 5 \\
\hline
\end{tabular}
\hspace{20pt}
$T_{X=1}$:
\begin{tabular}{|c|c|c|}
\hline
 \multicolumn{3}{|c|}{X \ Z \ Y} \\
\hline
 1 & 1 & 2 \\
\hline
 1 & 2 & 3 \\
\hline
 \textbf{1} & 3 & \textbf{4} \\
\hline
\end{tabular}
\hspace{20pt}
$(T_{X=1})^{X=1}$:
\begin{tabular}{|c|c|c|}
\hline
 \multicolumn{3}{|c|}{X \ Z \ Y} \\
\hline
 1 & 1 & 2 \\
\hline
 1 & 2 & 3 \\
\hline
 1 & 3 & 4 \\
\hline
\end{tabular}

\vspace{10pt}

$T^{X=1}$:
\begin{tabular}{|c|c|c|}
\hline
 \multicolumn{3}{|c|}{X \ Z \ Y} \\
\hline
 1 & 1 & 2 \\
\hline
 1 & 2 & 3 \\
\hline
\end{tabular}
\hspace{20pt}
$(T^{X=1})_{X=1}$:
\begin{tabular}{|c|c|c|}
\hline
 \multicolumn{3}{|c|}{X \ Z \ Y} \\
\hline
 1 & 1 & 2 \\
\hline
 1 & 2 & 3 \\
\hline
\end{tabular}
\end{center}
From the tables, it is immediate to see that $T\models X=1 \supset (X=1\cf (Y=2\lor Y=3))$, but $T\not\models X=1 \cf (X=1\supset (Y=2\lor Y=3))$.


A counterexample for the opposite direction is given by the following teams, with graph and equation as before:

\begin{center}
$S$:
\begin{tabular}{|c|c|c|}
\hline
 \multicolumn{3}{|c|}{X \ Z \ Y} \\
\hline
 1 & 1 & 2 \\
\hline
 1 & 2 & 3 \\
\hline
 2 & 1 & 3 \\
\hline
\end{tabular}
\hspace{20pt}
$S_{Z=1}$:
\begin{tabular}{|c|c|c|}
\hline
 \multicolumn{3}{|c|}{X \ Z \ Y} \\
\hline
 1 & 1 & 2 \\
\hline
 2 & 1 & 3 \\
\hline
\end{tabular}
\hspace{20pt}
$(S_{Z=1})^{Y=3}$:
\begin{tabular}{|c|c|c|}
\hline
 \multicolumn{3}{|c|}{X \ Z \ Y} \\
\hline
  2 & 1 & 3 \\
\hline
\end{tabular}

\vspace{10pt}

$S^{Y=3}$:
\begin{tabular}{|c|c|c|}
\hline
 \multicolumn{3}{|c|}{X \ Z \ Y} \\
\hline
 1 & 2 & 3 \\
\hline
 2 & 1 & 3 \\
\hline
\end{tabular}
\hspace{20pt}$(S^{Y=3})_{Z=1}$:
\begin{tabular}{|c|c|c|}
\hline
 \multicolumn{3}{|c|}{X \ Z \ Y} \\
\hline
 1 & 1 & 2 \\
\hline
 2 & 1 & 3 \\
\hline
\end{tabular}
\end{center}
which show that $S\models Z=1 \cf (Y=3 \supset Y=3)$ but $S\not\models Y=3\supset (Z=1 \cf Y=3)$.

The noncommutativity of $\cf$ with $\supset$ is a phenomenon that, to the best of our knowledge, has no analogous in the usual notations of the causal inference literature; we return to this point in Section \ref{SECDEFPROB}.

 In our search for further rules linking $\cf$ and $\supset$, we can now have a look at the literature on Stalnaker-Lewis counterfactuals, and check whether in some case $\supset$ may take the role that in Stalnaker-Lewis rules is taken by material implication. For example, it is known that 1) a Stalnaker-Lewis counterfactual $\theta\cf\chi$ implies the corresponding material conditional $\theta\rightarrow\chi$, and 2) the converse holds, \emph{provided $\theta$ holds}. It might be interesting to inquire whether similar principles hold for our counterfactual and selective implication, respectively.  The literature on causation seems to suggest a negative answer: looking for example at rule 2 of Pearl's \emph{do calculus} (\cite{Pea2000}, sec. 3.4) we see that, in some languages that also allow the discussion of probabilities, replacing an intervention with an observation, or viceversa, can be done only under some graph-theoretical (and not logical) assumptions. Explicit counterexamples to 1) and 2) can be constructed in our languages, following the lines of \cite{Bri2012} (sect. 3.1 and 4), for some formulas in which the consequent $\chi$ is itself a counterfactual. The counterexample in \cite{Bri2012} is based on the famous ``execution line'' example (\cite{Pea2000}), but we can construct a much simpler one. Consider the following causal teams, with boolean ranges and  invariant functions $f_Y(X):= X$ and $f_Z(X,Y):= X\land Y$:

\begin{center}
$T = T^{Y=0}$: 
\begin{tabular}{|c|c|c|}
\hline
 \multicolumn{3}{|c|}{ }\\ 
 \multicolumn{3}{|c|}{X\tikzmark{XA} \  \tikzmark{YA}Y\tikzmark{YA'} \  \tikzmark{ZA}Z} \\
\hline
 0 & 0 & 0\\
\hline
\end{tabular}
\hspace{10pt}$(T^{Y=0})_{X=1}$:
\begin{tabular}{|c|c|c|}
\hline
 \multicolumn{3}{|c|}{ }\\ 
 \multicolumn{3}{|c|}{X\tikzmark{XB} \  \tikzmark{YB}Y\tikzmark{YB'} \  \tikzmark{ZB}Z} \\
\hline
 1 & 1 & 1\\
\hline
\end{tabular}

\vspace{10pt} $T_{Y=0}$:
\begin{tabular}{|c|c|c|}
\hline
 \multicolumn{3}{|c|}{ }\\ 
 \multicolumn{3}{|c|}{X\tikzmark{XC} \  \tikzmark{YC}Y\tikzmark{YC'} \  \tikzmark{ZC}Z} \\
\hline
 0 & 0 & 0\\
\hline
\end{tabular}
\hspace{10pt}$(T_{Y=0})_{X=1}$
\begin{tabular}{|c|c|c|}
\hline
 \multicolumn{3}{|c|}{ }\\ 
 \multicolumn{3}{|c|}{X\tikzmark{XD} \  \tikzmark{YD}Y\tikzmark{YD'} \  \tikzmark{ZD}Z} \\
\hline
 1 & 0 & 0\\
\hline
\end{tabular}

 \begin{tikzpicture}[overlay, remember picture, yshift=.25\baselineskip, shorten >=.5pt, shorten <=.5pt]
  \draw [->] ([yshift=3pt]{pic cs:XA})  [line width=0.2mm] to ([yshift=3pt]{pic cs:YA});
	\draw [->] ([yshift=3pt]{pic cs:YA'})  [line width=0.2mm] to ([yshift=3pt]{pic cs:ZA});
  \draw ([yshift=7pt]{pic cs:XA})  edge[line width=0.2mm, out=55,in=125,->] ([yshift=7pt]{pic cs:ZA});
    
	 \draw [->] ([yshift=3pt]{pic cs:XB})  [line width=0.2mm] to ([yshift=3pt]{pic cs:YB});
	\draw [->] ([yshift=3pt]{pic cs:YB'})  [line width=0.2mm] to ([yshift=3pt]{pic cs:ZB});
  \draw ([yshift=7pt]{pic cs:XB})  edge[line width=0.2mm, out=55,in=125,->] ([yshift=7pt]{pic cs:ZB});

	\draw [->] ([yshift=3pt]{pic cs:YC'})  [line width=0.2mm] to ([yshift=3pt]{pic cs:ZC});
	\draw ([yshift=7pt]{pic cs:XC})  edge[line width=0.2mm, out=55,in=125,->] ([yshift=7pt]{pic cs:ZC});

	\draw [->] ([yshift=3pt]{pic cs:YD'})  [line width=0.2mm] to ([yshift=3pt]{pic cs:ZD});
	\draw ([yshift=7pt]{pic cs:XD})  edge[line width=0.2mm, out=55,in=125,->] ([yshift=7pt]{pic cs:ZD});
  \end{tikzpicture}
\end{center}
It is then immediate to see that  $T\models Y=0\cf (X=1 \cf Z\neq 1)$ but $T\not\models Y=0\supset (X=1 \cf Z\neq 1)$. The important element in this example (and similar ones) is that we have an intervention $do(Y=0)$ which does not modify the team component of the causal team, but modifies the graph -- which encodes counterfactual relations.

Other simple examples could be constructed for consequents which contain probabilistic statements -- we will define a semantics for them in later sections. We could not find, instead, a counterexample with consequents that involve only selective implications, dependence atoms and connectives; but here is a counterexample using selective implications and \emph{contradictory negations} of dependence atoms, as defined in subsection \ref{SUBSEM}. Consider the following causal teams, with invariant function $f_Z(X,Y):= Y$:

\begin{center}
$S$: 
\begin{tabular}{|c|c|c|}
\hline
 \multicolumn{3}{|c|}{ }\\ 
 \multicolumn{3}{|c|}{X\tikzmark{XE} \  \tikzmark{YE}Y\tikzmark{YE'} \  \tikzmark{ZE}Z} \\
\hline
 0 & 0 & 0\\
\hline
1 & 1 & 1\\
\hline
\end{tabular}

\vspace{10pt}$(S_{X=0})=(S_{X=0})^{X=0}$:
\begin{tabular}{|c|c|c|}
\hline
 \multicolumn{3}{|c|}{ }\\ 
 \multicolumn{3}{|c|}{X\tikzmark{XF} \  \tikzmark{YF}Y\tikzmark{YF'} \  \tikzmark{ZF}Z} \\
\hline
 0 & 0 & 0\\
\hline
 0 & 1 & 1\\
\hline
\end{tabular}
\hspace{10pt} $S^{X=0} = (S^{X=0})^{X=0}$:
\begin{tabular}{|c|c|c|}
\hline
 \multicolumn{3}{|c|}{ }\\ 
 \multicolumn{3}{|c|}{X\tikzmark{XG} \  \tikzmark{YG}Y\tikzmark{YG'} \  \tikzmark{ZG}Z} \\
\hline
 0 & 0 & 0\\
\hline
\end{tabular}

 \begin{tikzpicture}[overlay, remember picture, yshift=.25\baselineskip, shorten >=.5pt, shorten <=.5pt]
	\draw [->] ([yshift=3pt]{pic cs:YE'})  [line width=0.2mm] to ([yshift=3pt]{pic cs:ZE});
  \draw ([yshift=7pt]{pic cs:XE})  edge[line width=0.2mm, out=55,in=125,->] ([yshift=7pt]{pic cs:ZE});
    
	\draw [->] ([yshift=3pt]{pic cs:YF'})  [line width=0.2mm] to ([yshift=3pt]{pic cs:ZF});
	
	%
	\draw [->] ([yshift=3pt]{pic cs:YG'})  [line width=0.2mm] to ([yshift=3pt]{pic cs:ZG});
	\draw ([yshift=7pt]{pic cs:XG})  edge[line width=0.2mm, out=55,in=125,->] ([yshift=7pt]{pic cs:ZG});
	
	\end{tikzpicture}
\end{center}

Then $S\models X=0 \cf (X=0 \supset \not\dep{X}{Z})$ but $S\not\models X=0 \supset (X=0 \supset \not\dep{X}{Z})$. Of course, it is questionable whether we should allow this kind of arrow $X\rightarrow Z$ in our definition of causal teams.

Let us now see how dependence atoms interact with the two kinds of implication. For what regards selective implication, the following rules are valid:
\[
\frac{\dep{\SET{X}}{Y}}{\SET{X} = \SET{x} \supset \con{Y}} \hspace{35pt} \frac{\bigwedge_{\SET{x}\in Ran(\SET X)}\bigvee_{y\in Ran(Y)}(\SET{X}=\SET{x}\supset Y=y)}{\dep{\SET{X}}{Y}}.
\]
As a simple special case of the second rule, we have that $Y=y$ implies $\con{Y}$. An analogous rule holds for counterfactual implication:
\[
 \frac{\bigwedge_{\SET{x}\in Ran(\SET X)}\bigvee_{y\in Ran(Y)}(\SET{X}=\SET{x}\cf Y=y)}{\dep{\SET{X}}{Y}}.
\]
but the implication from $\dep{\SET{X}}{Y}$ to $\SET{X} = \SET{x} \cf \con{Y}$ only holds on teams where the tuple $\SET{x}$ occurs.

\subsection{Permutation and exportation/importation of antecedents}

Consider the following laws of \emph{permutation, exportation} and \emph{importation} of antecedents:
\[
(P): \psi \cf (\psi'\cf \chi)  \equiv \psi' \cf (\psi\cf \chi)
\]
\[
(E): (\psi \land \psi')\cf \chi \models \psi \cf (\psi'\cf \chi)
\]
\[
(I): \psi \cf (\psi'\cf \chi) \models (\psi \land \psi')\cf \chi
\]
These laws famously fail for Stalnaker's (\cite{Sta1968}) and Lewis' (\cite{Lew1973}) counterfactuals; see \cite{Sid2010}, chap. 8; they are also claimed to fail, in general, for natural language counterfactuals. 
The purpose of this section is to show that under some restriction (that the variables mentioned in $\psi$ and, resp. $\psi'$ be distinct) these laws are valid for our counterfactuals. Under this restriction (which will be somewhat  relaxed in the next section), exportation/importation amounts to the statement that an intervention over a set of variables can be split into successive interventions over disjoint subsets of variables; permutation corresponds to the assertion the two interventions of this kind can be performed in any order.


\begin{rema}
We assume in this subsection that causal teams are explicit. However, the results presented here hold for non-explicit causal teams as well, although the proofs in that case are more involved.
\end{rema}





\begin{lem} \label{LEMPA}
Let $G$ be a finite acyclic graph, $\mathbf{X},Y,W$ vertices of $G$ with $W\in PA_Y$, $Y\notin \SET X$ and $Y$ reachable from $\SET X$ (that is, $d(\mathbf{X},Y)\geq 0$). Then $d(\mathbf{X},W)<d(\mathbf{X},Y)$. 
\end{lem}

\begin{proof}
$d(\mathbf{X},Y) = sup\{lenght(P) | P \text{ directed path of } G_{-\SET X} \text{ from some } X\in \mathbf{X}$ $\text{to } Y\}$. But any directed path from $\mathbf{X}$ to $W$ is included in some path from this set. So $d(\mathbf{X},W)\leq d(\mathbf{X},Y)$.

Assume first that at least one path from $\mathbf{X}$ to $W$ exists in $G_{-\SET X}$. Since $G$ is finite and acyclic, $d(\mathbf{X},W)=n\in \mathbb{N}$. Let $P$ be a path of $G_{-\SET X}$ from some $X\in \mathbf{X}$ to $W$ of lenght $n$. Then $P\cup\{(W,Y)\}$ is a path of $G_{-\SET X}$ from $X$ to $Y$ of lenght $n+1$. 
This holds for any $X\in \SET X$. Since $\SET X$ is finite, $d(\SET X,Y) = \operatorname{max}_{X\in \SET X}d(X,Y) > \operatorname{max}_{X\in \SET X}d(X,W) = d(\SET X,W)$. 
 
If instead no path from $\mathbf{X}$ to $W$ exists, then $d(\mathbf{X},W)=-1$, and the claim of the Lemma is satisfied.
\end{proof}

In the next theorem exportation and importation of antecedents are expressed as rules. A more general version will be proved later on (Theorem \ref{FULLIMPEXP}).

\begin{teo}\label{IMPEXP}
Let $T$ be a causal team with $G(T)$ finite acyclic, $\mathbf{X},\mathbf{Y}\in dom(T)$ such that $\mathbf{X} \cap \mathbf{Y}= \emptyset$, $\mathbf{x}\in Ran(\mathbf{X})$, and $\mathbf{y}\in Ran(\mathbf{Y})$. Then $T_{\mathbf{X}=\mathbf{x}\land \mathbf{Y}=\mathbf{y}} = (T_{\mathbf{X}=\mathbf{x}})_{\mathbf{Y}=\mathbf{y}}$. Therefore, the following rules
\[
(IMP): \frac{\SET X=\SET x \cf (\SET Y=\SET y \cf \chi)}{(\SET X=\SET x \land \SET Y=\SET y) \cf \chi}  \hspace{25pt} (EXP): \frac{(\SET X=\SET x \land \SET Y=\SET y) \cf \chi}{\SET X=\SET x \cf (\SET Y=\SET y \cf \chi)}
\]
are valid, under the above restrictions, over recursive teams.
\end{teo}

\begin{proof}
Notice first
that for any $s'\in (T_{\SET{X}=\SET{x}})_{\SET{Y}=\SET{y}}$ and any $s''\in T_{\SET{X}=\SET{x}\land \SET{Y}=\SET{y}}$, $s'(\SET{X})=s''(\SET{X}) = \SET{x}$ and $s'(\SET{Y})=s''(\SET{Y}) = \SET{y}$.

Secondly, the restriction to explicit causal teams implies that  $\mathcal{F}_{T} \supseteq\mathcal{F}_{T_{{\SET{X}=\SET{x}}}} \supseteq \mathcal{F}_{(T_{{\SET{X}=\SET{x}}})_{\SET{Y}=\SET{y}}} = \mathcal{F}_{T_{{\SET{X}=\SET{x}}\land\SET{Y}=\SET{y}}}$.

Fix $s\in T$. Let $s'\in (T_{\SET{X}=\SET{x}})_{\SET{Y}=\SET{y}}$ be the assignment obtained by applying to $\{s\}$ the intervention $do(\SET{X}=\SET{x})$ followed by $do(\SET{Y}=\SET{y})$, and $s''$ the assignment obtained by applying $do(\SET{X}=\SET{x}\land\SET{Y}=\SET{y})$ to $s$. We shall prove, by a simultaneous induction on $m$ and $n$, that $s'(Z)=s''(Z)$ for any variable $Z$ such that $d(\SET{X},Z)\leq m$ and $d(\SET{Y},Z)\leq n$.  In the base case $Z\in  \SET{X}  \cup  \SET{Y}   $, this result is obvious.

Suppose that the statement holds for $m$ and $n$, and let $Z\in dom(T) \setminus( \SET{X}  \cup  \SET{Y} )$ be a variable such that $d(\SET{X},Z) \leq m$ and $d(\SET{Y},Z) = n+1$ (the proof of the symmetric case is analogous). 

From the algorithm $do(\SET{Y}=\SET{y})$, it follows that $s'(Z) = \mathcal F_{T_{\SET{X}=\SET{x}}}(Z)(s'(PA_Z))$ (notice indeed that, thanks to acyclicity, the algorithm does not modify the $PA_Z$-columns after modifying the $Z$-column), which, by 
the observation above, is equal to $\mathcal F_{T}(Z)(s'(PA_Z))$; and from the algorithm $do(\SET{X}=\SET{x}\land \SET{Y}=\SET{y})$, it follows that $s''(Z) = \mathcal F_{T}(Z)(s''(PA_Z))$. 
By Lemma \ref{LEMPA} (which we can apply because $d(\SET Y, Z) = n+1$ implies $Z\notin \SET Y$), the variables in $PA_Z$ have strictly smaller distances from 
 $\SET{Y}$ than $Z$ (and, obviously, at most the same distance from $\SET X$). Then, by the inductive hypothesis, $s'(PA_Z) = s''(PA_Z)$. So, applying $\mathcal F_{T}(Z)$, we obtain $s'(Z)=s''(Z)$.
\end{proof}


As a corollary, we obtain the following variant of the permutation rule: 

\begin{teo}\label{PERM}
Let $T$ be a causal team with $G(T)$ finite acyclic, $\mathbf{X},\mathbf{Y}\in dom(T)$ such that $\mathbf{X} \cap \mathbf{Y}= \emptyset$, $\mathbf{x}\in ran(\mathbf{X})$, and $\mathbf{y}\in ran(\mathbf{Y})$. Then $(T_{\mathbf{X}=\mathbf{x}})_{\mathbf{Y}=\mathbf{y}} = (T_{\mathbf{Y}=\mathbf{y}})_{\mathbf{X}=\mathbf{x}}$. Therefore, the following rule 
\[
(PERM): \frac{\SET X=\SET x \cf (\SET Y=\SET y \cf \chi)}{\SET Y=\SET y \cf (\SET X=\SET x \cf \chi)}
\]
is valid, under the above restrictions, over recursive causal teams.
\end{teo}

\begin{proof}

From Theorem \ref{IMPEXP} we obtain the equalities $(T_{\mathbf{X}=\mathbf{x}})_{\mathbf{Y}=\mathbf{y}} = T_{\mathbf{X}=\mathbf{x}\land \mathbf{Y}=\mathbf{y}}$ and $(T_{\mathbf{Y}=\mathbf{y}})_{\mathbf{X}=\mathbf{x}} = T_{\mathbf{Y}=\mathbf{y}\land\mathbf{X}=\mathbf{x} }$. Since the order of variables is irrelevant in the definition of the $do$ algorithm, we also have $T_{\mathbf{X}=\mathbf{x}\land \mathbf{Y}=\mathbf{y}} = T_{\mathbf{Y}=\mathbf{y}\land \mathbf{X}=\mathbf{x}}$. Transitivity yields the desired result.
\end{proof}
\subsection{Generalized import-export and permutation rules} 

The purpose of this section is to show that the assumptions of the permutation and the import-export rules can be relaxed in the following way: if the interventions involved are over two sets of variables $\mathbf{X}$ and $\mathbf{Y}$ with nonempty intersection, we have only to require that the two interventions act ``in the same way'' over the common set of variables $\mathbf{X}\cap\mathbf{Y}$.

\begin{lemma}\label{REP}
For any $\SET{X}\subseteq dom(T)$, $(T_{\SET{X}=\SET{x}})_{\SET{X}=\SET{x}} = T_{\SET{X}=\SET{x}}$.
\end{lemma}

\begin{proof}
By a simple induction on the steps of the ``for'' cycle in the $do$ algorithm.
\end{proof}

\begin{lem}\label{LEMSETDIFF}
Let $T$ be a causal team with $G(T)$ finite acyclic, $\mathbf{X},\mathbf{Y}\in dom(T)$, $\mathbf{x}\in Ran(\mathbf{X})$, and $\mathbf{y}\in Ran(\mathbf{Y})$. Suppose also that, if $X_i=Y_j\in \mathbf{X}\cap \mathbf{Y}$, then $x_i = y_j$. Let $\mathbf{Y'} = \mathbf{Y} \setminus \SET{X}$, and $\mathbf{y'}=\{y_j|Y_j\in\mathbf{Y'}\}$. Then $(T_{\mathbf{X}=\mathbf{x}})_{\mathbf{Y}=\mathbf{y}} = (T_{\mathbf{X}=\mathbf{x}})_{\mathbf{Y'}=\mathbf{y'}}$.
\end{lem}

\begin{proof}
Let $\SET{X'}=\SET{X}\setminus \SET{Y}$, and $\SET{x'}=\{x_j|X_j\in \SET{X'}\}$.
Let $\mathbf{Z} = \mathbf{X} \cap \mathbf{Y}$ and $\mathbf{z}= \SET x \cap \SET y$.
 Applying twice Theorem \ref{IMPEXP} (import-export), we get $(T_{\mathbf{X}=\mathbf{x}})_{\mathbf{Y}=\mathbf{y}} = (((T_{\mathbf{X'}=\mathbf{x'}})_{\mathbf{Z}=\mathbf{z}})_{\mathbf{Z} = \mathbf{z}})_{\mathbf{Y'}=\mathbf{y'}}$. But this last term is equal to $((T_{\mathbf{X'}=\mathbf{x'}})_{\mathbf{Z}=\mathbf{z}})_{\mathbf{Y'}=\mathbf{y'}}$, by Lemma \ref{REP}. Given that $\mathbf{Z}$ and $\mathbf{X'}$ are disjoint, we can apply again Theorem \ref{IMPEXP} to obtain $((T_{\mathbf{X'}=\mathbf{x'}})_{\mathbf{Z}=\mathbf{z}})_{\mathbf{Y'}=\mathbf{y'}} = (T_{\mathbf{X}=\mathbf{x}})_{\mathbf{Y'}=\mathbf{y'}}$. The chain of equalities yields the desired result.
\end{proof}

Next we show that the import-export rule works under the weaker hypothesis that the two interventions act in the same way on the shared set of variables.

\begin{teo}\label{FULLIMPEXP}
Let $T$ be a causal team with $G(T)$ finite acyclic, $\mathbf{X},\mathbf{Y}\in dom(T)$.  $\mathbf{x}\in ran(\mathbf{X})$, and $\mathbf{y}\in ran(\mathbf{Y})$. Suppose also that, if $X_i=Y_j\in \mathbf{X}\cap \mathbf{Y}$, then $x_i = y_j$. Then $T_{\mathbf{X}=\mathbf{x}\land \mathbf{Y}=\mathbf{y}} = (T_{\mathbf{X}=\mathbf{x}})_{\mathbf{Y}=\mathbf{y}}$.
\end{teo}

\begin{proof}
Apply Lemma \ref{LEMSETDIFF} to see that $(T_{\mathbf{X}=\mathbf{x}})_{\mathbf{Y}=\mathbf{y}} = (T_{\mathbf{X}=\mathbf{x}})_{\mathbf{Y'}=\mathbf{y'}}$, where $\mathbf{Y'}=\mathbf{Y}\setminus\mathbf{X}$. Given that $\mathbf{X}$ and $\mathbf{Y'}$ are disjoint sets of variables, we can use the weak import-export rule (Theorem \ref{IMPEXP}) to obtain $(T_{\mathbf{X}=\mathbf{x}})_{\mathbf{Y'}=\mathbf{y'}} = T_{\mathbf{X}=\mathbf{x} \land \mathbf{Y'}=\mathbf{y'}}$. But now the formula $\mathbf{X}=\mathbf{x} \land \mathbf{Y}=\mathbf{y}$ differs from $\mathbf{X}=\mathbf{x} \land \mathbf{Y'}=\mathbf{y'}$ only in that it contains some repetitions of conjuncts that are already present in $\mathbf{X}=\mathbf{x} \land \mathbf{Y'}=\mathbf{y'}$. Therefore, the two formulas define the same interventions (as can be seen by checking the $do$ algorithm); this immediately implies that $T_{\mathbf{X}=\mathbf{x} \land \mathbf{Y}=\mathbf{y}} = T_{\mathbf{X}=\mathbf{x} \land \mathbf{Y'}=\mathbf{y'}}$. The whole chain of equalities yields the result.
\end{proof}

Again, under these hypotheses the order of the interventions is irrelevant.

\begin{teo}\label{FULLPERM}
Let $T$ be a causal team with $G(T)$ finite acyclic, $\mathbf{X},\mathbf{Y}\in dom(T)$, $\mathbf{x}\in ran(\mathbf{X})$, and $\mathbf{y}\in ran(\mathbf{Y})$. Suppose also that, if $X_i=Y_j\in \mathbf{X}\cap \mathbf{Y}$, then $x_i = y_j$. Then $(T_{\mathbf{X}=\mathbf{x}})_{\mathbf{Y}=\mathbf{y}} = (T_{\mathbf{Y}=\mathbf{y}})_{\SET{X}=\mathbf{x}}$.
\end{teo}

\begin{proof}
This works like the proof of Theorem \ref{PERM}, using the relaxed import-export rule (Theorem \ref{FULLIMPEXP}) instead of Theorem \ref{IMPEXP}.
\end{proof}

\subsection{Composition and effectiveness rules}\label{SUBSCOMPL}

Galles and Pearl (\cite{GalPea1998}) present a complete and sound system for interventionist counterfactuals over recursive causal models (to be more precise, the result holds only for causal models that are parametric and under the restriction that the range of each variable is \emph{finite}). This system is based on two main rules called \emph{composition} and \emph{effectiveness}. In informal presentations (\cite{Pea2000}), Pearl claims that these two axioms exhaust the system; a closer look at the original paper (and at a subsequent clarifying paper by Halpern, \cite{Hal2000}) shows that two further axioms (\emph{definiteness} and \emph{uniqueness}), plus ``the rules of logic'' are part of the system. Let us first point out that Pearl uses a very different syntax from ours which originates in a tradition of representation of counterfactuals common in the statistical literature (the ``potential outcome'' approach, \cite{Ney1923},\cite{Rub1974}). He uses
\[
Y_x(u) =y
\]
to mean that, in an actual state in which the (set of) exogenous variables $U$ take values $u$, if we fixed the value of $X$ to $x$ by an intervention, then $Y$ would take value $y$. Given that in a recursive structural equation model the endogenous variables are functionally determined by the exogenous ones, we can represent such an assignment $u$ of values to the exogenous variables as a unique assignment $s_u$. Therefore, we might represent Pearl's counterfactual in a notation closer to the logical tradition as:
\[
s_u \models X=x \cf Y = y.
\]

Assuming that, unless otherwise specified, variables and values are implicitly universally quantified, here is now the list of properties that Galles and Pearl ascribe to causal models with unique solutions (which include recursive models): 
\[
(Composition): W_x(u) = w \Rightarrow Y_{xw}(u) = Y_x(u)
\]
\[
(Effectiveness): X_{xw}(u)=x 
\]
\[
(Definiteness):
 \text{ there is an $x$ such that } X_{\SET{y}}(u)=x 
\]
\[
(Uniqueness): X_{\SET{y}}(u)=x \land X_{\SET{y}}(u)=x' \Rightarrow x=x'
\]
to which, according to Halpern (\cite{Hal2000}), we should add all classical propositional tautologies plus Modus Ponens. A further axiom scheme (call it REC) forces the system to be recursive (i.e., the graph to be acyclic). Translating the REC axiom into our language is somewhat problematic; we consider this issue towards the end of the present subsection.


Do these axioms hold for our causal teams? First of all, we should decide how these axioms should be interpreted in our framework. One possible way to go is to conceive these counterfactuals as encoding properties of a team instead of an actual state; that is, to have a team $T$ play the role that was previously played by $u$ (or by $s_u$). It is then straightforward to translate effectiveness as the statement that each causal team $T$ satisfies:
\[
(EFF): \models (X=x \land \SET{W}=\SET{w}) \cf X=x.
\]
Translating composition is a more complex task. First of all, Galles\&Pearl have an operator $\Rightarrow$ which is most likely intended as a material implication (which, by the deduction theorem, reflects logical implication). We do not have in $\CD$ any such operator (we know that our selective implication  $\supset$ obeys a deduction theorem, but only for $\CO$ antecedents). For this reason, we will focus here on the languages $\C^{neg}$ and $\CO^{neg}$, which are flat and closed under dual negation. We can therefore add modus ponens in the following form:
\[
MP_\lor: \frac{\theta \hspace{15pt} \neg\theta\lor\chi}{\chi}.
\]

However, in order to facilitate the transition to more general languages like $\CD$, which are not closed under negation, we may also adopt the strategy of rewriting all the implications corresponding to Pearl's axioms as inference rules. For example, it will be reasonable to move from the formal language to the semantic metalanguage, and think of composition as a rule of inference rather than an axiom. The composition rule poses a further problem: the formula $Y_{xw}(u) = Y_x(u)$ is not a counterfactual. We may think instead of the equivalent expression (for all $y\in \mathcal R(Y)$): $Y_{xw}(u) = y \iff Y_x(u)=y$. It then becomes natural to decompose ``composition'' into \emph{two} rules of inference:
\[
(CE): \frac{\SET{X}=\SET{x} \cf W=w \hspace{25pt} (\SET{X}=\SET{x} \land W=w)\cf Y=y}{\SET{X}=\SET{x} \cf Y=y}
\]
\[
(CI): \frac{\SET{X}=\SET{x} \cf W=w \hspace{25pt} \SET{X}=\SET{x}\cf Y=y}{(\SET{X}=\SET{x} \land W=w) \cf Y=y}
\]
Definiteness (``Decidability'') contains an existential quantification. \emph{If we assume the ranges of variables to be finite}, then we can formulate this axiom in a way that works relativized to teams that have a specific finite range for $X$:
\[
(DEC): \lor_{x\in Ran(X)} (\SET{Y}=\SET{y}\cf X=x)
\]
What this axiom expresses about a team is that the team can be decomposed into subteams, each of which satisfies, after the intervention $do(\SET Y = \SET y)$, one of the $X=x$ atoms; it is true of any parametric causal team (in the nonparametric case, the disjunction should not be taken over $Ran(X)$, but over $Ran(X)\cup Term_{G(T)}$).

The uniqueness axiom mentions formulas of the form $x=x'$, which are absent from our language. But we can think of a reformulation of these axioms due to Halpern (which he simply calls the ``equality'' axiom scheme), which rephrases it as a scheme of rules of inference (one rule for each variable $X$, set of variables $\SET{Y}$, and pair of values $x,x'\in Ran(X)$ with $x\neq x'$):
\[
(UNI): \frac{\SET{Y}=\SET{y}\cf X=x}{\SET{Y}=\SET{y}\cf X\neq x'}
\]
The import of this axiom seems much weaker than in the Galles-Pearl case, given that $X=x$ and $X\neq x'$ are global statements over many assignments.

Our translations of the axioms and rules all fall within the language $\C^{neg}$. By its flatness, it immediately follows that the axioms EFF and the appropriate instances of DEC also hold for our causal teams.

\begin{teo}
The rules CE, CI and UNI are sound for parametric acyclic causal teams. 
\end{teo}

\begin{proof}
CE) Assume $T$ satisfies the two assumptions of the rule. Since the consequent $Y=y$ of the counterfactuals involved is a formula that does not contain counterfactuals, in order to prove that $T$ satisfies the conclusion of the rule, it is sufficient to show that $(T_{\SET{X}=\SET{x}})^- = (T_{\SET{X}=\SET{x} \land W = w})^-$. In case $W\in \SET{X}$, the rule holds for trivial reasons. Otherwise, a previous theorem (\ref{IMPEXP}) tells us that $T_{\SET{X}=\SET{x} \land W = w} = (T_{\SET{X}=\SET{x}})_{W = w}$. The first assumption of CE tells us that $T_{\SET{X}=\SET{x}} \models W=w$. But then, obviously, applying $do(W=w)$ to $T_{\SET{X}=\SET{x}}$ leaves its underlying team unvaried.

CI) An analogous argument.

UNI) Immediate.
\end{proof}

Finally, notice that, by the flatness of $\C^{neg}$, all classical validities that can be formulated in $\C^{neg}$ are sound for acyclic teams. (But remember that some classical semantic principles, like SEM, fail).
Halpern (\cite{Hal2000}) considers a language $L_{uniq}$ that allows boolean formulas as consequents, and boolean combinations of counterfactuals, but does not allow embeded counterfactuals. This language can be identified with an appropriate fragment of $\C^{neg}$, call it $\C_u$ . What Halpern shows is that the axioms described above form a sound and complete system for $L_{uniq}$ over recursive (i.e. acyclic) causal models. In particular, they axiomatize the set of valid formulas of the Galles-Pearl language (which is restricted to conjunctions of counterfactuals of the form $\SET{X}=\SET{x} \cf Y=y$), although derivations may use formulas that are not part of the Galles-Pearl language.

As we have seen, these axioms and rules are sound also for $\C^{neg}$, over recursive causal teams. But unlike $\C_u$, the language $\C^{neg}$ also allows counterfactual conditionals to occur in the consequents. Thus it may be that some more axioms are needed in order to obtain a complete system for $\C^{neg}$ over recursive causal teams.

Our strategy for the rest of this section is to show that any counterfactual in $\C^{neg}$ is equivalent to a counterfactual with a "simple" consequent (i.e., a consequent which does not contain embedded counterfactuals, or even atomic formulas of the form $Y\neq y$), and to isolate the rules that allow this transformation. 

Our first observation is that Halpern \cite{Hal2000} proves that recursive causal models (or, more generally, causal models whose equations have unique solutions) satisfy some rules that allow removing connectives from consequents. In our language, these rules may be stated as follows:

\[
(OR-OUT): \frac{\SET{X}=\SET{x}\cf \psi\lor \psi'}{(\SET{X}=\SET{x}\cf \psi)\lor(\SET{X}=\SET{x}\cf \psi')}
\]
\[
(AND-OUT): \frac{\SET{X}=\SET{x}\cf \psi\land \psi'}{(\SET{X}=\SET{x}\cf \psi)\land(\SET{X}=\SET{x}\cf \psi')}
\]
\[
(NEG-OUT): \frac{\SET{X}=\SET{x}\cf \neg\psi}{\neg(\SET{X}=\SET{x}\cf \psi)}
\]
Halpern also shows the soundness of their inverses (call them OR-IN, AND-IN, NEG-IN). We check now that the same holds in our causal teams.

\begin{lemma} \label{INTSPLIT}
Suppose a causal team of the form $T_{\SET X = \SET x}$ has causal subteams $U',V'$ such that $(U')^-\cup (V')^- = (T_{\SET X = \SET x})^-$. Then $T$ has causal subteams $U,V$ such that $U_{\SET X = \SET x} = U'$ and $V_{\SET X = \SET x} = V'$.
\end{lemma}

\begin{proof}
Let $T,U',V'$ as above. For each $t\in (T_{\SET X = \SET x})^-$, define the set of assignments $S_t = \{s\in T^-| \{s\}_{\SET X = \SET x} = \{t\}\}$. Define then $U^-:= \bigcup_{t\in U'} S_t$ and $V^-:= \bigcup_{t\in V'} S_t$. Define $G(U), \mathcal F_U$, etc. in the obvious way.

Let us show that $U,V$ behave as wanted. Let $t\in (U')^-$. Pick an $s \in S_t \subseteq U^-$. $\{s\}_{\SET X = \SET x} = \{t\}$; so, $t\in (U_{\SET X = \SET x})^-$. This proves $(U')^-\subseteq(U_{\SET X = \SET x})^-$.

Let $t\notin (U')^-$. Suppose for the sake of contradiction that there is $s\in U$ such that $\{s\}_{\SET X = \SET x}= \{t\}$. But then there is a $t'\in U'$ such that $s\in S_{t'}\subseteq U^-$. But this implies that $\{t\} = \{s\}_{\SET X = \SET x} = \{t'\}$. Therefore $t \in (U')^-$: contradiction.
\end{proof}

\begin{teo}
The rules OR-OUT, AND-OUT, NEG-OUT, OR-IN, AND-IN, NEG-IN are sound for $\C^{neg}$ on parametric causal teams with unique solutions; except for NEG-IN/OUT, this also holds in $\CD$.
\end{teo} 

\begin{proof}
The AND and NOT case are straightforward. The soundness of OR-OUT/OR-IN follows immediately from Lemma \ref{INTSPLIT}.
\end{proof}


How should we treat counterfactuals occurring in the consequent? The importation rule presented in previous sections (Theorem \ref{FULLIMPEXP}) can eliminate them only under some synctactical restrictions. We need a different inference rule that may work without restrictions. We prove here that, if two consecutive interventions affect some common set of variables, then the action performed by the first intervention on these common variables is completely overwritten by the second intervention.

\begin{lemma} \label{REWRITE}
Let $T$ be a recursive causal team, $\SET{X}\subseteq dom(T)$, and $\SET{x},\SET{x'}\in Ran(\SET X)$. Then $(T_{\SET{X} = \SET{x}})_{\SET{X} = \SET{x'}} = T_{\SET{X} = \SET{x'}}$.
\end{lemma}

\begin{proof}
Since all the interventions act on the same set of variables $\SET{X}$, they give rise to the same graph. A straightforward induction on $d(\SET{X},Y)$ can be used to show that, for any $s\in T$, if $\{t\}= (\{s\}_{\SET{X} = \SET{x}})_{\SET{X} = \SET{x'}}$ and $\{t'\} = \{s\}_{\SET{X} = \SET{x'}}$, then $t(Y) = t'(Y)$ for any $Y\in dom(T)$.  
\end{proof}

\begin{teo}
Let $T$ be a recursive causal team, $\SET{X},\SET{Y}\subseteq dom(T)$, $\SET{X'} = \SET{X}\setminus \SET{Y}$ and $\SET{x'} = \SET{x}\setminus \SET{y}$.  Then $(T_{\SET{X}=\SET{x}})_{\SET{Y}=\SET{y}} = (T_{\SET{X'}=\SET{x'}})_{\SET{Y}=\SET{y}}$.
\end{teo}

\begin{proof}
Let $\SET{Y'} = \SET{Y}\setminus \SET{X}$ and $\SET{y'} = \SET{y}\setminus \SET{x}$.   Let $\SET{Z} = \SET{X}\cap \SET{Y}$. Suppose $\SET{Z}=(X_{i_1},\dots, X_{i_k}) = (Y_{j_1},\dots, Y_{j_k})$. Write $\SET{z_1}$ for $(x_{i_1},\dots, x_{i_k})$ (the restriction of $\SET x$ to values for $\SET Z$), and $\SET{z_2}$ for $(y_{j_1},\dots, y_{j_k})$ (the restriction of $\SET y$ to values for $\SET Z$). Then, by Theorem \ref{IMPEXP}, $(T_{\SET{X}=\SET{x}})_{\SET{Y}=\SET{y}} = (((T_{\SET{X'}=\SET{x'}})_{\SET{Z} = \SET{z_1}})_{\SET{Z} = \SET{z_2}})_{\SET{Y'}=\SET{y'}}$. By Lemma \ref{REWRITE}, this is equal to $((T_{\SET{X'}=\SET{x'}})_{\SET{Z} = \SET{z_2}})_{\SET{Y'}=\SET{y'}}$. Applying  Theorem \ref{IMPEXP} again, we obtain $(T_{\SET{X'}=\SET{x'}})_{\SET{Y}=\SET{y}}$.
\end{proof}

This result, combined with Theorem \ref{IMPEXP}, can be immediately turned into a pair of inference rules:
\[
(CF-OUT): \frac{\SET{X}=\SET{x} \cf (\SET{Y}=\SET{y} \cf \psi)}{(\SET{X'}=\SET{x'} \land \SET{Y}=\SET{y}) \cf \psi}
\]
where $\SET{X'} = \SET{X}\setminus \SET{Y}$ and $\SET{x'} = \SET{x}\setminus \SET{y}$.
Its inverse, CF-IN, works for $\SET{X'}\cap\SET{Y}=\emptyset$ and $\SET{X}\supseteq\SET{X'}$. This pair of rules corresponds essentially to the axiom CM11 proposed by Briggs (\cite{Bri2012}, sect. 4).

Finally, we return to the axiom REC in Halpern (\cite{Hal2000}), which characterizes recursive systems.  Halpern shows how to represent, in the potential outcome formalism, the notion ``$X$ affects $Y$ under some intervention'' (we believe this notion is the same as Woodward's notion of \emph{contributing cause}). We were not able to express this notion in $C^{neg}$. Instead we can express it in an extended language $C^{neg}(\sqcup)$ where we allow positive occurrences of the so-called \emph{intuitionistic disjunction} $\sqcup$, which follows the semantical clause:
\[
T\models \psi \sqcup \psi' \iff T\models \psi \text{ or } T\models \psi'
\]
With this additional connective, ``$X$ affects $Y$'' can be written as follows; write $ND_X$ for the set of nondescendants of $X$, and $\mathcal W:= \{\SET Z\subset dom(T)\setminus \{X,Y\} | ND_X \subseteq \SET Z\}$:
\[
CC(X,Y): \bigsqcup \big\{\SET Z = \SET z \cf [(X=x \cf Y=y) \land (X=x' \cf Y=y')]\big\}.
\]
where the disjunction is taken over all $\SET Z\in \mathcal W, \SET z\in Ran(\SET Z),x\neq x'\in Ran(X),y\neq y'\in Ran(Y)$. Then in $\C^{neg}(\sqcup)$ we can write:
\[
(REC_k): \frac{CC(X_1,X_2) \hspace{10pt} \dots \hspace{10pt} CC(X_{n-1}, X_n)}{\neg CC(X_n,X_1)} 
\] 
and take REC to be the set of all the REC$_k$ rules, for $k\in \mathbb{N}$, $k\geq 2$.

Given a causal team $T$, let $AX_\C[T]$\footnote{The extra parameter $T$ is there because we need, as was done in \cite{Hal2000}, to restrict the axioms to teams which share the same ``signature'' as $T$, that is, they have the same domain of variables and the same ranges for each of the variables.} be the axiom system constituted of axiom schemes EFF; DEC; 
 UNI;
 and the rules CE, CI, OR/AND/NEG/CF-OUT, OR/AND/NEG/CF-IN, MP$_\lor$ and REC. Now $AX_\C[T]$ is a set of valid $C^{neg}(\sqcup)$ formulas and sound rules for $C^{neg}(\sqcup)$, but we can show it to be complete for the restricted set of formulas $C^{neg}$:\footnote{The approach which consists of defining a formal system over a language more general than the set of target formulas is not unheard of; it is for example the approach taken in \cite{GalPea1998}.}

\begin{teo}
$AX_\C[T]$ is a sound and complete axiom system for the language $\C^{neg}$ on recursive causal teams of domain $dom(T)$ and range function $\mathcal R_T$.
\end{teo}

\begin{proof}
Soundness in $\C^{neg}$ was shown above; it is not difficult to show that soundness extends to $C^{neg}(\sqcup)$. Let $\Gamma\subseteq \C^{neg}$, $\varphi\in\C^{neg}$ be such that $\varphi$ holds in any recursive causal team on which $\Gamma$ holds. Notice that any $\psi\in \Gamma$ can be transformed into a $\psi'\in \C_u$ by repeated applications of OR/AND/NEG/CF-OUT; call $\Gamma'$ the set of the $\psi'$. By the soundness and invertibility of the rules, $\Gamma$ and $\Gamma'$ hold in the same causal teams. Similarly, $\varphi$ can be transformed into a $\varphi'\in \C_u$ such that $\varphi$ and $\varphi'$ hold in the same causal teams. Therefore, in order to prove completeness, we only have to show that, whenever $\Gamma'\vdash \varphi'$ according to the (translation of the) axiom system of Halpern, we can construct a derivation $\Gamma \vdash_{AX_\C[T]}\varphi$.

For suppose we have a derivation $\Gamma'\vdash \varphi'$ in Halpern's axiom system. Then we can compose it with a derivation of $\Gamma'$ from $\Gamma$ (using the OUT rules) and a derivation of $\varphi$ from $\varphi'$ (using the IN rules) to obtain a derivation of $\varphi$ from $\Gamma$  within $AX_\C[T]$ (after Halpern's rules and axioms are replaced with our translations).
\end{proof}

 
The logic $\CO^{neg}$ is, again, flat and closed under dual negation. The obvious way to obtain a complete axiomatization for $\CO^{neg}$ is then to add to $AX_\C[T]$ the conversion rule
\[
(SEL_E): \frac{\theta \supset \chi}{\neg \theta\lor\chi}
\]
and its inverse $SEL_I$, which allow to reduce the problem to that of completeness of $\C^{neg}$.

Notice that SEL$_E$ is valid also for non-flat logics, of course under the restriction that $\theta$ be flat (and its negation be present in the language); the validity of SEL$-I$, instead, requires downward closure of the logic.

An alternative approach to obtain a complete axiomatization of $\CO^{neg}$ is to add a rule for the extraction of selective implications from counterfactual consequents:
\[
(SEL-OUT): \frac{\SET X = \SET x\cf(\psi \supset \chi)}{(\SET X = \SET x\cf\psi) \supset (\SET X = \SET x \cf\chi)}
\]
and its inverse SEL-IN. These rules are sound for $\CO^{neg}$ and are both likely to hold (for the appropriate instances) also in languages that are not downward closed.

\section{The nonparametric case: falsifiability and admissibility} \label{NONPARAM}

We have allowed, in our causal teams, the possibility of having incomplete information concerning the invariant functions that relate the variables, up to the extreme case where no information at all is present (except for the set of arguments of the function, and the range of its allowed values). As a consequence, it is in general impossible to fill the counterfactual causal teams generated by interventions with proper values; if we do not have sufficient information to fill an entry with a value, we fill it instead with a formal term which describes the functional dependence of this entry on the values of other variables.

But how should we evaluate counterfactual statements which involve variables whose columns cannot be filled with proper values? We think it should not be possible to ascribe truth values to such statements by the usual clauses. 
 For example, we are not entitled to assert that $Y=3$ in a team whose non-formal entries for $Y$ are all equal to $3$.

Yet in some cases we might be able to observe the falsity of similar statements; therefore, to state their contradictory negation. Let us write $\downarrow s(X)$ to signify that $s(X)$ is a proper value and not a formal term. Let $T$ be a causal team, possibly having some formal entries. We read $T\models^f \psi$ as ``$\psi$ is falsifiable in $T$'', but notice that this is not falsifiability in the dual sense; rather, it is a partial notion of contradictory negation. We propose the following semantical clauses:
\begin{itemize}
\item $T\models^f X=x \text{ if there is } s\in T^- \text{ such that } \downarrow s(X) \text{ and } s(X) \neq x$
\item $T\models ^f X\neq x \text{ if there is } s\in T^- \text{ such that } \downarrow s(X) \text{ and } s(X) = x$
\item $T\models ^f \dep{\SET X}{Y} \text{ if there are } s,s'\in T^- \text{ such that } \downarrow s(Y),\downarrow s'(Y), s(\SET X) = s'(\SET X)  \text{ and } s(Y) \neq s'(Y)$ (notice that the $s(\SET X)$ need not all be proper values) 
\item $T\models^f \psi\land \chi$ if $T\models^f \psi$ or $T\models^f \chi$ 
\item $T\models^f \psi\lor \chi$ if for all subteams $T_1,T_2$ of $T$ with $T_1^-\cup T_2^-=T^-$, we have $T_1 \models^f \psi$ or $T_2\models^f \chi$. 
\end{itemize}
(Notice that the clause for disjunction involves universal quantification over teams).

Can we extend these clauses also to selective and counterfactual implication? Let us consider as an example the team

\begin{center}
\begin{tabular}{|c|c|}
\hline
 \multicolumn{2}{|c|}{$\hspace{-5pt}X \ \rightarrow  Y$} \\
\hline
 2 & 1 \\
\hline
 1 & $\hat f_Y(1)$ \\
\hline
\end{tabular}
\end{center}

It seems unreasonable to assert that this team falsifies the formula $Y=1 \supset X=2$, because, as long as we do not know the function $f_Y$, we cannot know whether $\hat f_Y(1)$ is meant to denote $1$ or some other number; therefore, we do not know whether the second assignment is compatible or not with our selection (if it were, then the formula would be falsified, otherwise it would not be). We now give a clause that takes into account this kind of problem.

 First, let $\psi$ be a classical formula; let $\SET{V}$ be the set of variables occurring in $\psi$; define $T^\psi_*:=T^\psi \cup \{s\in T^- | \not\downarrow s(V) \text{ for some } V\in\SET{V}\}$. Then, for $\chi$ \emph{downward closed} formula (as we know formulas of $\CD$ to be), a somewhat reasonable clause for selective implication seems to be:
\begin{itemize}
\item $T\models^f \psi \supset \chi$ if $T^\psi_* \models^f \chi$
\end{itemize}
This clause is not completely satisfying, because there might be variables in $\SET V$ that are irrelevant for determining the truth value of $\psi$.

For counterfactuals, the natural clause is more immediate:
\begin{itemize}
\item $T\models^f \SET X = \SET x \cf \chi$ if $T_{\SET X = \SET x} \models^f \chi$.
\end{itemize}

Thirdly, notice that, even though we cannot most often decide of the truth of a statement when columns are incomplete, we might however be interested in asserting that some proposition is \emph{admissible} in the given team, that is, consistent with the data we possess. The following seem to be reasonable clauses for the atomic formulas: 

\begin{itemize}
 \item $T\models^a X=x \text{ if for all } s\in T^- \text{ such that } \downarrow s(X), s(X) = x$
 \item $T\models^a X\neq x \text{ if for all } s\in T^- \text{ such that } \downarrow s(X), s(X) \neq x$
\item $T\models^a \dep{\SET X}{Y} \text{ if for all } s,s'\in T^- \text{ such that }  \downarrow s(Y) \downarrow s'(Y), s(\SET X) = s'(\SET X)$, we have $s(Y) =s'(Y)$. 
\end{itemize}

Finding general clauses for composite formulas appears to be a difficult task; for simplicity, we define this notion only for classical formulas in disjunctive normal form; we also require, in each conjunctive clause, that atomic formulas $P^i_j$ contain distinct variables.

\begin{itemize}
\item $T\models^a \bigvee_{i=1..m} \bigwedge_{j=1..n(i)} P^i_j$ ($P^i_j$ being of the form $X^i_j = x^i_j$ or $X^i_j \neq x^i_j$) if there are subteams $T_i$ of $T$, for $i=1..m$, such that 
		\begin{enumerate}
		\item	$T_i \models^a  P^i_j$, for all $j$. 
		\item for each $j,j'=1..n(i)$, if $j\neq j'$, $P^i_j$ is $X^i_j = a$ and $P^i_{j'}$ is $X^i_{j'} = b$ (with $a\neq b$), then for all $s\in T_i$, $s(X^i_j) \neq s(X^i_{j'})$.
		\item for each $j,j'=1..n(i)$, if $j\neq j'$, $P^i_j$ is $X^i_j = a$ and $P^i_{j'}$ is $X^i_{j'} \neq a$, then for all $s\in T_i$, $s(X^i_j) \neq s(X^i_{j'})$.
\end{enumerate}	
	
\end{itemize}

The clauses 2. and 3. above refer to formal inequality between terms. To have an idea of the intuition behind clause 2., the reader may think, for example, of the problem of checking the admissibility of $X=1 \land Y=2$; imagine that there is a row in which both the $X$-column and the $Y$-columm contain the formal term $f(3, g(2))$; then, surely, the formula is not admissible (for $X=1 \land Y=2$ to hold in the team, it is necessary that the $X$ and $Y$-column differ on each row).

An intuition for clause 3. can be similarly obtained by thinking of checking the admissibility of $X=1\land Y\neq 1$; again, for admissibility, we need the term $f(3, g(2))$ not to occur both in the $X$ and $Y$ column.

\section{Probabilistic semantics} \label{PROBSEC}

\subsection{A probabilistic language}

In Woodward's approach, probabilistic notions of causality are considered next to the deterministic ones; for example, a variable $X$ is said to be a \emph{total cause} (in the probabilistic sense) of a different variable $Y$ if there is an intervention $do(X=x)$ which changes the probability of $Y = y$ for some value of $y$. 

In Pearl (\cite{Pea2000}), and Spirtes, Glymour and Scheines (\cite{SpiGlySch1993}), one finds a representation of causal and counterfactual reasoning in a mixed framework which
combines a deterministic component (functional equations) with a probabilistic
one (probability distributions over exogenous variables). Such structures, which
have been widely studied in the literature, induce causal Bayesian networks, as
shortly mentioned in the introduction.

In the context of team semantics, probabilities have been recently introduced via the notion of multiteam. A multiteam differs from a team in that it may feature multiple copies of the same assignment; it is therefore closer to a collection of experimental data than teams are. There have been at least two different approaches to the formalization of multiteams in the literature (\cite{Vaa2017}, \cite{DurHanKonMeiVir2016}). However, the approach we take in the following is that of simulating multiteams by means of teams; this can be easily accomplished by assuming that each team has an extra variable Key (which is never mentioned in our formal languages) which takes distinct values on distinct assignments of the same causal team. In this way, we can have two assignments that agree on all the significant variables and just differ on Key. With this assumption in mind, the definition of \emph{causal multiteam} can follow word by word the definition of causal team.




We will now work towards the definition of languages that can be appropriate for the discussion and study of probabilistic notions as interpreted over causal multiteams. If we wish to talk about probabilities, it seems natural to add appropriate atomic formulas to our formal language. In theory, it would make sense to assign probabilities to any \emph{flat} formula of the resulting language. However, for technical simplicity, we only allow our formal language to talk of probabilities of $\CO$ formulas (which were proved to be flat in earlier sections).

\begin{df}
The set of probabilistic literals is given by:
\[
\neg \alpha | Pr(\chi) \leq \epsilon | Pr(\chi) \geq \epsilon | Pr(\chi) \leq Pr(\theta) |  Pr(\chi) \geq Pr(\theta)
\]
where $\alpha$ is a probabilistic literal, $\chi,\theta$ are formulas of $\CO$ and $\epsilon \in \mathbb{R}\cap [0,1]$. Literals and probabilistic literals without negation will be called atomic formulas.

The \emph{(basic) probabilistic causal language} ($\PCD$) is given by the following clauses:
\[
\alpha | \psi \land \chi | \psi \lor \chi | \psi \sqcup \chi | \theta\supset \psi | \SET X = \SET x\cf \psi
\]
where $\alpha$ is a literal or probabilistic literal, $\psi,\chi$ are $\PCD$ formulas, and $\theta$ a boolean combination of (non probabilistic) atoms.
\end{df}

As in the non-probabilistic case, we also define restricted languages:
\begin{itemize}
\item $\PC$ if selective implication is not allowed

\item $\PO$ if counterfactuals are not allowed

\item $\PP$ if neither selective implications nor counterfactuals are allowed.
\end{itemize}

Some words are needed to motivate our choice of languages and their restrictions. First of all, notice that, besides the usual (tensor) disjunction $\lor$ that most naturally arises in the context of logics of dependence, we also add to our probabilistic languages what is called in the literature \emph{intuitionistic disjunction}; that is, a binary connective $\sqcup$ which satisfies the following semantical clause (already introduced in an earlier section):
\begin{itemize}
\item $T\models \psi \sqcup \chi \iff T\models \psi$ or $T\models \chi$.
\end{itemize}
If one wants to express that the disjunctive event ``$X=x$ or $Y=y$'' has probability $\epsilon$, it can be done using the formula $Pr(X=x \lor Y=y)=\epsilon$ (the semantics of probabilistic atoms will be defined soon below). If one wishes to make disjunctive statements about probabilities, the correct operator is instead the intuitionistic disjunction. For example, the statement ``either $X=x$ has probability less than one third, or probability greater than two thirds'' will be expressed as $Pr(X=x) <1/3 \sqcup Pr(X=x) > 2/3$. Intuitionistic disjunction turns out to be particularly useful for defining other operators. For example, dependence atoms and conditional independence atoms turn out to be  (model-theoretically) \emph{definable} already in $\PO$ (this will be shown later on in the section). This is the reason why we have omitted such atoms from the syntax. Furthermore, we could not come up with any reasonable semantics for selective implications with probabilistic antecedents, and so we also omitted such constructs from the syntax.

We will use some obvious abbreviations, such as $Pr(\chi) = \epsilon$ for $Pr(\chi) \leq \epsilon \land Pr(\chi) \geq \epsilon$, or $Pr(\chi) < \epsilon$ for $Pr(\chi) \leq \epsilon \land \neg Pr(\chi) \geq \epsilon$.

The semantic clauses are the same as for the language $\CD$, enriched with clauses for the probabilistic atoms. In the following, we will often abuse notation and write $\{s\}$ for a \emph{causal} team whose underlying team contains exactly one assignment. Given any $\CO$ formula\footnote{This definition applies as well to any flat formula.} $\chi$ and any team $T$ with \emph{finite, nonempty} support $T^-$, we can define the \textbf{probability of $\chi$ in $T$} as:
\[
Pr_T(\chi):= \frac{card(\{s\in T^- | \{s\}\models \chi\})}{card(T^-)}.
\]
For $T^-=\emptyset$, we conventionally assume $Pr_T(\chi)$ to be undefined; and consequently, expressions like $Pr_T(\chi)\leq \epsilon$, etc. to be false.

Then the clauses for probabilistic atoms follow quite smoothly:
\begin{itemize}
\item $T\models Pr(\chi)\leq \epsilon  \text{ iff } Pr_T(\chi)\leq \epsilon$
\item $T\models Pr(\chi)\geq \epsilon  \text{ iff } Pr_T(\chi)\geq \epsilon$
\item $T\models Pr(\chi)\leq Pr(\theta) \text{ iff } Pr_T(\chi)\leq Pr_T(\theta)$
\item $T\models Pr(\chi)\geq Pr(\theta)  \text{ iff } Pr_T(\chi)\geq Pr_T(\theta)$
\end{itemize}
It is easy to see that the resulting logic is \emph{not} downward closed. For example, a team which is such that less than half of its assignments satisfy $\chi$ will satisfy $Pr(\chi)\leq \frac{1}{2}$. But the subteam $T^\chi$ consisting \emph{only} of the assignments that satisfy $\chi$ will not satisfy $Pr(\chi)\leq \frac{1}{2}$. 
Notice also that the atoms of the form $Pr(\chi)\leq Pr(\theta)$ (and $Pr(\chi)\geq Pr(\theta)$) could be eliminated if we allowed quantification over real values: we could then replace them with $\exists \epsilon(Pr(\chi)\leq \epsilon \land Pr(\theta)\geq \epsilon)$.

A crucial question which arises right away is: does our definition really define a probability distribution on a team $T$ of finite support? 
Let $\mathcal{E}_T$ be the set of all subsets $S$ of $T$ which are definable by a $\CO$ formula $\chi$ (that is, $S= \{s\in T^- | \{s\}\models \chi\}$). 
This will be our set of \emph{events}. For any event $S\in \mathcal{E}_T$, define $P_T(S):= Pr_T(\psi)$, where $\psi$ is any formula which defines $S$.

Now, a nice thing about $\CO$ formulas is that one can explicitly find (within $\CO$ itself) their complementary negations (a formula is a complementary negation of $\psi$ if it is satisfied exactly by those singleton causal teams that do not satisfy $\psi$). We define a canonical complementary negation $\psi^c$ for each formula $\psi$ of $\CO$:
\begin{itemize}
\item $(X = x)^c$ is $X \neq x$ (for any variable $X$ and value $x$)
\item $(X \neq x)^c$ is $X = x$ (for any variable $X$ and value $x$)
\item $(\theta \land \chi)^c$ is $(\theta^c \lor\chi^c)$
\item $(\theta \lor \chi)^c$ is $(\theta^c \land \chi^c)$
\item $(\theta\supset\chi)^c$ is $\theta \land \chi^c$
\item $(\theta\cf\chi)^c$ is $\theta\cf\chi^c$
\end{itemize} 

Complementary negations of $\CO$ formulas may fail to be boolean formulas, but they turn out to be useful because of the following:

\begin{lem}
Let $T$ be a causal multiteam, $\psi\in\CO$, and $S = \{s\in T^- | \{s\}\models \psi\}$. Then $T^- \setminus S = \{s\in T^- | \{s\}\models \psi^c\}$.
\end{lem}

\begin{proof}
By induction on the synctactical complexity of $\psi\in \CO$. We consider the less obvious cases.

Let $\psi = \theta\supset \chi$. Then $\{s\}\models\psi^c$ iff $\{s\}\models \theta \land\chi^c$
iff (by inductive hypothesis) $\{s\}\models \theta$ and $\{s\}\not\models \chi$ iff $\{s\}\not\models \theta\supset \chi$.

Let $\psi = \theta\cf \chi$. Then $\{s\}\models \psi^c$ iff $\{s\}\models \theta\cf \chi^c$ iff $\{s\}_\theta\models \chi^c$ iff (by inductive hypothesis) $\{s\}_\theta\not\models \chi$ iff $\{s\}\not\models \theta\cf\chi$.
\end{proof}




\begin{teo}
Let $T$ be a finite causal multiteam. Then $(T,\mathcal{E}_T,
 P_T)$ is a probability space.
\end{teo}

\begin{proof}
1) $T^-$ is definable by the formula $X=x \lor X\neq x$. Therefore $T^-\in \mathcal{E}_T$.

2) $\mathcal{E}$ is obviously closed under unions (due to finiteness).

3) We want to show that $\mathcal{E}$ is closed under complementation. Let $S\in\mathcal{E}$. Then $S$ is defined by some $\CO$ formula $\psi$. But then $T^-\setminus S$ is defined by the complementary formula $\psi^c$, therefore $T^-\setminus S \in \mathcal{E}$.




4) $Pr_T(T^-) =  1$.

5) Let $S_1,S_2 \in \mathcal{E}_T$, with $S_1 \cap S_2 = \emptyset$. $S_1$ and $S_2$ are defined by two $\CO$ formulas $\psi_1,\psi_2$, respectively. Then $S_1\cup S_2$ is defined by $\psi_1 \lor \psi_2$. So 
\[
P_T(S_1\cup S_2) = Pr_T(\psi_1 \lor \psi_2) = \frac{card(\{s\in T^- | \{s\}\models \psi_1\lor \psi_2\})}{card(T^-)} =
\]
\[
= \frac{card(\{s\in T^- | \{s\}\models \psi_1\}) + card(\{s\in T^- | \{s\}\models \psi_2\})}{card(T^-)} =
\]
\[
 =\frac{card(\{s\in T^- | \{s\}\models \psi_1\})}{card(T^-)} + \frac{card(\{s\in T^- | \{s\}\models \psi_2\})}{card(T^-)} =
\]
\[
 = Pr_T(\psi_1) + Pr_T(\psi_2) = P_T(S_1) + P_T(S_2);
\]
 the third equality holds due to the assumption $S_1 \cap S_2 = \emptyset$.
\end{proof}

This result allows us to use the usual rules of probability calculus.

\subsection{Comparison with causal models} \label{SUBSCOMP}

At this point we are in the position to notice a difference between our approach and the usual causal models. Recall that in the semi-deterministic approach of \cite{Pea2000} and \cite{SpiGlySch1993}, by a \emph{causal model} one usually understands (the terminology is by no means used consistently) a structural equation model enriched with a probabilistic distribution over the exogenous variables $\SET{U}$. If the causal model is recursive (i.e. the graph is acyclic), then each endogenous variable can be represented as a function of the exogenous variables (just consider the structural equation defining an endogenous variable, $V:=f_V(PA_V)$; iteratively replace each of the endogenous variables occurring in the right term of this expression with its structural equation until you obtain a function $f'_V(\SET{U})$). Given recursivity, the probability distribution over exogenous variables induces a joint probability distribution over \emph{all} the variables of the system given by the following equation:
\begin{equation}\label{EQEXTENDPROB}
P(X_1 = x_1\land\dots\land X_n = x_n) :=
\end{equation}
\[
 \sum_{\{(u_1,\dots,u_k)|\text{for all }i=1..n \text{ : } f'_{X_i}(u_1,\dots,u_k)=x_i\}} P(U_1=u_1\land\dots\land U_k=u_k)
\]
From this joint probability, one can obtain by marginalization the distributions of each set of variables. 

This approach fails, however, for non-recursive causal models, 
in case their system of structural equations has multiple solutions or no solutions at all for some configuration(s) of values of the exogenous variables. In this case some endogenous variables may fail to be functions of the exogenous variables (consider e.g. a system with variables $U,X,Y$ all having range $\{0,1\}$, and consisting of the equations $X:=Y$ and $Y:= X$, which induce the arrows $X\rightarrow Y$ and $Y\rightarrow X$; here $U$ is the only exogenous variable; given a fixed value $u$ for $U$, both triples $(u,0,0)$ and $(u,1,1)$ are solutions to the system; therefore, neither $X$ nor $Y$ can be determined as functions of $U$). 

Unlike the semi-deterministic approach of causal models, our definition of probabilities for causal teams applies to both recursive and non-recursive systems. This is due to the fact that we do not conflate in our framework two different accounts of exogeneity, as it is often the case in the literature. In the first account, exogenous variables are treated as quantities whose causal mechanisms are not further examined in the model. In the second account, exogenous variables are treated as \emph{unobserved} variables. The treatment of exogenous variables in our causal teams supports the first account: they have no structural equations, but the assignments of the causal teams encode quantitative information about them. To the extent to which causal teams stick to the first account, they are more general than causal models. (Of course, if we further extend causal teams to allow for unobserved variables, probabilities in non-recursive systems might fail again to be determined by the exogenous distribution; such a generalisation is beyond the scope of the present paper). 

We might ask at this point whether our probabilities extend correctly (``conservatively'') the probabilities typical of recursive causal models. That is, if we take the team-defined probability over the exogenous variables $\SET{U}$, and we define the joint probability by formula \eqref{EQEXTENDPROB}, do we obtain the team-defined joint probability? The answer is affirmative:

\begin{teo}
Let $T$ be a recursive causal multiteam with exogenous variables $\SET{U} = U_1,\dots, U_k$. For each $(u_1,\dots,u_k)\in \prod_{i=1..k}{Ran(U_i)}$, let $P(U_1=u_1\land\dots\land U_k=u_k):= Pr_T(U_1=u_1\land\dots\land U_k=u_k)$. Let $X_1,\dots,X_n$ be a list of all variables of $T$, and $P(X_1,\dots,X_n)$ be defined from $P(U_1=u_1\land\dots\land U_k=u_k)$ as in equation \eqref{EQEXTENDPROB}. Then $P(X_1=x_1,\dots,X_n=x_n) = Pr_T(X_1=x_1\land\dots\land X_n=x_n)$ for any tuple $(x_1,\dots,x_n) \in \prod_{X_i\in dom(T)}Ran(X_i)$.
\end{teo}

\begin{proof}
By definition, 
\[
Pr_T(X_1=x_1 \land\dots\land X_n=x_n) = \frac{card(\{s\in T| s(X_1) = x_1 \land\dots\land s(X_n) = x_n\})}{card(T)}.
\]
 Clause c) in the definition of causal teams entails that this last term is equal to 
\[
\frac{card(\{s\in T | f'_{X_i}(s(U_1),\dots,s(U_k)) = x_i \text{ for } i=1..n\})}{card(T)}
\]
where the $f'_{X_i}(u_1,\dots,u_k)$ are defined as above. We can further decompose the numerator according to the values taken by $U_1,\dots,U_k$, which ultimately leads to
\[
\sum_{(u_1,\dots,u_k)|\text{ for all } i, f'_{X_i}(u_1,\dots,u_k)=x_i}\frac{card(\{s\in T | s(U_1)=u_1,\dots,s(U_k)) = u_k \})}{card(T)}
\] 
that is
\[
\sum_{(u_1,\dots,u_k)|\text{ for all } i, f'_{X_i}(u_1,\dots,u_k)=x_i}P(U_1=u_1 \land\dots\land U_k = u_k)
\] 
which is the definition of $P(X_1=x_1,\dots,X_n=x_n)$.
\end{proof}

\subsection{Definability in the probabilistic language} \label{SECDEFPROB}

Can we define conditional probabilities in our framework? Semantically, the obvious definition of a conditional probability over a team is (
assuming that $\chi_1,\chi_2$ are $\CO$ formulas, and $\chi_1$ is satisfied by at least one assignment of the team):
\[
Pr_T(\chi_2|\chi_1):= \frac{Pr_T(\chi_1\land\chi_2)}{Pr_T(\chi_1)}.
\]
It is more complicated to assess whether our current formal languages can talk about conditional probabilities. We can see that, in a number of circumstances, the selective implication can take the role of conditioning. For example, under the same assumptions as above, we can write $Pr(\chi_2|\chi_1)\leq \epsilon$ as an abbreviation for 
\[
\chi_1 \supset Pr(\chi_2)\leq \epsilon
\]
Here is the proof that the abbreviation has the intended meaning:
\[
T\models \chi_1 \supset Pr(\chi_2)\leq \epsilon \iff
\]
\[
\iff T^{\chi_1}\models Pr(\chi_2)\leq \epsilon
\]
\[
\iff \frac{card(\{s\in (T^{\chi_1})^- | \{s\}\models \chi_2\})}{card((T^{\chi_1})^-)}\leq \epsilon
\]
\[
\iff \frac{card(\{s\in (T^{\chi_1})^- | \{s\}\models \chi_2\})}{card(T^-)}\frac{card(T^-)}{card((T^{\chi_1})^-)} \leq \epsilon
\]
\[
\iff \frac{Pr_T(\chi_1 \land \chi_2)}{Pr_T(\chi_1)} \leq \epsilon \iff Pr_T(\chi_2|\chi_1) \leq \epsilon
\]
provided $Pr_T(\chi_1)> 0$. In case $Pr_T(\chi_1)= 0$, then the causal team is empty, so that our earlier conventions impose that 
 $T^{\chi_1}\not\models Pr(\chi_2|\chi_1)\leq \epsilon$. Therefore, in this case we treat conventionally the expression $Pr(\chi_2|\chi_1)\leq \epsilon$ as not being satisfied by $T$.
 Things work analogously for inequalities in the opposite direction, and one can similarly prove that $T\models \chi \supset Pr(\psi_1)\leq Pr(\psi_2)$ iff $Pr_T(\psi_1|\chi)\leq Pr_T(\psi_2|\chi)$.

We would like, however, to point out that our logical language is in some respect more general than other languages typically used in the causation literature (see e.g. \cite{Pea2000}, chap. 3). To give an example, Pearl would write $P(y|do(x),u)=\epsilon$, or $P(y|\hat x,u)=\epsilon$, to denote the probability, given the observation that $U$ is equal to $u$, of $Y$ being equal to $\epsilon$, in the model where $X$ has been fixed to $x$ by an intervention (see e.g. \cite{Pea2000} sec. 3.4). This is expressed in our notation by $X=x \cf (U=u \supset Pr(Y=y) = \epsilon)$. A straightforward computation shows that this expression is equivalent to the statement that $Pr_{T_{X=x}}(Y=y |U=u) = \epsilon$, that is, that the probability of $Y=y$ conditional on $U=u$ \emph{in the intervened team} is $\epsilon$. But in our language we also have available formulas like $U=u \supset (X=x \cf P(Y=y) = \epsilon)$.  This formula is not necessarily equivalent to the previous one.  This point can be seen by looking back at an example considered in subsection \ref{SUBSINTER} to show the non-commutativity of selective and counterfactual implications. The team $S$ in that example satisfies $Z=1 \cf (Y=3 \supset Pr(Y=3)=1)$ but does not satisfy $Y=3 \supset ( Z=1\cf Pr(Y=3)=1)$.  The difference between the two formulas is that in the second we consider the observational evidence that \emph{in the initial model}, $U$ has value $u$. We can explicitly show that the second formula expresses the fact that the probability of $Y=y$, after the intervention $do(X=x)$, is smaller than $\epsilon$, conditional on the evidence that $U=u$ \emph{in the pre-intervention team}.

The equivalence can be proved as follows:
\[
T\models U=u \supset (X=x \cf P(Y=y) = \epsilon) \iff
\]
\[
\iff T^{U=u} \models X=x \cf P(Y=y) = \epsilon
\]
\[
\iff (T^{U=u})_{X=x} \models P(Y=y) = \epsilon
\]
\[
\iff \frac{card(\{s\in ((T^{U=u})_{X=x})^-|\{s\}\models Y=y\})}{card(((T^{U=u})_{X=x})^-)} = \epsilon
\]
\[
\iff \frac{card(\{t\in (T^{U=u})^-|\{t\}_{X=x}\models Y=y\})}{card((T^{U=u})^-)} = \epsilon
\]
\[
\iff \frac{card(\{t\in (T^{U=u})^-|\{t\}_{X=x}\models Y=y\})}{card(T^-)}\cdot\frac{card(T^-)}{card((T^{U=u})^-)} = \epsilon
\]
\[
\iff \frac{Pr_T(U=u \land (X=x \cf Y=y))}{Pr_T(U=u)} = \epsilon
\]
\[
\iff Pr_T(X=x \cf Y=y|U=u) = \epsilon
\]
(the fourth equivalence makes an essential use of the fact that we are using multiteams, not teams). In Pearl's notation, there is always the implicit convention that what occurs as a conditioning variable (in this case $u$) has to be conditioned on in the post-intervention  model. This is also the reason why $u$ and $do(x)$ can occur in the same position in the notation, without any problems of non-commutativity arising. But this limitation seems problematic in the context of causal reasoning. We might want, for example, to discuss the effects of an intervention over a restricted population, and Pearl's notation does not seems to be of any help in this case. The same problem seems to afflict other notational representations of counterfactuals, such as the one coming from the potential outcome tradition (\cite{Ney1923}, \cite{Rub1974}, \cite{GalPea1998}). In that notational system we can write, say, $Y_x(u)=y$ to express that, under the observation that $U$ is $u$, fixing $X$ to $x$ forces $Y$ to take value $y$ ($U=u \supset (X=x \cf Y=y)$), but there seems to be no way to express the formula $X=x \cf (U=u \supset Y=y)$. Since exogenous variables are not indirectly affected by interventions, this is just a minor problem; but the limits of the notation immediately arise again if one discusses conditional interventions that involve some exogenous variable.

Let us now see what kind of (probabilistic) dependencies and independencies are definable in our framework. First of all, we can give a model-dependent definition of \emph{two $\CO$ formulas being probabilistically independent}: $\chi_1\pindep\chi_2$. To this purpose, notice that all probabilities expressible over a team $T$ are multiple of the minimum unit $u = 1/card(T^-)$. Thanks to this, probabilistic atoms of the form $Pr(\chi_2|\chi_1)\leq Pr(\theta)$ can be given a model-dependent definition within $\PO$ (under the assumption $Pr_T(\chi_1) > 0$):
\[
T\models Pr(\chi_2|\chi_1)\leq Pr(\theta) \iff 
\]
\[
\iff T \models \bigsqcup_{m=0..card(T^-)} [(\chi_1 \supset Pr(\chi_2)\leq mu) \land Pr(\theta)\geq mu)]
\]

 Assuming by default that, if either $\chi_1$ or $\chi_2$ has probability 0, then the two formulas are probabilistically independent, we can take  $\chi_1\pindep\chi_2$ as an abbreviation for $Pr(\chi_1) = 0 \sqcup Pr(\chi_2) = 0 \sqcup Pr(\chi_2|\chi_1) = Pr(\chi_2)$. More generally, we can talk of \textbf{conditional independence atoms} $\chi_1\pindep_{\chi_3}\chi_2$ as abbreviations for $Pr(\chi_1|\chi_3) = 0 \sqcup Pr(\chi_2|\chi_3) = 0 \sqcup Pr(\chi_2|\chi_1\land\chi_3) = Pr(\chi_2| \chi_3)$. 

It would be important to understand whether independence atoms $\chi_1\perp_{\chi_3}\chi_2$ (\cite{GraVaa2013}) are definable or not in this framework, or whether it would be worth to add them to the language. 
This would lead to a simple uniform definition of dependence atoms; but we do not know the answer to this question. However, we can see that it is possible to give a model-dependent definition (in $\PO$) of dependence atoms in terms of conditional independence atoms. Given a sequence $(x_1,\dots,x_n)$ of values (implicitly associated to variables $X_1,\dots,X_n \in dom(T)$),
 we say that $(x_1,\dots,x_n)$ \emph{occurs} in $T$ if there is an $s\in T$ such that $s(X_1)=x_1,\dots,s(X_n)=x_n$.

\begin{teo}
For any finite causal multiteam $T$, $T\models \dep{\SET X}{Y}$ if and only if 
\[
T\models \bigwedge_{(\SET x,y)\in Ran(\SET X) \times Ran(Y)} Y=y \hspace{5pt}\pindep_{\SET X = \SET x} Y=y. 
\]
\end{teo}



\begin{proof}
The conjunction says that, for every tuple $\SET x,y$ of appropriate values for $\SET X,Y$, either $Pr_T(Y=y|\SET X = \SET x) = 0$, or $Pr_T(Y = y|Y = y \land \SET X = \SET x) = Pr_T(Y=y|\SET X = \SET x)$.

That is, for each $\SET x,y\in Ran(\SET X) \times Ran(Y)$, either this tuple does not occur in $T$ or $Pr_T(Y=y|\SET X = \SET x)=1$, which means, more explicitly:
\[
 \frac{card(\{s\in T^-|s(Y)=y \text{ and } s(\SET X) = \SET x\})}{card(\{s\in T^-| s(\SET X) = \SET x \})} = 1
\]
which is equivalent to
\[
card(\{s\in T^-|s(Y)=y \text{ and } s(\SET X) = \SET x\}) = card(\{s\in T^-| s(\SET X) = \SET x \})
\]

Given the finiteness of the team, this last equality is equivalent to the statement that any assignment $s$ which satisfies $s(\SET X) = \SET x$ also satisfies $s(Y) = y$.
\end{proof}

\subsection{The Markov Condition}

An assumption underlying a vast literature on causal studies is the so-called Markov Condition. As a condition on a causal team $T$, it seems natural to paraphrase it as follows: 

\begin{center}
$Pr_T(X=x | PA_X = pa_X) = Pr_T(X=x | PA_X = pa_X, \SET{N} =\SET{n})$

 for any set $\SET{N}$ of nondescendants of $X$ disjoint from $PA_X$ 

and any tuple $(x,pa_X,\SET{n})$ occurring in $T$.
\end{center}

This is a scheme of conditions which state that each variable is independent of its nondescendants, conditional on its parents. This condition is considered important in causal inference, for example, because it allows to write the joint probability of the system in a compact and usually easy to compute form (product formula):
\[
Pr_T(X_1=x_1\land \dots \land X_n=x_n)=\prod_{i=1..n} Pr_T(X_i=x_i|PA_i=pa_i)
\] 
where $pa_i$ is $x_1,\dots,x_n$ restricted to the values for $PA_i$.
This formula, in turn, gives a modular method for calculating the joint probability in intervened teams:
\[
Pr_{T_{\SET{Z} = \SET{z}}}(X_1=x_1\land \dots \land X_n=x_n)=\prod_{X_i \notin \SET{Z}} Pr_T(X_i=x_i|PA_i=pa_i).
\] 






However, in causal (multi)teams the Markov Condition is not automatically satisfied. As a simple counterexample, consider the team
\begin{center}
\begin{tabular}{|c|c|c|}
\hline
 \multicolumn{2}{|c|}{X \ Y} \\
\hline
 1 & 1 \\
\hline
 2 & 1 \\
\hline
 1 & 2 \\
\hline
\end{tabular}
\end{center}

Since the parent set of $X$ is the empty set, and the tuple $X=1,Y=2$ occurs in the team, the Markov Condition would imply: $P(X=1) = P(X=1|Y=2)$. 
However, here we have $Pr(X=1) = 2/3$, but $P(X=1 | Y=2) =1$.

The Markov Condition is a rather complex condition to require on a team; for example, it might be noticed that its verification can require an exponential time in the number of variables of the domain (since we have to verify such a condition for every appropriate \emph{set} $\SET{Z}$ of variables). It is probably desirable to find a simpler axiom that implies the Markov Condition. We suggest the following \emph{Markov Axiom Scheme} (MA):

\begin{center}
$Pr_T(X=x | PA_X = pa_X) = Pr_T(X=x | PA_X = pa_X, \SET{Y} =\SET{y})$

 where $\SET{Y}$ is the set of \emph{all} nondescendants of $X$ not occurring in $PA_X$ 

and for any tuple $(x,\SET{y})$ occurring in $T$.
\end{center}

Notice that in this case we have one axiom for each variable in the domain; therefore, the complexity of verifying the Markov Axiom Scheme is linear in the number of variables. We have to prove that this condition is sufficient to enforce the Markov condition. This is obtained by a merely probabilistic reasoning, which we are allowed to used thanks to our earlier proof that teams endowed with our probability function form a $\sigma$-algebra. The following lemma is an immediate consequnce of basic probability calculus:

\begin{lemma} \label{SHIFTLEMMA}
For any $\CO$ formulas $\psi,\chi,\theta$ and causal multiteam $T$, and assuming that $Pr_T(\chi \land \theta)>0$, it follows that $Pr_T(\psi \land \chi | \theta) = Pr_T(\psi| \chi\land \theta) \cdot Pr_T(\chi|\theta)$.
\end{lemma}

\begin{teo}
Suppose that a causal multiteam $T$ satisfies the Markov Axiom Scheme. Then it satisfies the Markov Condition.
\end{teo}

\begin{proof}
Let $X_i\in dom(T)$, $PA_i$ its set of parents, $\SET{Z}$ a set of variables such that $PA_i \subseteq \SET{Z} \subseteq dom(T)\setminus \{X_i\}$, and such that any element of $\SET{Z}\setminus PA_i$ is a nondescendant of $X_i$.
We want to show that $Pr_T(X_i=x_i | \bigwedge_{X_j\in \SET{Z}} X_j = x_j) = Pr_T(X_i=x_i | \bigwedge_{X_j\in PA_i} X_j = x_j)$ for any tupla $(x_1,\dots,x_n)$ occurring in $T$.

Let $\SET{W} = \{X_j\in dom(T) | X_j \text{ nondescendant of $X_i$ and } X_j\notin \SET{Z}\}$. In case $\SET{W}$ is empty, this means that $\SET{Z}\setminus PA_i$ is the set of all nondescendants of $X_i$ that are not in $PA_i$; therefore the equality we are trying to prove is simply an instance of the Markov Axiom.

Suppose then that $\SET{W}$ is nonempty. We now write $\SET{z}$ for the subsequence of values of the specific sequence $(x_1,\dots,x_n)$ that correspond to the variables in $\SET{Z}$; and $\SET{wz}$ for the concatenation of sequence $\SET{z}$ with some sequence $\SET{w}$ of values for $\SET{W}$. Define $T_{\SET{Z}}(\SET{W}) = \{\SET{w}|\text{ there is } s\in T^- \text{ such that } s(\SET{W}) = \SET{w} \text{ and } s(\SET{Z}) = \SET{z}\} = \{\SET{w}| Pr_T(\SET{W} = \SET{w} \land \SET{Z} = \SET{z})>0 \}$. Notice then that, if  
$Pr_T(\SET{W} = \SET{w} \land \SET{Z} = \SET{z})=0$, it follows that $Pr_T(X_i = x_i \land \SET{W} = \SET{w} | \bigwedge_{X_j\in\SET{Z}} X_j = x_j) = 0$. Therefore we can write:
\[
Pr_T(X_i = x_i | \bigwedge_{X_j\in\SET{Z}} X_j = x_j) = 
\]
\[
= \sum_{\SET{w}\in T_{\SET{Z}}(\SET{W})}
 Pr_T(X_i = x_i \land \SET{W} = \SET{w} | \bigwedge_{X_j\in\SET{Z}} X_j = x_j) 
\]
because the events $X_i = x_i \land \SET{W} = \SET{w}$, as $\SET{w}$ varies among tuples occurring in $T_{\SET{Z}}(\SET{W})$, are disjoint; and for other values $\SET{w}\notin T_{\SET{Z}}(\SET{W})$, $Pr_T(X_i = x_i \land \SET{W} = \SET{w} | \bigwedge_{X_j\in\SET{Z}} X_j = x_j) = 0$.

Since, furthermore, for any $\SET{w}\in  T_{\SET{Z}}(\SET{W})$ it holds that  $Pr_T(\SET{W} = \SET{w} \land \bigwedge_{X_j\in\SET{Z}} X_j = x_j)>0$, we are entitled to use Lemma \ref{SHIFTLEMMA} on each summand, which yields the expression
\[
\sum_{\SET{w}\in T_{\SET{Z}}(\SET{W})} [ Pr_T(X_i = x_i | \SET{W} = \SET{w}, \bigwedge_{X_j\in\SET{Z}} X_j = x_j)\cdot Pr_T(\SET{W} = \SET{w} | \bigwedge_{X_j\in\SET{Z}} X_j = x_j) ].
\]
Now, applying the appropriate instances of the Markov Axiom, one obtains 
\[
\sum_{\SET{w}\in T_{\SET{Z}}(\SET{W})} [ Pr_T(X_i = x_i | \bigwedge_{X_j\in PA_i} X_j = x_j)\cdot Pr_T(\SET{W} = \SET{w} | \bigwedge_{X_j\in\SET{Z}} X_j = x_j) ].
\]
Since the left factor does not depend anymore on $\SET{w}$, we can extract it:
\[
Pr_T(X_i = x_i | \bigwedge_{X_j\in PA_i} X_j = x_j)\cdot \sum_{\SET{w}\in T_{\SET{Z}}(\SET{W})} Pr_T(\SET{W} = \SET{w} | \bigwedge_{X_j\in\SET{Z}} X_j = x_j). 
\]
Since the events form a $\sigma$-algebra, and $T_{\SET{Z}}(\SET{W})$ contains all the values $\SET{w}$ such that $\SET{W} = \SET{w} \land \bigwedge_{X_j\in\SET{Z}} X_j = x_j$ has positive probability, the right factor sums up to 1, and we obtain the desired equality.
\end{proof}

\subsection{Pearl's example in causal teams}

We conclude this section illustrating the workings of the probabilistic semantics for causal teams over an example taken from Pearl (\cite{Pea2000}, $1.4.4$). 

For Pearl the purpose of the example is to show that causal Bayesian networks are insufficient to answer counterfactual questions, which require structural equations.
Our concern here would only be with that aspect of the example (Pearl's ``model $2$'') which
illustrates the three stages in his treatment of counterfactuals
in the structural equation framework. The following quote gives the
spirit of the enterprise: 
\begin{quote}
The preceding considerations further imply that the three tasks listed
in the beginning of this section - prediction, intervention, and counterfactuals
- form a natural hierarchy of causal reasoning tasks, with increasing
levels of refinement and increasing demands on the knowledge required
for accomplishing the tasks. Prediction is the simplest of the three,
requiring only a specification of a joint distribution function. The
analysis of interventions requires a causal structure in addition
to a joint distribution. Finally, processing counterfactuals is the
hardest task because it requires some information about the functional
relationships and/or the distribution of the omitted factors. (p.
38)
\end{quote}
The example is about a clinical test. There are two exogenous variables
$X$ (being treated: $X=1$; not being treated: $X=0$) and $Y$ (which
measures the subject's response: recovery $Y=0$, death, $Y=1)$. There
are two exogenous (hidden), independent variables $U_{1}$ and $U_{2}$.
The equations of the model are:
\[
\left\{\begin{array}{l}
x=u_{1} \\
y=xu_{2}+(1-x)(1-u_{2}) 
\end{array}\right.
\]

They induce the following $DAG$: 

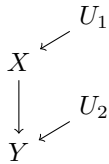
\begin{figure}[htbp]
	\centering
		\begin{tikzpicture}
		
		\node(X) {$X$};
		
		\node(Y) at (0,-1.2) {$Y$};
		 
		\node(U1) at (1,0.6) {$U_1$};
		
		\node(U2)  at (1, -0.6) {$U_2$}; 
		
		\draw[->] (X)--(Y) ; 
		
		\draw[->] (U1)--(X); 
		
		\draw[->] (U2)--(Y); 
		
		\end{tikzpicture}
\end{figure}


The two independent binary variables are governed by the probability
distribution $P(u_{1}=1)=P(u_{2}=1)=\frac{1}{2}$, which generates joint
probability distributions $P(x,y,u_{2})$ and $P(x,y)$ that may be
calculated using the formula from subsection \ref{SUBSCOMP}. For instance,  $P(X=x\land Y=y)=$
\[
= \sum_{\left\{ \left(u_{1},u_{2}\right)\mid u_{1}=x\wedge u_{1}u_{2}+(1-u_{1})(1-u_{2})=y\right\} }P(U_{1}=u_{1}\wedge U_{2}=u_{2})  = 0,25
\]
 for all $x$ and $y$. Here is a table which lists $P(x,y,u_{2})$
and $P(x,y):$
\[
\small
\begin{array}{c|ccc|ccc|ccc}
 &  & u_{2}=1 &  &  & u_{2}=0 &  &  & Marginal\\
 & x=1 &  & x=0 & x=1 &  & x=0 & x=1 &  & x=0\\
\hline y=1 & 0 &  & 0.25 & 0.25 &  & 0 & 0.25 &  & 0.25\\
y=0 & 0.25 &  & 0 & 0 &  & 0.25 & 0.25 &  & 0.25
\end{array}
\]

Pearl considers the following counterfactual query: 
\begin{description}
\item [{({*})}] What is the probability $Q$ that a subject who died under
treatment $(x=1,y=1)$ would have recovered $(y=0)$ had he or she
not been treated $(x=0)?$
\end{description}
He calls the fact that the subject died under treatment the evidence
\[
e=\left\{ (x=1,y=1)\right\} 
\]
and points out that answering the above query involves three
steps: ($1$) given the evidence $e$, ($3$) we compute the probability
$Y=y$ ($2$) under the hypothetical condition $X=x$. 

In his example this amounts to:
\begin{enumerate}
\item We apply the evidence $e$ to the set of equations, and conclude that
it is compatible with only one realization: $u_{1}=1$ and $u_{2}=1$.
This gives us a value for $u_{2}$. 
\item We make an intervention $do(X=0)$ on the variable $X$. It results
in the first equation being replaced with $x=0$. 
\item We solve the second equation for $x=0,$ $u_{2}=1$. We get $y=0$
from which we conclude that the probability $Q$ is 1. 
\end{enumerate}
We are told that in the general case
\begin{itemize}
\item step ($1$) corresponds to an abduction: update the probability $P(\SET u)$
to obtain $P(\SET u|\SET e)$
\item step ($2$) corresponds to an action: replace the equations corresponding
to each intervened variable $X$ by $X=x$ 
\item step ($3$) corresponds to a prediction: compute $P(y)$ in the new model. 
\end{itemize}
(Pearl 2009, p.61)

Pearl is explicit on the point that if for each value $\SET u$ of the
exogenous variables $\SET U$ there is a unique solution for $Y$, then
step ($3$) always gives a unique solution for the needed probability:
we simply sum up the probabilities $P(\SET u|e)$ assigned to all $\SET u$
which yield $Y=y$ as a solution. More technically: 
\[
\begin{array}{c}
P(y)=\sum_{\left\{ \SET u:f'_{Y}(\SET u)=y\right\} }P(\SET u|e)
\end{array}
\]


Pearl's model $2$ can be represented in our framework by using the causal team $T$ 

\begin{center}
\begin{tabular}{|c|c|c|c|}
\hline
 \multicolumn{4}{|c|}{ } \\
 \multicolumn{4}{|c|}{$U_1$\tikzmark{UP1}  \tikzmark{UP2}$U_2$\tikzmark{UP2'} \   \tikzmark{XP1}$X$\tikzmark{XP1'} \   \tikzmark{YP1}$Y$} \\
\hline
 0 & 0 & 0 & 1\\
\hline
0 & 1 & 0 & 0\\
\hline
1 & 0 & 1 & 0\\
\hline
1 & 1 & 1 & 1\\
\hline
\end{tabular}
 \begin{tikzpicture}[overlay, remember picture, yshift=.25\baselineskip, shorten >=.5pt, shorten <=.5pt]
  \draw [->] ([yshift=3pt]{pic cs:XP1'})  [line width=0.2mm] to ([yshift=3pt]{pic cs:YP1});
  \draw ([yshift=7pt]{pic cs:UP1})  edge[line width=0.2mm, out=35,in=125,->] ([yshift=4pt]{pic cs:XP1});
	\draw ([yshift=5pt]{pic cs:UP2'})  edge[line width=0.2mm, out=35,in=125,->] ([yshift=7pt]{pic cs:YP1});
  \end{tikzpicture}
\end{center}


 where the range of each variable is $\left\{ 0,1\right\}$
and the arrows of the DAG 
are induced by the two structural equations given earlier, which are part of the description of the team. Notice also that this team, interpreted as a multiteam, bears the same probabilities that are carried by the system described by Pearl.




It can be checked that $T$ is an explicit causal team. 
We formulate Pearl's query in our language as 
\begin{equation}\label{PEARL1}
 X=1\wedge Y=1\supset(X=0\cf Y=0)
\end{equation}
or, in its probabilistic version, as

\begin{equation}\label{PEARL2}
 X=1\wedge Y=1\supset(X=0\cf Pr(Y=0)=1).
\end{equation}

By the clause for selective implication, the formula \eqref{PEARL1} is true in the causal team $T$ if and only if the
formula $(X=0\cf Y=0)$ is true in the causal team represented by the annotated table
\begin{center}
\begin{tabular}{|c|c|c|c|}
\hline
 \multicolumn{4}{|c|}{ } \\
 \multicolumn{4}{|c|}{$U_1$\tikzmark{UQ1}  \tikzmark{UQ2}$U_2$\tikzmark{UQ2'} \  \tikzmark{XQ1}$X$\tikzmark{XQ1'} \  \tikzmark{YQ1}$Y$} \\
\hline
1 & 1 & 1 & 1\\
\hline
\end{tabular}
 \begin{tikzpicture}[overlay, remember picture, yshift=.25\baselineskip, shorten >=.5pt, shorten <=.5pt]
  \draw [->] ([yshift=3pt]{pic cs:XQ1'})  [line width=0.2mm] to ([yshift=3pt]{pic cs:YQ1});
  \draw ([yshift=7pt]{pic cs:UQ1})  edge[line width=0.2mm, out=35,in=125,->] ([yshift=4pt]{pic cs:XQ1});
	\draw ([yshift=5pt]{pic cs:UQ2'})  edge[line width=0.2mm, out=35,in=125,->] ([yshift=7pt]{pic cs:YQ1});
  \end{tikzpicture}
\end{center}

 and the two structural equations. Now the formula $(X=0\cf Y=0)$ is true, given that the intervention
$do(X=0)$ generates the causal team represented by the annotated team

\begin{center}
\begin{tabular}{|c|c|c|c|}
\hline
 \multicolumn{4}{|l|}{ } \\
 \multicolumn{4}{|l|}{$U_1$\tikzmark{UR1}  \tikzmark{UR2}$U_2$\tikzmark{UR2'} \   \tikzmark{XR1}$X$\tikzmark{XR1'} \   \tikzmark{YR1}$Y$} \\
\hline
1 & 1 & 0 & 0\\
\hline
\end{tabular}
 \begin{tikzpicture}[overlay, remember picture, yshift=.25\baselineskip, shorten >=.5pt, shorten <=.5pt]
  \draw [->] ([yshift=3pt]{pic cs:XR1'})  [line width=0.2mm] to ([yshift=3pt]{pic cs:YR1});
	\draw ([yshift=5pt]{pic cs:UR2'})  edge[line width=0.2mm, out=35,in=125,->] ([yshift=7pt]{pic cs:YR1});
  \end{tikzpicture}
\end{center}
 and the same system of equations, except for the first one, replaced by $X=0$\footnote{Actually, in the causal team formalism this is represented by the removal of the invariant function $\mathcal{F}_{T^{X=1 \land Y=1}}(X)$, and the change of role of $X$, from endogenous to exogenous variable.}. Obviously $Y=0$ is true in this causal team, which shows that the formula \eqref{PEARL2} is true. 

Later on in section 7, where we will introduce various notions of cause within causal team semantics, we will see that in this example it is also true that $X$ is a \textit{direct
cause} of $Y$ in $T$.

Now recall our observations from subsection \ref{SECDEFPROB} concerning the definability of conditional probabilistic formulas. Those results show that
\[
X=1\wedge Y=1\supset(X=0\cf Pr(Y=0)=1)
\]
 is true in a causal team $T$ if and only if 
\[
Pr_{T}(X=0\cf Y=0 \mid X=1\wedge Y=1)=1.
\]
 We notice that the formula inside the brackets shows nicely Pearl's
three ``mental steps''.

\section{Causal modeling} \label{CAUSMOD}

\subsection{Direct and total cause}

What might be interesting to express in our languages? We consider the two basic type-causal notions from Woodward (\cite{Woo2003}), \emph{direct} and \emph{total} cause.

Let us consider the case of \textbf{direct cause}. Recall Woodward's definition of a direct cause in section \ref{SUBSCAUSES}:
\begin{quote}
A necessary and sufficient condition for $X$ to be a direct cause of $Y$ with respect to some variable set $V$ is that there be a possible intervention on $X$ that will change $Y$ (or the probability distribution of $Z$) when all other variables in $V$ besides $X$ and $Z$ are held fixed at some value by interventions.
\end{quote}

Woodward's definition is slightly ambiguous in that it talks about a change in $Y$, but  does not say with respect to what the change is made; to $Y$'s actual value? To some possible value of $Y$, i.e., some $y\in Ran(Y)$? To a $y$ occurring in some allowed configuration (assignment)? We resolve the ambiguity by stipulating that the values of $Y$ to compared be generated by \emph{two} distinct interventions. 

 The kind of intervention needed for the evaluation of direct causation of $X$ on $Y$ is a simultaneous intervention on all variables in the system except for $X$ and $Y$. As an example, let $T$ be the team

\begin{center}
\begin{tabular}{|c|c|c|}
\hline
 \multicolumn{3}{|c|}{ }\\ 
 \multicolumn{3}{|l|}{X\tikzmark{X14} \  \tikzmark{Y14}Z\tikzmark{Y14'} \  \tikzmark{Z14}Y} \\
\hline
 1 & 1 & 2 \\
\hline
 2 & 2 & 4 \\
\hline
 3 & 3 & 6 \\
\hline
\end{tabular}
 \begin{tikzpicture}[overlay, remember picture, yshift=.25\baselineskip, shorten >=.5pt, shorten <=.5pt]
  \draw [->] ([yshift=3pt]{pic cs:X14})  [line width=0.2mm] to ([yshift=3pt]{pic cs:Y14});
	\draw [->] ([yshift=3pt]{pic cs:Y14'})  [line width=0.2mm] to ([yshift=3pt]{pic cs:Z14});
  \draw ([yshift=7pt]{pic cs:X14})  edge[line width=0.2mm, out=55,in=125,->] ([yshift=7pt]{pic cs:Z14});
  \end{tikzpicture}
\end{center}
with $\mathcal{F}_Z(x):= x$ and $\mathcal{F}_Y(x,z) = x+z$. It is an essential ingredient of this example that there is an arrow from $X$ to $Y$. We show that $X$ is a direct cause of $Y$ in $T$. First of all we must fix all other variables (in this case, just $Z$) to an appropriate value (we choose $1$) by an intervention, which also updates $Y$:
\begin{center}
\begin{tabular}{|c|c|c|}
\hline
 \multicolumn{3}{|c|}{ }\\ 
 \multicolumn{3}{|l|}{X\tikzmark{X17} \  \tikzmark{Y17}Z\tikzmark{Y17'} \ \ \tikzmark{Z17}Y} \\
\hline
 1 & \textbf{1} & \textbf{2} \\
\hline
 2 & \textbf{1} & \textbf{3} \\
\hline
 3 & \textbf{1} & \textbf{4} \\
\hline
\end{tabular}
 \begin{tikzpicture}[overlay, remember picture, yshift=.25\baselineskip, shorten >=.5pt, shorten <=.5pt]
	\draw [->] ([yshift=3pt]{pic cs:Y17'})  [line width=0.2mm] to ([yshift=3pt]{pic cs:Z17});
  \draw ([yshift=7pt]{pic cs:X17})  edge[line width=0.2mm, out=55,in=125,->] ([yshift=7pt]{pic  cs:Z17});
  \end{tikzpicture}
\end{center}
Then we intervene in two different ways on $X$:
\begin{center}
\begin{tabular}{|c|c|c|}
\hline
 \multicolumn{3}{|c|}{ }\\ 
 \multicolumn{3}{|l|}{X\tikzmark{X15} \  \tikzmark{Y15}Z\tikzmark{Y15'} \ \ \tikzmark{Z15}Y} \\
\hline
 \textbf{1} & 1 & \textbf{2} \\
\hline
\end{tabular}
 \begin{tikzpicture}[overlay, remember picture, yshift=.25\baselineskip, shorten >=.5pt, shorten <=.5pt]
	\draw [->] ([yshift=3pt]{pic cs:Y15'})  [line width=0.2mm] to ([yshift=3pt]{pic cs:Z15});
  \draw ([yshift=7pt]{pic cs:X15})  edge[line width=0.2mm, out=55,in=125,->] ([yshift=7pt]{pic cs:Z15});
  \end{tikzpicture}
\begin{tabular}{|c|c|c|}
\hline
 \multicolumn{3}{|c|}{ }\\ 
 \multicolumn{3}{|l|}{X\tikzmark{X16} \  \tikzmark{Y16}Z\tikzmark{Y16'} \ \ \tikzmark{Z16}Y} \\
\hline
  \textbf{2} & 1 & \textbf{3} \\
\hline
\end{tabular}
 \begin{tikzpicture}[overlay, remember picture, yshift=.25\baselineskip, shorten >=.5pt, shorten <=.5pt]
	\draw [->] ([yshift=3pt]{pic cs:Y16'})  [line width=0.2mm] to ([yshift=3pt]{pic cs:Z16});
  \draw ([yshift=7pt]{pic cs:X16})  edge[line width=0.2mm, out=55,in=125,->] ([yshift=7pt]{pic cs:Z16});
  \end{tikzpicture}	
\end{center}
The fact that the two interventions select distinct values for $Y$ proves that $X$ is a direct cause of $Y$.

The peculiar form of these kinds of interventions makes so that, \emph{if there is an arrow from $X$ to $Y$}, after the intervention all the columns in the team are constant, that is, a singleton team is produced. This suggests that direct cause can also be defined as an operation on teams, instead of as an operation on assignments. Let $Fix(\SET z)$ be an abbreviation for $\bigwedge_{Z\in Dom(T)\setminus\{X,Y\}}Z=z$: 
\[
T\models DC(X;Y) \text{ iff there are } \SET z, x\neq x', y\neq y' \text{ such that }
\]
\[
 T\models (Fix(\SET z) \land X=x)\cf Y=y 
\]
\[
\text{ and } T\models (Fix(\SET z) \land X=x')\cf Y=y'. 
\]
This semantical definition suggests that direct causation might be definable in an appropriate  quantifier extension of $\CD$. We do not venture into this direction, but we look instead for a model-dependent definition in the language $\CD$ \emph{enriched with intuitionistic disjunction} $\sqcup$\footnote{The results of \cite{YanVaa2016} entail that, in the case of propositional dependence logic, the intuitionistic disjunction is eliminable in terms of the other propositional connectives (although not uniformly definable, \cite{Yan2017}). It is unclear, at the current state of investigation, whether these results translate into our framework.} \footnote{A remark: adding to $\C$ or $\CO$ the intuitionistic disjunction, one obtains logics that are still downward closed, but not flat.}. With the help of this operator, we can express the fact that $X$ is a direct cause of $Y$ in the causal team $T$, $T\models DC(X;Y)$, by the formula
\[
\bigsqcup_{x\neq x', y\neq y',\SET z}[(Fix(\SET z) \land X=x) \cf Y=y] \land [(Fix(\SET z) \land X=x') \cf Y=y']
\]
where the disjunction ranges over all tuples $(x,x',y,y',\SET z)\in Ran(X)^2\times Ran(Y)^2\times Ran(\SET Z)$ such that $x\neq x'$ and $y\neq y'$. This definition is model-dependent in that it makes reference to the ranges of varianbles, and to a set $\SET Z = dom(T)\setminus \{X,Y\}$ of variables that is determined by the domain of the causal team.



If there is an arrow from $X$ to $Y$, any intervention of this kind will collapse the team to a singleton team, so that the condition for direct cause can be checked (it can fail; see the example below). But if instead some intervention of this form does \emph{not} collapse the team into a singleton team, then necessarily there is no arrow from $X$ to $Y$. Therefore, the graph of direct causes is a subgraph of the graph of the team. It can be a \emph{proper} subgraph, that is, there might be an arrow between two variables $X$ and $Y$, and yet those fail to be in the direct cause relationship; this is what happens in the following causal team (with, say, $Ran(X)=Ran(Y)=\{0,1\}$):
\begin{center}
\begin{tabular}{|c|c|}
\hline
 \multicolumn{2}{|c|}{$X$ \hspace{-3pt}$\rightarrow$ \hspace{-3pt}$Y$} \\
\hline
 0 & 0 \\
\hline
1 & 0 \\
\hline
 \end{tabular}
\end{center}

The definition of probabilistic direct cause in $\PCD$ presents a little surprise.
As in the deterministic case, we use counterfactual antecedents of the form $X=x \land Fix(\SET z)$ to try to reduce to one the number of configurations occurring in the multiteam; however, now, even in presence of an arrow from $X$ to $Y$, the resulting team is not necessarily a singleton team, but it will be in general a collection of identical copies of one and the same assignment; the formula $Y=y$ will assume probability 0 or 1. This leads to simplifications in the definition of probabilistic direct cause. Write $\Psi_{x,x';y;\SET z}$ for
\[
 ((Fix(\SET z) \land X=x) \cf Pr(Y=y) = 0 \land ((Fix(\SET z) \land X=x') \cf Pr(Y=y)=1.
\]
Then, probabilistic direct cause can be formulated as:
\[
T\models PDC(X;Y) \text{ iff there are } \SET z, x, x', y \text{ such that }
 T\models \Psi_{x,x';y;\SET z}
\]
A non-uniform definition in  $\PCD$ is as follows:
\[
\bigsqcup_{x\neq x',y,\SET z}\Psi_{x,x';y;\SET z}
\]
We may conclude that the probabilistic notion of direct cause collapses to the deterministic one.

The case of \textbf{total cause} (\cite{Woo2003}) is more complicated. Recall Woodward's definition from section 2.3:
\begin{quote}
$X$ is a total cause of $Y$ if and only if there is a possible intervention on $X$ that will change $Y$ or the probability distribution of $Y$.
\end{quote}
This formulation of Woodward is again very ambiguous in some respects; as before, we can think this ``change'' in terms of two interventions; but, are the ``changes'' mentioned relative to some fixed actual configuration of the system (i.e. assignment), or should one check all possible configurations? We opt here for the second interpretation. This gives us a reasonable condition for obtaining again the collapse of the team to a singleton. The idea to obtain this is to fix (to some value) all the variables that are not descendants of $X$. Let us call this set $ND_X$; we use letter $\SET w$ to indicate values for these variables. Write $Fix'(\SET w)$ for $\bigwedge_{W\in ND_X} W=w$, and $\Xi_{x, x';y,y';\SET w}$ for 
\[
Fix'(\SET w) \cf [(X=x\cf Y=y) \land (X=x'\cf Y=y')].
\]

We can then define total cause as:

\[
T\models TC(X;Y) \text{ iff there are } x\neq x',\SET w\in T(ND_X), y\neq y'  
\]
\[
\text{ such that } 
 T\models \Xi_{x, x';y,y';\SET w}.
\]


This, once more, can be witten a a single formula:
\[
T\models \bigsqcup_{x\neq x',\SET w\in T(ND_X), y\neq y'} \Xi_{x, x';y,y';\SET w}.
\]

Things work differently in the probabilistic framework. Notice that in the language $\PCD$ we have (abbreviated) formulas of the form
\[
X=x \cf Pr(Y=y) = \epsilon
\]
saying that an intervention fixing the value of $X$ at $x$ will affect in a certain way the  probability distribution of $Y$. If a conjunction of the following form 
\[
\Theta_{x,x';y;\epsilon,\epsilon'}: (X=x \cf Pr(Y=y) = \epsilon) \land (X=x' \cf Pr(Y=y) = \epsilon')
\]
holds for two distinct values $\epsilon, \epsilon'$, we can conclude that $X$ is a total cause in the (probabilistic) sense given by Woodward. More precisely, we can treat probabilistic total causation as follows. $X$ is a \textbf{(probabilistic) total cause} of $Y$, $T\models PTC(X;Y)$, iff
\[
 \text{there are }  x\neq x', y,\epsilon\neq\epsilon' \text{ such that } T\models \Theta_{x,x';y;\epsilon,\epsilon'}.
\]
Compare this with the definition of deterministic total cause; the clauses have been considerably simplified. This follows from the fact that, differently from deterministic atoms of the form $Y=y$, the probabilistic atoms $Pr(Y=y)=\epsilon$ are not flat; they instead express a global property of the team. Therefore, there is no need to use the trick that was applied in the deterministic case (fixing counterfactually the nondescendants of $X$, so that  the team collapses to a single assignment). In this probabilistic context, total causation is genuinely a property of teams, not of single assignments.

 Observing that the probability of any formula in the language is an integer multiple of the smallest unit $u = \frac{1}{card(T^-)}$, and that probabilities range within $[0,1]\cap \mathbb{R}$, we can also explictly define (in a non-uniform way) $T\models PTC(X;Y)$ as:
\[
\bigsqcup_{x\neq x',y, m\neq n,m,n=0..card(T^-)}\Theta_{x,x';y;mu,nu}.
\]



\subsection{Higher-order effects}

In general, a variable $X$ might be thought to have some kind of ``causal effect'' on another variable $Y$ if intervening in some way on $X$ can change the state of variable $Y$. In the kind of language that is usually applied to describe a causal model, the ``state'' of a variable $Y$ is completely summarized by a formula of the form $Y=y$, asserting that $Y$ is taking a specific value $y$. But, in a team, it may happen that a variable $Y$ satisfies \emph{none} of such formulas. Yet there are properties that can be ascribed or not to a variable $Y$ in $T$, and that might cease or begin to hold after an intervention on $X$. One example are statements on the probability distribution of $Y$. Or, there might be some deterministic property, say $\psi(Y)$, which does not hold in $T$, but holds after an intervention that reduces the range of $Y$-values occurring in the team.
This suggests that ``$X$ is a total cause of $Y$'' might be redefined, in contexts more general than the usual causal modeling, as the statement that ``there is an intervention on $X$ that changes some property of $Y$''; and similarly for other notions of causation. Therefore, the arguments above rather show that the notion of direct cause is special, in that its evaluation over a team can be reduced to evaluation over causal models.


Our formalism allows many more counterfactuals than those that are typically used in the definition of causal notions; it allows us to ponder, for example, what effects an intervention  has on a \emph{set} of variables. Let us consider an example which is natural in the team context, but is unlikely to emerge in the usual discussions on interventionist causality. Consider a causal team $T$ of the form
\begin{center}
\begin{tabular}{|c|c|c|}
\hline
 \multicolumn{3}{|c|}{X\tikzmark{X12} \  \tikzmark{Y12}Y\tikzmark{Y'12} \  \tikzmark{Z12}Z} \\
\hline
 1 & 2 & 1 \\
\hline
 2 & 3 & 1 \\
\hline
 \end{tabular}
 \begin{tikzpicture}[overlay, remember picture, yshift=.25\baselineskip, shorten >=.5pt, shorten <=.5pt]
  \draw [->] ([yshift=3pt]{pic cs:X12})  [line width=0.2mm] to ([yshift=3pt]{pic cs:Y12});
  \end{tikzpicture}
\end{center}

Clearly $T\not\models\dep{Z}{Y}$. However, the intervened team $T_{X=1}$
\begin{center}
\begin{tabular}{|c|c|c|}
\hline
 \multicolumn{3}{|c|}{X\tikzmark{X13} \  \tikzmark{Y13}Y\tikzmark{Y'13} \  \tikzmark{Z13}Z} \\
\hline
 1 & 2 & 1 \\
\hline
 \end{tabular}
 \begin{tikzpicture}[overlay, remember picture, yshift=.25\baselineskip, shorten >=.5pt, shorten <=.5pt]
  \draw [->] ([yshift=3pt]{pic cs:X13})  [line width=0.2mm] to ([yshift=3pt]{pic cs:Y13});
  \end{tikzpicture}
\end{center}
is such that $T_{X=1}\models\dep{Z}{Y}$. Said otherwise, $T\models X=1 \cf \dep{Z}{Y}$. So, intervening on $X$ has changed a property of variable $Y$ (that of being functionally dependent, or not, on $Z$). Clearly, such a team property is completely different from those that are usually considered in discussions about causation; and they may be thought as properties of the \emph{set} of variables $\{Y,Z\}$. 
Said otherwise, intervening on $X$ may affect the set $\{Y,Z\}$.

We can construct similar examples also in the language of probabilities, without need for counterfactuals nor selective implications. Consider the following multiteam with variables $X,Y$ and no arrows, and the multiteam obtained from it by applying $do(X=0)$:

\begin{center}
$S$:
\begin{tabular}{|c|c|}
\hline
 \multicolumn{2}{|c|}{$X$ \ \ $Y$} \\
\hline
 0 & 0 \\
\hline
0 & 1 \\
\hline
1 & 2 \\
\hline
 \end{tabular}
\hspace{20pt}
$S_{X=0}$:
 \begin{tabular}{|c|c|}
\hline
 \multicolumn{2}{|c|}{$X$ \ \ $Y$} \\
\hline
 0 & 0 \\
\hline
0 & 1 \\
\hline
0 & 2 \\
\hline
 \end{tabular}
\end{center}

Notice that the intervened team contains a new assignment $\{(X,0),(Y,2)\}$. Therefore, the probability of formula $X=0 \land Y=2$ raised from $0$ (before the intervention) to $1/3$ (after the intervention).
In this case, intervening on $X$ affected the set $\{X,Y\}$, in the sense that it changed a property of the relationship between $X$ and $Y$.

\subsection{Invariance}

We want to make some remarks about one methodological aspect of Pearl's approach to causation. In his causal models, arrows represent the statement, or guess, that some functional dependencies are invariant under \emph{most} of the interventions that can be carried out on the system (a structural equation $X:= f_X(PA_X)$ is invariant under all interventions that do not act on $X$). Graphs therefore draw a line between some of the dependencies occurring in the data, which might be considered either 1) contingent or 2) due to a third common cause, and those that 3) have to be considered causal dependencies. But it must be immediately observed that there may be dependencies that are invariant in the sense specified above, and yet are not represented by arrows in the causal graph. 

 Similarly, in our causal teams, the graph $G(T)$ encodes a set of invariant functional dependencies of the form $ \dep{PA_Y}{Y}$,
but the causal team could still sustain invariant dependencies which are not explicitated by $G(T)$ (in this case, we say that a functional dependency $\dep{X_1,\dots,X_n}{Y}$ is \emph{invariant in $T$} if it holds after any intervention on $T$\footnote{One might raise the protest that it should be required that $\dep{X_1,\dots,X_n}{Y}$ be invariant under any \emph{sequence} of interventions. However, the importation law we have proved in earlier sections shows that for any sequence of interventions, there is a single intervention generating the same modified causal team. Notice, futhermore, that in this context the actual team $T$ can be thought of as the result of a null intervention.}); call $G_{inv}(T)$ the graph which represents all the invariant dependencies. Notice that while the structural equation for $Y$ is not invariant under interventions that act on $Y$, the functional dependency $\dep{X_1,\dots,X_n}{Y}$ \emph{is} invariant under such interventions, although its arrows are removed from the graph. What happens is this (we consider for simplicity an intervention over only one variable $Y$): if $T\models\dep{X_1,\dots,X_n}{Y}$, then $(X_i,Y)\in G(T)$ but $(X_i,Y)\notin G(T_{Y=y})$. But since $Y$ is constant in $T_{Y=y}$, no further intervention can disrupt the dependency $\dep{X_1,\dots,X_n}{Y}$; therefore $(X_i,Y)\in G_{inv}(T_{Y=y})$. This shows that there may be invariant dependencies in a causal team that are not represented by arrows of the graph.

 In general, we can isolate three graphs of interest: $G(T)\subseteq G_{inv}(T) \subseteq G_{cont}(T)$, where $G_{cont}(T)$ represents \emph{all} the functional dependencies in $T^-$, including those that are not invariant. (Furthermore, as shown in a previous subsection, it might be of interest to consider the subgraphs of $G(T)$ that represent direct causes, total causes and probabilistic total causes).
The point is that the choice of $G(T)$ determines what an intervention is/does; in turn, the set of interventions 
 thus defined may determine some further invariant dependencies (set $G_{inv}(T)$), those that hold both in the initial team and in all the intervened ones. The arrows of this new graph might happen to be inappropriate to capture a notion of ``direct'' or also of ``indirect'' causation. 

It might be interesting to ask whether this kind of invariance is definable in our languages; the question makes sense not only for functional dependencies, but for any property definable by a formula. The answer is positive; we can give a model-dependent definition. A formula $\psi$ is invariant for $T$ if:


\[
T\models INV(\psi) := \bigwedge \{ \SET{Z}=\SET{z}\cf \psi\}
\]
where the conjunction ranges over all consistent conjunctions of the form $\SET{Z}=\SET{z}$.  
 Notice that this is not a single formula, but a scheme of formulas of $\CD$; which formula is actually picked depends on the domain of the causal team $T$ and the ranges of its variables. We do not know whether a uniform definition can be given, or if instead adding invariance statements (for sentences, or just for dependencies) might increase the expressivity of our languages.

\section{Conclusions}

We have shown how to generalize team semantics in a way that it may incorporate the kind of counterfactual and causal reasoning which is typical of recursive, possibly nonparametric structural equation models. We briefly explained how the semantics can be extended to cover also nonrecursive, parametric models with unique or no solutions. We introduced languages for causal discourse (both deterministic and probabilistic) over causal teams; we showed that these languages are in some ways more general than notations typically used in the literature (in that, for example, they may distinguish pre- and post-intervention observations).

We explored somewhat systematically the logic underlying deterministic and probabilistic languages for interventionist counterfactual reasoning when these are applied to recursive, parametric causal teams; we also shortly sketched what challenges arise in the treatment of the nonparametric case, and proposed tentative clauses for nonclassical semantical notions that naturally arise in this case. Finally, we analyzed how some of the notions of cause and invariance that arise from manipulationist theories of causation can be expressed and generalized in our framework.

These initial explorations of causal team semantics leave space for further investigation. On the side of causal modeling, parametric nonrecursive and nonparametric causal teams certainly deserve a more extensive study than what was sketched here (we find it unlikely, instead, that the study of the nonrecursive nonparametric teams might be fruitful). Another important case in causal modeling, which we only briefly grazed in this paper, are models with unobserved variables; defining a causal team semantics for this case would surely require some deep rethinking of the framework. Similar challenges might be posed by the study of forms of interventions beyond the ``surgical'' ones considered here.

On the logical side, it would be important to develop tools for the comparison of expressive power of distinct languages, and to investigate systematically whether the introduction of counterfactuals leads to significant new definability results. The inferential aspects of causal languages should also be further investigated; in particular, complete deduction systems for deterministic languages with (in)dependence atoms are missing; and the case of probabilistic languages is still unexplored. Further directions for investigation are comparisons with previous notion of counterfactuals, and the study of the modal structure induced by interventions on teams.

The present work has been developed within the Academy of Finland Project n.286991, ``Dependence and Independence in Logic: Foundations and Philosophical Significance''.

\bibliographystyle{elsart-num-sort}
\bibliography{iilogics}

\end{document}